\documentclass[10pt]{article}
\font\smallit=cmti10
\font\smalltt=cmtt10

\usepackage{amssymb,latexsym,amsmath,epsfig,amsthm}

\makeatletter

\renewcommand\section{\@startsection {section}{1}{\z@}
{-30pt \@plus -1ex \@minus -.2ex}
{2.3ex \@plus.2ex}
{\normalfont\normalsize\bfseries\boldmath}}

\renewcommand\subsection{\@startsection{subsection}{2}{\z@}
{-3.25ex\@plus -1ex \@minus -.2ex}
{1.5ex \@plus .2ex}
{\normalfont\normalsize\bfseries\boldmath}}

\renewcommand{\@seccntformat}[1]{\csname the#1\endcsname. }

\makeatother

\theoremstyle{plain}
\newtheorem{theorem}{Theorem}[section]
\newtheorem{lemma}[theorem]{Lemma}
\newtheorem{proposition}[theorem]{Proposition}
\newtheorem{corollary}[theorem]{Corollary}

\theoremstyle{definition}
\newtheorem{definition}[theorem]{Definition}

\theoremstyle{definition}  
\newtheorem{remark}[theorem]{Remark}
\newtheorem{example}[theorem]{Example}

\input{hex-cgt.sty}
\input{hexboard.sty}

\noshadows

\begin{document}

\begin{center}
\uppercase{\bf On the combinatorial value of Hex positions}
\vskip 20pt
{\bf Peter Selinger}\\
{\smallit Department of Mathematics and Statistics, Dalhousie
  University, Nova Scotia, Canada}\\
\end{center}
\vskip 20pt

\centerline{\smallit } 
\vskip 30pt 

\centerline{\bf Abstract}

\noindent
  We develop a theory of combinatorial games that is appropriate for
  describing positions in Hex and other monotone set coloring games.
  We consider two natural conditions on such games: a game is {\em
    monotone} if all moves available to both players are good, and
  {\em passable} if in each position, at least one player has at least
  one good move available. The latter condition is equivalent to
  saying that if passing were permitted, no player would benefit from
  passing.  Clearly every monotone game is passable, and we prove that
  the converse holds up to equivalence of games. We give some examples
  of how this theory can be applied to the analysis of Hex positions.

\pagestyle{myheadings}
\markright{\smalltt INTEGERS: 22 (2022)\hfill}
\thispagestyle{empty}
\baselineskip=12.875pt
\vskip 30pt 

\section{Introduction}

Hex is a strategic perfect information game for two players, invented
in 1942 by Piet Hein and later independently discovered by John Nash
{\cite{Hein,Nash}}. Hex is attractive because its rules are extremely
simple, yet the game possesses a surprising amount of strategic
depth. Hex is played on a parallelogram-shaped board, typically of
size $n\times n$, that is made of hexagonal fields called {\em hexes}
or {\em cells}, as shown in Figure~\ref{fig-board}(a). One pair of
opposing edges is colored black, and the other pair is colored
white. The players, called Black and White, alternate placing a stone
of their color on any empty hex, with Black starting.  Once placed,
these stones are never moved or removed. Each player's objective is to
connect the two edges of that player's color with an unbroken chain of
that player's stones. It is easy to see that if the board is
completely filled with stones, exactly one player will have such a
connection; to see this, consider the connected component of one of
Black's edges. This either includes the other black edge, in which
case Black has a connection, or it does not, in which case the white
stones along its boundary form a connection for White. Consequently,
there are no draws in the game of Hex; exactly one player will win.
Also, the game always ends in a finite number of moves.
Conventionally, the game ends as soon as one player has completed a
connection. However, since the winner cannot change after that point,
it would be equivalent, and sometimes simpler from a theoretical point
of view, to require the game to continue until all hexes have been
filled. Figure~\ref{fig-board}(b) shows an example of a completed game
that has been won by Black.

\begin{figure}
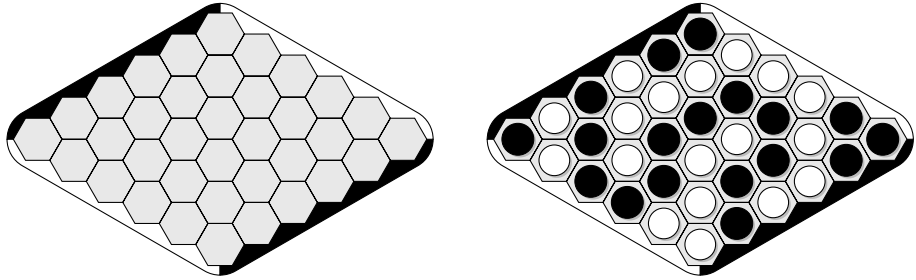

  \[
  \begin{hexboard}[scale=0.8]
    \board(6,6)
  \end{hexboard}
  \qquad
  \begin{hexboard}[scale=0.8]
    \shadows
    \board(6,6)
    \black(1,1)\white(2,1)\black(3,1)\white(4,1)\black(5,1)\black(6,1)
    \white(1,2)\black(2,2)\white(3,2)\white(4,2)\white(5,2)\white(6,2)
    \black(1,3)\white(2,3)\black(3,3)\black(4,3)\black(5,3)\white(6,3)
    \black(1,4)\black(2,4)\white(3,4)\white(4,4)\black(5,4)\white(6,4)
    \white(1,5)\white(2,5)\black(3,5)\black(4,5)\white(5,5)\black(6,5)
    \white(1,6)\black(2,6)\white(3,6)\white(4,6)\black(5,6)\black(6,6)
  \end{hexboard}
  \]
  \caption{(a) A Hex board of size $6\times 6$. (b) A winning position
  for Black.}
  \label{fig-board}
\end{figure}

Hex has many interesting properties. For example, there is an easy
non-con\-struc\-tive proof that Hex on a board of size $n\times n$ is a
win for the first player {\cite{Nash}}; however, no actual winning
strategy is known for $n\geq 11$. The proof is by a strategy stealing
argument: if the second player had a winning strategy, the first
player could mentally place an opponent's stone on the board and then
play the same strategy.  Determining whether a winning strategy exists
in a given position on a board of arbitrary size is known to be a
PSPACE-complete problem {\cite{pspace}}. Of interest to us in this
paper is another easy property: Hex is a monotone game. Informally, by
this we mean that additional black stones on the board can only help
Black, and additional white stones can only help White. A consequence
is that making a move is always at least as good as making no move;
although the ability to pass is not typically part of the rules of
Hex, passing would never be to a player's advantage if it were
permitted.

In practice, the rules of Hex, as stated above, give too much of an
advantage to the first player. Although no explicit winning strategy
for Black is known for large enough board sizes, Black still ends up
winning a majority of games. For this reason, the additional {\em swap
  rule} is employed. It states that after Black's first move, White
may choose to switch colors. This rule incentivizes Black to play a
first move that is as fair as possible, and therefore leads to a more
balanced game. The swap rule is not relevant for the rest of this
paper and we will not consider it here.

Combinatorial game theory is a formalism for the study of sequential
perfect information games that was introduced by Conway {\cite{ONAG}}
and Berlekamp, Conway, and Guy {\cite{WinningWays}}. Its roots go
back further to the study of impartial games such as Nim.
Combinatorial game theory was initially developed for {\em normal
  play} games, in which the first player who is unable to make a move
loses the game.  It has also been adapted to a variety of other
conventions, such as {\em mis\`ere play}, in which the first player
who is unable to move wins, or {\em scoring games}, in which the final
outcome is a numerical score.

In this paper, we develop a variant of combinatorial game theory that
is appropriate for Hex and other monotone set coloring games. The main
difference between the games we describe here and other kinds of
combinatorial games is the winning condition. We already mentioned
that in normal play games, the loser is the first player who cannot
move. In some games like Gomoku (also known as ``Five in a Row''), the
winner is the first player who achieves a winning condition, such as
building a straight line of 5 of the player's stones. By contrast,
although Hex also has a winning condition, it is immaterial {\em when}
the winning condition is achieved. As we will see, this feature, along
with monotonicity, gives rise to a particular family of combinatorial
games with attractive mathematical properties.

An earlier version of this paper appeared on the arXiv at
{\cite{S2021-hex-cgt-arxiv}}. Several problems that were left open in the
earlier version have since been solved, including in
{\cite{DSW2021-gadgets}} and {\cite{DS2022-hex-l5}}. The present paper
has been updated accordingly.

\subsection{Related Work}

Set coloring games were considered by van Rijswijck
{\cite{VanRijswijck}}, and earlier, under the name ``division games'',
by Yamasaki {\cite{Yamasaki}}. These works only considered games with
two atomic outcomes (i.e., Black wins or White wins), rather than
local games over a partially ordered set of outcomes as we do
here. Monotone games, i.e., those where making a move is always at
least as good as passing, are sometimes called ``regular'' games in
the literature.

Some authors, such as van Rijswijck {\cite{VanRijswijck}} and
Henderson and Hayward {\cite{star-decomposition}}, have applied ideas
from combinatorial game theory to Hex and other set coloring
games. For example, they considered notions of dominated and
reversible moves appropriate to these games. However, while these
works were in the ``spirit'' of combinatorial games, they did not
develop the ``letter'' of an actual combinatorial game theory for
Hex. For example, Henderson and Hayward explicitly state that ``Hex is not
a combinatorial game in the strictest sense'', and consider only the
{\em outcome class} of games (such as positive, negative, or fuzzy),
rather than their full combinatorial value. The present paper remedies
this situation by providing such a theory.

Play in local Hex regions, as opposed to on the entire board, has also
been considered in the literature, though not at the same level of
generality as we do here. For example, what we call an $n$-terminal
region is called the carrier of a $2n$-sided decomposition by
Henderson and Hayward {\cite{star-decomposition}}.

The notion of combinatorial games we develop in this paper is closely
related to the concept of passing, i.e., allowing a player not to make
a move. In general, the notion of passing is problematic in
combinatorial game theory, because it creates the possibility of games
with infinite plays. The theory of {\em loopy games}
{\cite[Ch.~11]{WinningWays}} was developed to deal with such
infinite games. By contrast, the {\em passable games} we develop in
this paper are not loopy: in fact, passing is not allowed in these
games, all plays are finite, and therefore ordinary well-founded
Conway induction can always be used on them. However, passable games
have the property that they are {\em equivalent} to games in which
passing is permitted. The defining property of such games is that no
player has an incentive to pass, and therefore it would not alter the
nature of the game if passing were permitted.

Another related concept from combinatorial game theory is {\em
  temperature}, a notion that quantifies, roughly speaking, how
motivated the players are to make a move
{\cite[Ch.~6]{WinningWays}}. A game, or a component of a game, is
{\em hot} if the players can gain an advantage by making a move, and
{\em cold} if the players would prefer not to move. (Thus, the hotter
a game is, the more urgency the players feel to move in it. In cold
games, players only make a move when they have no other choice.) In
these terms, Hex and all monotone set coloring games are hot: making
the next move is never disadvantageous. We do not explicitly use the
concept of temperature in this paper, except to note that there is an
upper limit on how hot a Hex position can be
(Proposition~\ref{prop-n-terminal-limit}).

For book-length treatments of modern Hex strategy, see {\cite{Browne}}
and especially {\cite{Seymour}}. Seymour has also created an excellent
collection of Hex puzzles {\cite{Seymour-puzzles}}. For a detailed
scholarly study of the history of Hex, see
{\cite{Hayward-full-story}}.

\subsection{Contents}

The rest of this paper is organized as follows. In
Section~\ref{sec-outcomes}, we consider play in local Hex regions and
define the poset of outcomes for such a region. In
Section~\ref{sec-set-coloring-games}, we describe the class of set
coloring games to which our combinatorial game theory applies. In
Section~\ref{sec-combinatorial-games}, we describe a class of
combinatorial games appropriate for Hex and other monotone set
coloring games. These games are defined over a given poset of atomic
outcomes. We motivate and define the order relations $\leq$ and
$\tri$, whose definition is subtle and is the main technical vehicle
making these games work. We show that these game admit canonical
forms, and we define the class of monotone games. In
Section~\ref{sec-left-right}, we define relations of left and right
order and equivalence, which are useful for technical reasons. In
Section~\ref{sec-fundamental}, we define passable games and prove the
fundamental theorem of monotone games, which states that monotone
games and passable games are the same up to equivalence. In
Section~\ref{sec-linear}, we prove certain special properties of games
over linearly ordered sets of atoms; in particular, on this class of
games, canonical forms of monotone games are monotone. In
Section~\ref{sec-operations}, we show that certain common operations
on combinatorial games, and especially the sum operation, can be
generalized to passable games, and we consider copy-cat strategies in
this context. In Section~\ref{sec-contextual}, we introduce the
concept of global decisiveness and show that otherwise non-equivalent
games sometimes become equivalent in its presence. In
Section~\ref{sec-enumeration}, we enumerate all passable game values
over certain small atom sets, and show that in all other cases, the
set of passable game values is infinite. In
Section~\ref{sec-realizable}, we show that many, but not all, abstract
passable game values are realizable as Hex positions. In
Section~\ref{sec-k-by-n}, we give an application of this theory by
computing the size of the minimal virtual connection on Hex boards of
size $4\times n$, for all $n$. Finally, in
Section~\ref{sec-conclusion}, we list some open problems and point to
avenues for future work.

\section{Local Play in Hex}
\label{sec-outcomes}

At the end of a game of Hex, there are only two possible outcomes,
which we denote by $\top$ (Black wins) and $\bot$ (White
wins). However, when we are concerned with play in some local {\em
  region} of the board, the set of possible outcomes can be richer. We
first illustrate this idea with some examples.

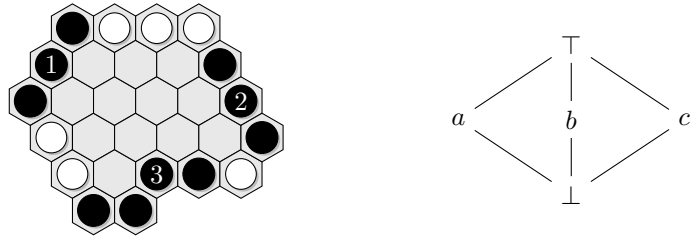
\begin{figure}
  \[
  \m{$
    \begin{hexboard}[scale=0.8]
      \shadows
      \rotation{-30}
      \foreach\i in {3,...,6} {\hex(1,\i)}
      \foreach\i in {2,...,6} {\hex(2,\i)}
      \foreach\i in {1,...,5} {\hex(3,\i)}
      \foreach\i in {1,...,5} {\hex(4,\i)}
      \foreach\i in {1,...,5} {\hex(5,\i)}
      \foreach\i in {1,...,4} {\hex(6,\i)}
      \black(3,1)
      \black(2,2)\label{1}
      \black(1,3)
      \white(1,4)
      \white(1,5)
      \black(1,6)
      \black(2,6)
      \black(3,5)\label{3}
      \black(4,5)
      \white(5,5)
      \black(6,4)
      \black(6,3)\label{2}
      \black(6,2)
      \white(6,1)
      \white(5,1)
      \white(4,1)    
    \end{hexboard}
  $}
  \hspace{2cm}
  \m{
    \begin{tikzpicture}[scale=0.5]
      \node(top) at (0,2) {$\top$};
      \node(a) at (-3,0) {$a$};
      \node(b) at (0,0) {$b$};
      \node(c) at (3,0) {$c$};
      \node(bot) at (0,-2) {$\bot$};
      \draw (bot) -- (a) -- (top);
      \draw (bot) -- (b) -- (top);
      \draw (bot) -- (c) -- (top);
    \end{tikzpicture}
    }
  \]
  \caption{(a) A 3-terminal Hex region. (b) The outcome poset for this
    region.}
  \label{fig-region}
\end{figure}
Consider the board region shown in Figure~\ref{fig-region}(a). We
imagine that this region is part of a larger game. The region is
entirely surrounded by black and white stones. Among these boundary
stones, there are three connected groups of black stones, which we
have labelled 1, 2, and 3.  (Dually, there are also three connected
groups of white stones.) We refer to this type of region as a {\em
  3-terminal region}. Once the enclosed region has been completely
filled with stones, it is clear that the only thing within the region
that can affect the final outcome of the surrounding game is which of
the three black terminals are connected to each other (or
equivalently, which of the three white terminals are connected to each
other). We refer to this as the {\em outcome} of the region. There are
five possible outcomes for a 3-terminal region. We denote them by
$\s{\bot,a,b,c,\top}$, and they are defined as follows:
\begin{itemize}
\item $\bot$: None of Black's terminals are connected to each other.
\item $a$: Terminals $2$ and $3$ are connected to each other, but
  terminal $1$ is not.
\item $b$: Terminals $1$ and $3$ are connected to each other, but
  terminal $2$ is not.
\item $c$: Terminals $1$ and $2$ are connected to each other, but
  terminal $3$ is not.
\item $\top$: All of Black's terminals are connected to each other.
\end{itemize}
Here is an example of each outcome:
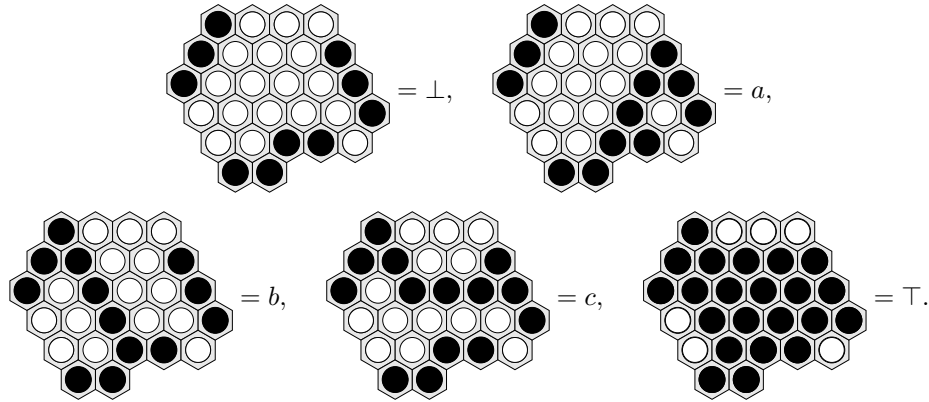
\begingroup
\def\scale{0.65}
\[
\m{
  \scalebox{\scale}{$\begin{hexboard}
      \rotation{-30}
      \noshadows
      \foreach\i in {3,...,6} {\hex(1,\i)\white(1,\i)}
      \foreach\i in {2,...,6} {\hex(2,\i)\white(2,\i)}
      \foreach\i in {1,...,5} {\hex(3,\i)\white(3,\i)}
      \foreach\i in {1,...,5} {\hex(4,\i)\white(4,\i)}
      \foreach\i in {1,...,5} {\hex(5,\i)\white(5,\i)}
      \foreach\i in {1,...,4} {\hex(6,\i)\white(6,\i)}
      \black(3,1)
      \black(2,2)
      \black(1,3)
      \white(1,4)
      \white(1,5)
      \black(1,6)
      \black(2,6)
      \black(3,5)
      \black(4,5)
      \white(5,5)
      \black(6,4)
      \black(6,3)
      \black(6,2)
      \white(6,1)
      \white(5,1)
      \white(4,1)  
    \end{hexboard}$}} = \bot,
\quad
\m{
  \scalebox{\scale}{$\begin{hexboard}
      \rotation{-30}
      \noshadows
      \foreach\i in {3,...,6} {\hex(1,\i)\white(1,\i)}
      \foreach\i in {2,...,6} {\hex(2,\i)\white(2,\i)}
      \foreach\i in {1,...,5} {\hex(3,\i)\white(3,\i)}
      \foreach\i in {1,...,5} {\hex(4,\i)\white(4,\i)}
      \foreach\i in {1,...,5} {\hex(5,\i)\white(5,\i)}
      \foreach\i in {1,...,4} {\hex(6,\i)\white(6,\i)}
      \black(3,1)
      \black(2,2)
      \black(1,3)
      \white(1,4)
      \white(1,5)
      \black(1,6)
      \black(2,6)
      \black(3,5)
      \black(4,5)
      \white(5,5)
      \black(6,4)
      \black(6,3)
      \black(6,2)
      \white(6,1)
      \white(5,1)
      \white(4,1)
      \black(5,3)
      \black(4,4)
    \end{hexboard}$}} = a,
\]\[
\m{
  \scalebox{\scale}{$\begin{hexboard}
      \rotation{-30}
      \noshadows
      \foreach\i in {3,...,6} {\hex(1,\i)\white(1,\i)}
      \foreach\i in {2,...,6} {\hex(2,\i)\white(2,\i)}
      \foreach\i in {1,...,5} {\hex(3,\i)\white(3,\i)}
      \foreach\i in {1,...,5} {\hex(4,\i)\white(4,\i)}
      \foreach\i in {1,...,5} {\hex(5,\i)\white(5,\i)}
      \foreach\i in {1,...,4} {\hex(6,\i)\white(6,\i)}
      \black(3,1)
      \black(2,2)
      \black(1,3)
      \white(1,4)
      \white(1,5)
      \black(1,6)
      \black(2,6)
      \black(3,5)
      \black(4,5)
      \white(5,5)
      \black(6,4)
      \black(6,3)
      \black(6,2)
      \white(6,1)
      \white(5,1)
      \white(4,1)
      \black(3,2)
      \black(3,3)
      \black(3,4)
    \end{hexboard}$}} = b,
\quad
\m{
  \scalebox{\scale}{$\begin{hexboard}
      \rotation{-30}
      \noshadows
      \foreach\i in {3,...,6} {\hex(1,\i)\white(1,\i)}
      \foreach\i in {2,...,6} {\hex(2,\i)\white(2,\i)}
      \foreach\i in {1,...,5} {\hex(3,\i)\white(3,\i)}
      \foreach\i in {1,...,5} {\hex(4,\i)\white(4,\i)}
      \foreach\i in {1,...,5} {\hex(5,\i)\white(5,\i)}
      \foreach\i in {1,...,4} {\hex(6,\i)\white(6,\i)}
      \black(3,1)
      \black(2,2)
      \black(1,3)
      \white(1,4)
      \white(1,5)
      \black(1,6)
      \black(2,6)
      \black(3,5)
      \black(4,5)
      \white(5,5)
      \black(6,4)
      \black(6,3)
      \black(6,2)
      \white(6,1)
      \white(5,1)
      \white(4,1)
      \black(3,2)
      \black(3,3)
      \black(4,3)
      \black(5,3)
    \end{hexboard}$}} = c,
\quad
\m{
  \scalebox{\scale}{$\begin{hexboard}
      \rotation{-30}
      \noshadows
      \foreach\i in {3,...,6} {\hex(1,\i)\black(1,\i)}
      \foreach\i in {2,...,6} {\hex(2,\i)\black(2,\i)}
      \foreach\i in {1,...,5} {\hex(3,\i)\black(3,\i)}
      \foreach\i in {1,...,5} {\hex(4,\i)\black(4,\i)}
      \foreach\i in {1,...,5} {\hex(5,\i)\black(5,\i)}
      \foreach\i in {1,...,4} {\hex(6,\i)\black(6,\i)}
      \black(3,1)
      \black(2,2)
      \black(1,3)
      \white(1,4)
      \white(1,5)
      \black(1,6)
      \black(2,6)
      \black(3,5)
      \black(4,5)
      \white(5,5)
      \black(6,4)
      \black(6,3)
      \black(6,2)
      \white(6,1)
      \white(5,1)
      \white(4,1)  
    \end{hexboard}$}} = \top.
\]
\endgroup
There is a natural partial order on the set of outcomes of any
region. Namely, if $x$ and $y$ are local outcomes, we say that $x\leq
y$ if for every way of embedding the region in a larger board, and for
every way of completely filling the rest of the board with stones, if
the position with outcome $x$ in the region is winning for Black, then
so is the corresponding position with outcome $y$ in the region. In
case of a 3-terminal region, we have $\bot< a,b,c < \top$, where $a$,
$b$, and $c$ are incomparable.  A Hasse diagram for this partially
ordered set is shown in Figure~\ref{fig-region}(b). We emphasize that
all 3-terminal regions have outcomes in the poset
$\s{\bot,a,b,c,\top}$, no matter the shape or size of the region.

We note that in a 2-terminal region, there are only two possible
outcomes: either Black or White connects their terminals. The entire
Hex board forms a 2-terminal region, with the colored edges as the
terminals. In a 4-terminal region, there are 14 possible outcomes, and
in a 5-terminal region, there are 42 possible outcomes. More
generally, the number of outcomes for an $n$-terminal region is equal
to the number of non-crossing partitions of $\s{1,\ldots,n}$, which is
equal to the $n$th Catalan number $C(n)=\frac{(2n)!}{n!(n+1)!}$
{\cite{A000108}}.

As pointed out by Henderson and Hayward {\cite{star-decomposition}},
it is not actually necessary for consecutive opposite-colored
terminals in an $n$-terminal region to touch. It is sufficient for
such terminals to be connected by the two-colored bridge pattern shown
in Figure~\ref{fig-two-colored-bridges}(a). This pattern serves to
separate the cells marked $a$ and $b$ so that one is inside and the
other outside the region. The reason for this is that we can, without
loss of generality, consider $a$ and $b$ to be non-adjacent. Indeed,
if they are both occupied by the same color, they are indirectly
connected anyway, either via the black stone on the left or the white
one on the right. For example, Figure~\ref{fig-two-colored-bridges}(b)
shows a 3-terminal region that is bounded by stones and two-colored
bridges.
\begin{figure}
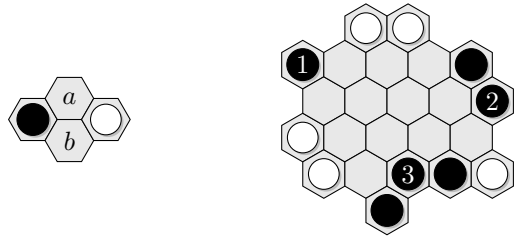

\[
\m{$
  \begin{hexboard}[scale=0.8]
    \shadows
    \hex(1,1)\black(1,1)
    \hex(1,2)\label{$b$}
    \hex(2,1)\label{$a$}
    \hex(2,2)\white(2,2)
  \end{hexboard}
  $}
\hspace{2cm}
  \m{$
    \begin{hexboard}[scale=0.8]
      \shadows
      \rotation{-30}
      \foreach\i in {4,...,5} {\hex(1,\i)}
      \foreach\i in {2,...,6} {\hex(2,\i)}
      \foreach\i in {2,...,5} {\hex(3,\i)}
      \foreach\i in {1,...,5} {\hex(4,\i)}
      \foreach\i in {1,...,5} {\hex(5,\i)}
      \foreach\i in {2,...,3} {\hex(6,\i)}
      \black(2,2)\label{1}
      \white(1,4)
      \white(1,5)
      \black(2,6)
      \black(3,5)\label{3}
      \black(4,5)
      \white(5,5)
      \black(6,3)\label{2}
      \black(6,2)
      \white(5,1)
      \white(4,1)    
    \end{hexboard}
  $}
\]
\caption{(a) A two-colored bridge. (b) A 3-terminal region surrounded
  by stones and two-colored bridges.}
\label{fig-two-colored-bridges}
\end{figure}
\begin{figure}
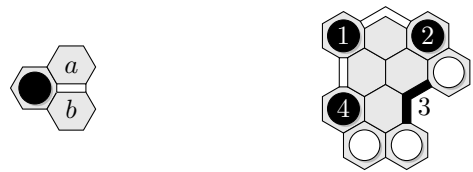

\[
\m{$
  \begin{hexboard}[scale=0.8]
    \shadows
    \hex(1,1)\black(1,1)
    \hex(1,2)\label{$b$}
    \hex(2,1)\label{$a$}
    \draw[color=black,fill=white]
    \coord(1.4,1.2) -- 
    \coord(1.2,1.4) -- 
    \coord(1.6,1.8) -- 
    \coord(1.8,1.6) --
    cycle;
  \end{hexboard}
  $}
\hspace{3cm}
\m{$
  \begin{hexboard}[scale=0.8]
    \shadows
    \rotation{-30}
    \foreach\i in {3,...,4} {\hex(1,\i)}
    \foreach\i in {1,...,4} {\hex(2,\i)}
    \foreach\i in {1,...,2} {\hex(3,\i)}
    \foreach\i in {1,...,2} {\hex(4,\i)}
    \black(1,3)\label{4}
    \white(1,4)
    \black(2,1)\label{1}
    \white(2,4)
    \black(4,1)\label{2}
    \white(4,2)
    \cell(2.95,2.95)\label{3}
    \draw[color=black,fill=white]
    \coord(1.6,1.6) --
    \coord(1.2,2.4) --
    \coord(1.4,2.4) --
    \coord(1.8,1.6) --
    cycle;
    \draw[color=black,fill=white]
    \coord(2.6,0.6) -- 
    \coord(2.6,0.8) -- 
    \coord(3.2667,0.4667) -- 
    \coord(3.6,0.8) --
    \coord(3.8,0.6) -- 
    \coord(3.4,0.2) --
    cycle;
    \draw[color=black,fill=black]
    \coord(2.2,3.4) -- 
    \coord(2.4,3.4) -- 
    \coord(2.7333,2.7333) -- 
    \coord(3.4,2.4) --
    \coord(3.4,2.2) -- 
    \coord(2.6,2.6) --
    cycle;
  \end{hexboard}
  $}
\]
\caption{(a) An invisible white terminal next to a black stone. (b) A
  4-terminal region with several invisible terminals.}
\label{fig-invisible-terminal}
\end{figure}

Another interesting phenomenon is that the boundary of a Hex region
can include what we may call ``invisible terminals''. For example, in
Figure~\ref{fig-invisible-terminal}(a), we have inserted a thin white
rectangle between two empty cells next to a black stone. Note that we
may regard this rectangle as an additional (strangely shaped) board
cell that is occupied by White. Indeed, doing so does not affect the
connectivity of other cells: if $a$ and $b$ are both occupied by
White, they are connected via the white rectangle, and if they are
both occupied by Black, they are connected via the black stone. Since
such thin rectangles are only imagined and not usually visible, we
call them {\em invisible rectangles}. They can be white or black. When
invisible rectangles form part of a region's boundary, as in
Figure~\ref{fig-invisible-terminal}(b), we refer to them as {\em
  invisible terminals}. The two-colored bridge of
Figure~\ref{fig-two-colored-bridges} is in fact a special case of an
invisible terminal whose color does not matter.

We can also consider other types of regions besides $n$-terminal
regions. For example, consider the region shown in
Figure~\ref{fig-region-gap}(a). This is a 2-terminal region whose
boundary contains a gap marked ``$\ast$''. In this example, we do not
consider the gap itself to be part of the region. Within the region,
there are only three distinct outcomes:
\begin{itemize}
\item $\bot$: White's terminals are connected.
\item $a$: Neither Black's nor White's terminals are connected within
  the region (but Black's terminals would be connected if the cell
  marked ``$\ast$'' were occupied by a black stone, and similarly for
  White's terminals).
\item $\top$: Black's terminals are connected.
\end{itemize}
Here is an example of each outcome:
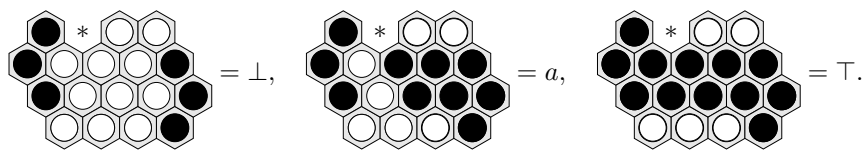
\[
\def\scale{0.7}
\m{\scalebox{\scale}{$
  \begin{hexboard}
    \rotation{-30}
    \foreach\i in {2,...,4} {\hex(1,\i)\white(1,\i)}
    \foreach\i in {1,...,4} {\hex(2,\i)\white(2,\i)}
    \foreach\i in {2,...,4} {\hex(3,\i)\white(3,\i)}
    \foreach\i in {1,...,4} {\hex(4,\i)\white(4,\i)}
    \foreach\i in {1,...,3} {\hex(5,\i)\white(5,\i)}
    \black(2,1)
    \black(1,2)
    \black(1,3)
    \white(1,4)
    \white(2,4)
    \white(3,4)
    \black(4,4)
    \black(5,3)
    \black(5,2)
    \white(5,1)
    \white(4,1)
    \cell(3,1)\label{\Large$\ast$}
  \end{hexboard}$}} = \bot,
\quad
\m{\scalebox{\scale}{$
  \begin{hexboard}
    \rotation{-30}
    \foreach\i in {2,...,4} {\hex(1,\i)\white(1,\i)}
    \foreach\i in {1,...,4} {\hex(2,\i)\white(2,\i)}
    \foreach\i in {2,...,4} {\hex(3,\i)\black(3,\i)}
    \foreach\i in {1,...,4} {\hex(4,\i)\black(4,\i)}
    \foreach\i in {1,...,3} {\hex(5,\i)\black(5,\i)}
    \black(2,1)
    \black(1,2)
    \black(1,3)
    \white(1,4)
    \white(2,4)
    \white(3,4)
    \black(4,4)
    \black(5,3)
    \black(5,2)
    \white(5,1)
    \white(4,1)
    \cell(3,1)\label{\Large$\ast$}
  \end{hexboard}$}} = a,
\quad
\m{\scalebox{\scale}{$
  \begin{hexboard}
    \rotation{-30}
    \foreach\i in {2,...,4} {\hex(1,\i)\black(1,\i)}
    \foreach\i in {1,...,4} {\hex(2,\i)\black(2,\i)}
    \foreach\i in {2,...,4} {\hex(3,\i)\black(3,\i)}
    \foreach\i in {1,...,4} {\hex(4,\i)\black(4,\i)}
    \foreach\i in {1,...,3} {\hex(5,\i)\black(5,\i)}
    \black(2,1)
    \black(1,2)
    \black(1,3)
    \white(1,4)
    \white(2,4)
    \white(3,4)
    \black(4,4)
    \black(5,3)
    \black(5,2)
    \white(5,1)
    \white(4,1)
    \cell(3,1)\label{\Large$\ast$}
  \end{hexboard}$}} = \top.
\]
The set of outcomes is linearly ordered; we have $\bot<a<\top$. The
order is also shown in Figure~\ref{fig-region-gap}(b).

Note that we could have alternatively regarded the region in
Figure~\ref{fig-region-gap} as a 3-terminal region, by inserting two
invisible terminals (a white one and a black one) between the region
and the gap. However, this would have yielded a less precise outcome
poset, because a generic 3-terminal region has five possible outcomes,
whereas a 2-terminal region with gap only has three possible outcomes.

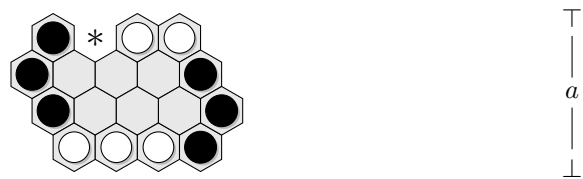
\begin{figure}
  \[
  \m{$
    \begin{hexboard}[scale=0.8]
      \shadows
      \rotation{-30}
      \foreach\i in {2,...,4} {\hex(1,\i)}
      \foreach\i in {1,...,4} {\hex(2,\i)}
      \foreach\i in {2,...,4} {\hex(3,\i)}
      \foreach\i in {1,...,4} {\hex(4,\i)}
      \foreach\i in {1,...,3} {\hex(5,\i)}
      \black(2,1)
      \black(1,2)
      \black(1,3)
      \white(1,4)
      \white(2,4)
      \white(3,4)
      \black(4,4)
      \black(5,3)
      \black(5,2)
      \white(5,1)
      \white(4,1)
      \cell(3,1)\label{\Large$\ast$}
    \end{hexboard}
    $}
  \hspace{4cm}
  \m{
    \begin{tikzpicture}[scale=0.5]
      \node(top) at (0,2) {$\top$};
      \node(a) at (0,0) {$a$};
      \node(bot) at (0,-2) {$\bot$};
      \draw (bot) -- (a) -- (top);
    \end{tikzpicture}
    }
  \]
  \caption{(a) A 2-terminal region with a gap. (b) The outcome poset
    for this region.}
  \label{fig-region-gap}
\end{figure}

As another example, consider the region shown in
Figure~\ref{fig-region-edges}(a). This is a 3-terminal region, with
the additional property that two of White's terminals are board edges
(or connected to board edges). This edge condition affects the outcome
poset: of the five possible outcomes $\s{\bot,a,b,c,\top}$ for a
generic 3-terminal region, two now become equivalent. Namely, if
Black's terminal 3 is connected to neither terminal 1 nor terminal 2,
then White's two edges are connected, which means that White wins the
surrounding game. In that case, it does not matter whether terminals 1
and 2 are connected to each other, so the outcomes $\bot$ (none of the
Black's terminals are connected) and $c$ (only terminals 1 and 2 are
connected) are now equivalent. Consequently, this type of region has
only four distinct outcomes $\s{\bot,a,b,\top}$, and their partial
order is as shown in Figure~\ref{fig-region-edges}(b). We call such a
region a {\em fork}, because all that matters about the outcomes is
what terminal 3 is connected to (left, right, both, or none).
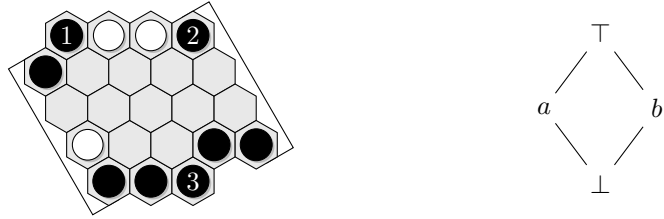
\begin{figure}
  \[
  \m{$
    \begin{hexboard}[scale=0.8]
      \shadows
      \rotation{-30}
      \foreach\i in {2,...,5} {\hex(1,\i)}
      \foreach\i in {1,...,5} {\hex(2,\i)}
      \foreach\i in {1,...,5} {\hex(3,\i)}
      \foreach\i in {1,...,4} {\hex(4,\i)}
      \foreach\i in {1,...,4} {\hex(5,\i)}
      \black(1,2)
      \black(2,1)\label{1}
      \white(3,1)
      \white(4,1)
      \black(5,1)\label{2}
      \white(1,4)
      \black(1,5)
      \black(2,5)
      \black(3,5)\label{3}
      \black(4,4)
      \black(5,4)
      \edge[\sw\noacutecorner\noobtusecorner](1,2)(1,5)
      \edge[\ne\noacutecorner\noobtusecorner](5,4)(5,1)
    \end{hexboard}
    $}
  \hspace{3cm}
  \m{
    \begin{tikzpicture}[scale=0.5]
      \node(top) at (0,2) {$\top$};
      \node(a) at (-1.5,0) {$a$};
      \node(b) at (1.5,0) {$b$};
      \node(bot) at (0,-2) {$\bot$};
      \draw (bot) -- (a) -- (top);
      \draw (bot) -- (b) -- (top);
    \end{tikzpicture}
    }
  \]
  \caption{(a) A 3-terminal region with edges, or ``fork''. (b) The
    outcome poset for this region.}
  \label{fig-region-edges}
\end{figure}

In general, a Hex region consists of a set of cells, some of which may
already be occupied by black or white stones (such as the boundaries
in the above examples). We may or may not stipulate additional
restrictions on how the region can be embedded in a larger game of
Hex; for example, we may require certain stones to be connected to
board edges. A {\em completion} of the region is an assignment of
black and white stones to all of its empty cells. We define a preorder
on the set of completions as outlined above, i.e., $x\leq y$ if for
every allowable embedding of the region in a larger game of Hex, and
for every completion of the remaining game, if $x$ gives a winning
position for Black then so does $y$. Two completions are {\em
  equivalent} is $x\leq y$ and $y\leq x$, and an {\em outcome} for the
region is an equivalence class of completions.

\section{Set Coloring Games}\label{sec-set-coloring-games}

Before we develop a combinatorial game theory for Hex, we describe a
more general class of games to which this theory can be applied. These
are the set coloring games of van Rijswijck {\cite{VanRijswijck}}. We
will only be concerned with set coloring games with two players, each
of whom uses one color. Rather than just considering games in which
one player wins and the other loses, we will consider set coloring
games over some arbitrary poset of outcomes. They are defined as
follows. Let $\Bool=\s{\bot,\top}$ be the set of booleans. Here,
$\bot$ denotes ``false'' or ``bottom'', and $\top$ denotes ``true'' or
``top''. If $X$ and $Y$ are sets, we write $Y^X$ for the set of
functions from $X$ to $Y$.

\begin{definition}[Set coloring game]
  Let $A$ be a partially ordered set, whose elements we call {\em
    outcomes}.  A {\em set coloring game} over $A$ is given by a
  finite set $X$ of {\em cells}, and a function $\pi:\Bool^X\to A$
  called a {\em payoff function}. It is played as follows: there is a
  game board, initially empty, whose cells are in one-to-one
  correspondence with the elements of $X$. The players, whom we call
  Black and White, take turns, with Black starting. Alternately, Black
  colors one cell with color $\top$, and White colors one cell with
  color $\bot$. This continues until all cells are colored. A fully
  colored game board corresponds to a function from $X$ to $\Bool$,
  i.e., an element $f\in\Bool^X$. Then the outcome of the game is
  $\pi(f)\in A$.
\end{definition}

We call a completely filled board an {\em atomic} position. An atomic
position has no moves for either player. All other positions have at
least one move for each player. Of particular interest to us are
set coloring games that are monotone.

\begin{definition}[Monotone set coloring game]
  We equip the set $\Bool$ with the natural order $\bot<\top$. Note
  that this induces a partial order on $\Bool^X$ given pointwise,
  i.e., $f\leq g$ if for all $x\in X$, we have $f(x)\leq g(x)$. A set
  coloring game $(X,\pi)$ is called {\em monotone} if its payoff
  function $\pi$ is a monotone function with respect to this order.
\end{definition}

Evidently, Hex is a monotone set coloring game over $\Bool$, since the
winning condition (Black has a connection between her two edges)
remains valid when changing any number of stones from white to
black. Moreover, any Hex region is a monotone set coloring game over
its outcome poset. For example, a 3-terminal region, as defined in
Section~\ref{sec-outcomes}, is a monotone set coloring game over
$A=\s{\bot,a,b,c,\top}$.

We can also consider certain generalizations of Hex that are monotone
set coloring games. A {\em planar connection game} is like Hex, except
the board is tiled by arbitrary polygonal cells, not necessarily
hexagons. To ensure that there is a unique winner, we must require
that at most three cells meet at any point. See {\cite{Pex}} for an
example. Another, even more general example of a monotone set coloring
game is the {\em vertex Shannon game}. It is played on an arbitrary
undirected graph with two distinguished vertices called {\em terminals}. The
players alternately color a non-terminal vertex in their color until
all vertices are colored; Black wins if and only if in the resulting
coloring, the two terminals are connected by a path of black
vertices. (Unlike Hex and other planar connection games, the class of
vertex Shannon games is not obviously self-dual, i.e., it is not
obviously invariant under switching the roles of the two players.)

\section{A Class of Combinatorial Games}\label{sec-combinatorial-games}

We are now ready to develop a notion of combinatorial games that is
appropriate for Hex and other monotone set coloring games.  In many
ways, the theory of these games is similar to standard combinatorial
game theory, say for normal play games. As already mentioned in the
introduction, the main difference is how the games end. Our games end
when an atomic position is reached, and in any non-atomic position,
there is at least one possible move for each player (so that the game
can never end due to a player's inability to make a move). In
addition, as we will see, our games are designed so that when $a$ is
an atomic outcome, the games $a$ and $\g{a|a}$ are equivalent. This
reflects the fact that in monotone set coloring games, once the
outcome in a region is determined, it does not matter if the players
are allowed additional useless moves. In this regard, our theory
differs from the class of games, such as Gomoku, where the outcome is
determined by which player {\em first} achieves a winning
condition. In such games, $\g{\top|\top}$ is not equivalent to $\top$,
because if one player needs one more move to achieve a win, the other
player might win first by playing in another region of the game.

Our notion of games has different properties than normal play
games. For example, in normal play games, $G\tri H$ holds if and only
if $H\nleq G$, whereas in our games, we often have $G\tri G$. Thus,
while many aspects of our proofs are the same as for other kinds of
combinatorial games, some crucial details are different.
Consequently, we give full proofs even of results that seem
elementary. Where proofs have been omitted, they are completely
routine.

\subsection{Games over an Outcome Poset}

\begin{definition}[Game over a poset]\label{def-game}
  Let $A$ be a partially ordered set, whose elements we call {\em
    atoms} or {\em outcomes}. Games over $A$ are defined inductively:
  \begin{itemize}
  \item For every $a\in A$, $[a]$ is a game, called an {\em atomic
    game}. We usually write $a$ instead of $[a]$ when no confusion
    arises.
  \item If $L$ and $R$ are non-empty sets of games, then $G=\g{L|R}$
    is a game, called a {\em composite game}. As usual, $L$ and $R$
    are called the sets of {\em left} and {\em right options} of $G$,
    respectively.
  \end{itemize}
\end{definition}

The fact that it is an inductive definition means that games are
freely generated by the two rules above. When we say that two games
are {\em equal}, we mean that they are literally the same game, i.e.,
they are both atomic with equal atoms, or they are both composite with
equal sets of left options and equal sets of right options. This
should not be confused with the notion of {\em equivalence} of games
that we will define later. Note that we have required non-atomic
games to have at least one left option and at least one right
option. Such games are also called ``all-small'' or ``dicot'' in the
context of normal-play and mis\`ere games, respectively
{\cite{ONAG,Milley2013}}.

We will follow the usual notational conventions of combinatorial game
theory. For example, we write $\g{x,y|z,w}$ instead of
$\g{\s{x,y}|\s{z,w}}$. We often write $G^L$ and $G^R$ for a typical
left and right option of $G$, respectively. We sometimes write $G =
\g{G^L|G^R}$ to indicate a game that has (possibly more than one)
typical left and right option. As is usual in combinatorial game
theory, the two players are called Left and Right. In set coloring
games, we identify Black with Left and White with Right.

\begin{definition}[Well-founded relation and Conway induction]
  On the collection of all games over $A$, we define $\ll$ to be the
  smallest reflexive transitive relation such that for any composite
  game $G$, we have $G^L\ll G$ for all left options $G^L$ of $G$ and
  and $G^R\ll G$ for all right options $G^R$ of $G$. Then $\ll$ is a
  well-founded partial order, i.e., it has no infinite strictly
  descending sequences. If $G\ll H$, we say that $G$ is a {\em
    position} of $H$. In other words, the positions of $H$ are $H$
  itself, the options of $H$, the options of options of $H$, and so
  on. If $G\ll H$ and $G\neq H$, we say that $G$ is a {\em smaller}
  game than $H$. Informally, it means that $G$ was defined ``before''
  $H$. We can prove statements about games by induction on this
  well-founded relation. This kind of induction is often called {\em
    Conway induction}.
\end{definition}

We will often (but not always) assume that the outcome set $A$ has
top and bottom elements, which we denote by $\top$ and $\bot$,
respectively.

\subsection{Order}\label{ssec-order}

We now define relations $G\leq H$ and $G\tri H$ on games. As in
standard combinatorial game theory, the intuition is that $G\leq H$
means that from Left's point of view, being first to move in $H$ is at
least as good as being first to move in $G$, and being second to move
in $H$ is at least as good as being second to move in $G$. Also,
$G\tri H$ means that from Left's point of view, being first to move in
$H$ is at least as good as being second to move in $G$. As usual, we
take Left's point of view, i.e., ``better'' means ``better for Left''
unless stated otherwise.

In combinatorial game theory, it is common to define these relations
by first defining the sum and negation operations on games, and then
to take $G\leq H$ to mean $0\leq H-G$. For reasons that will become
apparent later, this definition is not convenient in our setting ---
among other things, because there is no game ``$0$'', and because the
sum of games is not well-defined until we consider monotone games,
which requires the order $\leq$ to be defined first. We therefore
define the relations $G\leq H$ and $G\tri H$ directly, without
reference to sums.

Except for the atomic cases, the definitions of $\leq$ and $\tri$ are
the same as in standard combinatorial game theory. Since an atomic game
has no left or right options, by convention, when $G$ is atomic, we
take any statement of the form ``for all left options $G^L$'' to be
vacuously true, and any statement of the form ``there exists a left
option $G^L$'' to be trivially false.

\begin{definition}\label{def-order}
  For games over a partially ordered set $A$, the relations $\leq$ and
  $\tri$ are defined by mutual recursion as follows.
  \begin{itemize}
  \item $G\leq H$ if all three of the following conditions hold:
    \begin{enumerate}
    \item All left options $G^L$ satisfy $G^L\tri H$, and
    \item all right options $H^R$ satisfy $G\tri H^R$, and
    \item if $G$ or $H$ is atomic, then $G\tri H$.
    \end{enumerate}
  \item $G\tri H$ if at least one of the following
    conditions holds:
    \begin{enumerate}
    \item There exists a right option $G^R$ such that $G^R\leq H$, or
    \item there exists a left option $H^L$ such that $G\leq H^L$, or
    \item $G=[a]$ and $H=[b]$ are atomic and $a\leq b$.
    \end{enumerate}
  \end{itemize}
\end{definition}

Since this definition is at the heart of this paper, and since it is
not self-evident that it captures the correct notions, some further
explanations are in order. The fact that it is the ``right''
definition will also be substantiated by the theorems and applications
that follow later.

The first thing to note is that clauses 1.\@ and 2.\@ in the
definition of $\leq$ and $\tri$ are exactly the standard ones that can
be found, for example, in normal play games. Thus, if both games are
composite, there is nothing unusual in this definition. The novelty
lies in the treatment of atomic games.

To understand the atomic cases, consider the postulate that the game
$a$ should be equivalent to $\g{a|a}$, or in symbols,
$a\eq\g{a|a}$. The reason for the postulate is that it is true in
monotone set coloring games. Specifically, in a set coloring game, a
cell is called {\em dead} (or sometimes {\em negligible}) if the color
of that cell does not affect the outcome of the game
{\cite{Bjornsson-Hayward-Johanson-VanRijswijck,Yamasaki}}. The atomic
game $a$ represents a completely filled board, and the game $\g{a|a}$
represents a board with one remaining dead cell.  In a monotone set
coloring game, adding a dead cell does not change the strategic value
of the game, because playing there is not to either player's
advantage, and the cell will eventually be filled anyway.

If we postulate $a\eq\g{a|a}$ and plug this into the usual recursive
clauses for $\leq$ and $\tri$, we get the following properties:
\begin{equation}\label{eq-circular}
\begin{tabular}{ccc@{~}c@{~}c}
  $a\leq H$ &$\iff$& $a\tri H$ &and& $\forall H^R.a\tri H^R$.\\
  $a\tri H$ &$\iff$& $a\leq H$ &or& $\exists H^L.a\leq H^L$.\\
\end{tabular}
\end{equation}
Note, however, that we cannot take {\eqref{eq-circular}} as a
definition, because it is circular: $a\leq H$ is described in terms of
$a\tri H$ and vice versa. It turns out, however, that the clauses
in {\eqref{eq-circular}} have both a largest and a smallest solution, i.e.,
among the pairs of relations $(\leq,\tri)$ satisfying
{\eqref{eq-circular}}, there is a largest and a smallest one. It turns
out that the smallest solution is the one that we need. More
generally, consider a circular system of boolean equations of the form
\begin{equation}\label{eq-circular-propositional}
\begin{array}{ccccc}
  P &\iff& Q &\cand& R, \\
  Q &\iff& P &\cor&  S. \\
\end{array}
\end{equation}
It is easy to check that this system has exactly five solutions:
$(P,Q,R,S)=(\top,\top,\top,\top)$, $(\bot,\top,\bot,\top)$,
$(\top,\top,\top,\bot)$, $(\bot,\bot,\top,\bot)$, and
$(\bot,\bot,\bot,\bot)$.  In particular, unless $(R,S)=(\top,\bot)$,
$P$ and $Q$ are uniquely determined. In case $(R,S)=(\top,\bot)$,
there are the two solutions $(P,Q)=(\top,\top)$ and
$(P,Q)=(\bot,\bot)$. If we choose the second of these, the equations
{\eqref{eq-circular-propositional}} simplify to
\[
\begin{array}{ccccc}
  P &\iff& Q &\cand& R, \\
  Q &\iff& && S. \\
\end{array}
\]
Applying this process to {\eqref{eq-circular}} leads us to define
\[
\begin{tabular}{ccc@{~}c@{~}c}
  $a\leq H$ &$\iff$& $a\tri H$ &and& $\forall H^R.a\tri H^R$,\\
  $a\tri H$ &$\iff$& && $\exists H^L.a\leq H^L$.\\
\end{tabular}
\]
Note that $a\leq H$ nevertheless implies $a\tri H$, so that
{\eqref{eq-circular}} is also satisfied.

For the case where both games are atomic, the postulates $a\eq\g{a|a}$
and $b\eq\g{b|b}$ immediately yield that $a\leq b$ if and only if $a\tri
b$. This makes sense because if a game is atomic, its value is already
determined, so it no longer matters whose turn it is. The ordering of
atomic games is given a priori by the poset structure on the set of
atoms.

In summary, we arrive at the following desired properties:
\begin{itemize}
\item If both $G$ and $H$ are composite:\par
  \begin{tabular}{ccc@{~}c@{~}c}
    $G\leq H$ &$\iff$& $(\forall G^L. G^L\tri H)$ &and& $(\forall H^R.G\tri H^R)$,\\
    $G\tri H$ &$\iff$& $(\exists G^R. G^R\leq H)$ &or& $(\exists H^L.G\leq H^L)$.\\
  \end{tabular}
  
\item If $G$ is composite and $H=[b]$ is atomic:\par
  \begin{tabular}{ccc@{~}c@{~}c}
    $G\leq [b]$ &$\iff$& $(\forall G^L. G^L\tri [b])$ &and& $G\tri [b]$,\\
    $G\tri [b]$ &$\iff$& $\exists G^R. G^R\leq [b]$.
  \end{tabular}
  
\item If $G=[a]$ is atomic and $H$ is composite:\par
  \begin{tabular}{ccc@{~}c@{~}c}
    $[a]\leq H$ &$\iff$& $[a]\tri H$ &and& $\forall H^R.[a]\tri H^R$,\\
    $[a]\tri H$ &$\iff$& && $\exists H^L.[a]\leq H^L$.\\
  \end{tabular}
  
\item If $G=[a]$ and $H=[b]$ are both atomic:\par
  \begin{tabular}{ccccc}
    $[a]\leq [b]$ &$\iff$& $[a]\tri [b]$ &$\iff$& $a\leq b$.
  \end{tabular}
\end{itemize}
It is easy to check that Definition~\ref{def-order} is merely a more
compact statement of these properties.

It is worth emphasizing that we are not saying that the games $a$ and
$\g{a|a}$ are actually equal, as this would give rise to games with
infinite plays. Instead, we are merely saying that
Definition~\ref{def-order} is motivated by the desire to make $a$ and
$\g{a|a}$ {\em equivalent}. As we will see in
Lemma~\ref{lem-atomic-composite} below, this is indeed one of the
consequences of Definition~\ref{def-order}. Thus, we get the property
$a\eq\g{a|a}$ without having to consider infinite games or strategies.

\subsection{Properties of the Order}
\label{ssec-lemmas-order}

The relations $\leq$ and $\tri$ enjoy many of the usual properties
that hold, say, in normal play games, but not all of them. For
example, in normal play games, we have $G\tri H$ if and only if
$H\nleq G$. This is not the case here: all atomic games satisfy $a\tri
a$ and $a\leq a$, and we will see that many non-atomic games satisfy
$G\tri G$ as well.

In this section, we prove some basic properties of the order
relations. Throughout this section, we consider games over a fixed
poset $A$ of atoms.

\begin{remark}[Duality]
  For every statement about games, there is a dual statement, obtained
  by exchanging the roles of the left and right players (and replacing
  $\leq$ by $\geq$ and so on). A property of games is valid if and
  only if its dual is valid, so in the following, we sometimes prove a
  lemma and then use both the lemma and its dual.
\end{remark}

\begin{lemma}[Reflexivity]
  The relation $\leq$ is reflexive.
\end{lemma}

\begin{proof}
  We prove $G\leq G$ by induction on $G$. For atomic games, we have
  $[a]\leq [a]$ by definition. Suppose that $G$ is composite. To show 
  $G\leq G$, first take any left option $G^L$. By the induction
  hypothesis, $G^L\leq G^L$, hence by  definition of $\tri$, we have
  $G^L\tri G$. Similarly, take any right option $G^R$. By the induction
  hypothesis, $G^R\leq G^R$, hence by definition of $\tri$, we have
  $G\tri G^R$. It follows that $G\leq G$ as desired.
\end{proof}

\begin{lemma}[Transitivity]\label{lem-transitive}
  For games $G,H,K$ over $A$, we have:
  \begin{enumerate}\alphalabels
  \item $G\tri H\leq K$ implies $G\tri K$;
  \item $G\leq H\tri K$ implies $G\tri K$;
  \item $G\leq H\leq K$ implies $G\leq K$.
  \end{enumerate}
\end{lemma}

\begin{proof}
  We prove all three properties by simultaneous induction. The base
  cases will be handled uniformly along with the non-base
  cases. Therefore, let $G,H,K$ be any games (atomic or not), and
  assume that (a)--(c) are true for smaller triples of games. (Triples
  of games can be ordered componentwise with respect to the relation
  $\ll$, which is again a well-founded relation.) We first prove (a),
  then (b), then (c) (so that the proof of (c) can use the results of
  (a) and (b), even for games of the same size).

  \begin{enumerate}\alphalabels
  \item To show (a), assume $G\tri H$ and $H\leq K$. We must show
    $G\tri K$.  From the definition of $G\tri H$, there are three
    cases:
    
    \begin{itemize}
      \item Case 1: There exists some $G^R$ such that $G^R\leq
        H$. Since $G^R$ is smaller than $G$, by the induction hypothesis,
        from $G^R\leq H\leq K$, we get $G^R\leq K$. Therefore $G\tri
        K$ as desired.
        
      \item Case 2: There exists some $H^L$ such that $G\leq H^L$. By
        the definition of $H\leq K$, we have $H^L\tri K$. Since $H^L$
        is smaller than $H$, by the induction hypothesis, from $G\leq
        H^L\tri K$, we get $G\tri K$ as desired.
    
      \item Case 3: $G=[a]$ and $H=[b]$ are atomic and $a\leq b$. From
        $H=[b]\leq K$, by definition, we get $[b]\tri K$. By
        definition of $[b]\tri K$, there are two possible cases:

        Case 3.1: There exists some $K^L$ such that $[b]\leq K^L$. In
        that case, since $K^L$ is smaller than $K$, by the induction
        hypothesis, from $G\leq[b]\leq K^L$, we get $G\leq K^L$, hence
        $G\tri K$ as desired.

        Case 3.2: $K=[c]$ is atomic and $b\leq c$. In this case, from
        $a\leq b\leq c$, we get $a\leq c$, hence $G=[a]\tri[c]=K$ as
        desired.
    \end{itemize}
    
  \item The proof of (b) is dual to that of (a).
    
  \item
    To show (c), assume $G\leq H$ and $H\leq K$. We must show $G\leq
    K$. Using the definition of $G\leq K$, there are three things we
    need to prove:

    \begin{itemize}
    \item We must show that every $G^L$ satisfies $G^L\tri K$. So let
        $G^L$ be any left option of $G$.  Since $G\leq H$ was assumed,
        we have $G^L\tri H$. Since $G^L$ is smaller than $G$, by
        the induction hypothesis, from $G^L\tri H\leq K$, we get $G^L\tri
        K$, as desired.
    
      \item We must show that every $K^R$ satisfies $G\tri K^R$. The
        proof is dual to the previous case.
    
      \item Finally, we must show that if either $G$ or $K$ is atomic,
        then $G\tri K$. First assume $G$ is atomic. Since $G$ is
        atomic, from $G\leq H$, we get $G\tri H$.  By (a), from $G\tri
        H\leq K$, we get $G\tri K$, as desired. The case where $K$ is
        atomic is dual.\qedhere
    \end{itemize}

  \end{enumerate}
\end{proof}

\begin{corollary}
  The relation $\leq$ forms a preorder on the collection of games over
  a poset $A$.
\end{corollary}

At this point, it is customary to quotient out the equivalence
relation induced by this preorder, i.e., to identify games $G$ and $H$
if $G\leq H$ and $H\leq G$. We will not do so here, because we will
sometimes need to speak of properties of games that are not invariant
under equivalence. We will continue to write $G=H$ to mean that $G$
and $H$ are literally equal, and we will say that $G$ and $H$ are {\em
  equivalent}, written $G\eq H$, if $G\leq H$ and $H\leq G$. We will
reserve the word {\em value}, or {\em combinatorial value}, to mean an
equivalence class of games; thus, when we say two games have the same
value, we mean that they are equivalent.

\begin{lemma}\label{lem-composite-monotone}
  If $H\leq H'$, then $\g{H,G^L|G^R}\leq\g{H',G^L|G^R}$ and
  $\g{G^L|G^R,H} \leq \g{G^L|G^R,H'}$.
\end{lemma}

\begin{proof}
  This follows directly from the definitions.
\end{proof}

\begin{lemma}\label{lem-more-options}
  If $G$ is composite, we have $G\leq \g{H,G^L|G^R}$. In other
  words, an additional left option can only help Left.
\end{lemma}

\begin{proof}
  This follows directly from the definitions.
\end{proof}

\begin{lemma}[Gift horse lemma]
  \label{lem-gift-horse}
  If $G$ is composite, we have $H\tri G$ if and only if $G\eq
  \g{H,G^L|G^R}$.  Dually, if $H$ is composite, we have $H\tri G$ if
  and only if $H\eq\g{H^L|H^R,G}$.
\end{lemma}

\begin{proof}
  It suffices to show the first claim, as the second one is its dual.
  For the right-to-left direction, assume $G\eq \g{H,G^L|G^R}$. By
  definition of $\tri$, we have $H\tri \g{H,G^L|G^R}$. With
  $\g{H,G^L|G^R}\leq G$ and Lemma~\ref{lem-transitive}, this yields
  $H\tri G$.  For the left-to-right direction, assume $H\tri G$. Since
  $G\leq \g{H,G^L|G^R}$ holds by Lemma~\ref{lem-more-options}, we only
  need to show $\g{H,G^L|G^R}\leq G$. But this follows directly from
  the definitions.
\end{proof}

Because of Lemma~\ref{lem-gift-horse}, when $H\tri G$, we say that $H$
is a {\em left gift horse} for $G$, and that $G$ is a {\em right gift
  horse} for $H$.

\begin{lemma}[All games are equivalent to composite games]
  \label{lem-atomic-composite}
  For $a\in A$, we have $a\eq\g{a|a}$.
\end{lemma}

\begin{proof}
  This follows directly from the definitions.
\end{proof}

\begin{lemma}\label{lem-top}
  If $A$ has a top element $\top$, then all games $G$ over $A$ satisfy
  $G\leq\top$ and $G\tri\top$. Dually, if $A$ has a bottom element
  $\bot$, then all games $G$ over $A$ satisfy $\bot\leq G$ and
  $\bot\tri G$.
\end{lemma}

\begin{proof}
  We show that $G\leq\top$ and $G\tri\top$ by induction. If $G$ is
  atomic, this is clear by definition. So assume $G$ is composite. To
  show $G\tri\top$, pick any right option $G^R$ of $G$ (recall that
  every composite game has at least one right option). By the induction
  hypothesis, $G^R\leq\top$, so by definition of $\tri$, we have
  $G\tri\top$.  To show $G\leq\top$, first consider any left option
  $G^L$; by the induction hypothesis, we have $G^L\tri\top$. Since $\top$
  has no right options, and since we already showed $G\tri\top$, the
  definition of $\leq$ then implies that $G\leq\top$ as desired.
\end{proof}

If the atom poset $A$ has top and bottom elements, the partial order
of values is particularly nice: it is a complete lattice, as the
following proposition shows.

\begin{proposition}\label{prop-lattice}
  Assume the atom poset has a bottom element. Then any set of games
  has a least upper bound.
\end{proposition}

\begin{proof}
  The least upper bound of the empty set is $\bot$ by
  Lemma~\ref{lem-top}. Next, we show that any pair of games $H,K$ has
  a least upper bound. By Lemma~\ref{lem-atomic-composite}, we can
  assume without loss of generality that $H$ and $K$ are
  composite. Define
  \[ G = \g{H^L, K^L | \g{H, K | \bot}}.
  \]
  We claim that $G$ is the desired least upper bound. We first note
  that $G'=\g{H,K|\bot}$ is the smallest game such that $H,K\tri
  G'$. Indeed, it is obvious that $H,K\tri G'$. Suppose that $G''$ is
  some arbitrary game with $H,K\tri G''$. Then $G'\leq G''$ follows
  directly from the definition of $\leq$.

  Next, we show that $H, K \leq G$. Indeed, all their left options
  satisfy $H^L, K^L \tri G$ by definition of $G$; conversely, the only
  right option of $G$ satisfies $H, K \tri \g{H, K | \bot}$.
  
  Next, consider some arbitrary game $M$ with $H, K \leq M$. We claim
  that $G \leq M$. Without loss of generality, assume that $M$ is
  composite.  First, consider any left option of $G$. They are $H^L$
  or $K^L$. Both satisfy $H^L, K^L \tri M$ by the assumption that $H,
  K \leq M$. This shows the first part of $G \leq M$.  Next, consider
  any right option $M^R$ of $M$. We must show $G \tri M^R$. From the
  assumption $H,K\leq M$, we get $H,K\tri M^R$, and therefore $\g{H, K
    | \bot} \leq M^R$ by the ``first note'' above. This implies $G
  \tri M^R$ as desired. We have shown two of the three conditions for
  $G \leq M$. Since neither $G$ nor $M$ is atomic, the third condition
  does not apply, so $G\leq M$. Therefore $M$ is the least upper bound
  of $H, K$.

  Finally, nothing in the above proof relies on the fact that we have
  exactly two games $H,K$. The same proof works for any finite or even
  infinite non-empty set of games.
\end{proof}

Note that by the dual of Proposition~\ref{prop-lattice}, greatest
lower bounds also exist, provided that the atom poset has a top
element. Also note that Proposition~\ref{prop-lattice} only provides
suprema of {\em sets} of games, whereas the collection of all games
over a given atom poset is in general a proper class. For this reason,
the usual trick of constructing a greatest lower bound as the least
upper bound of all lower bounds does not work. In particular, we
cannot use Proposition~\ref{prop-lattice} to construct a top element
by taking the supremum of all games.

\subsection{Canonical Forms}

The theory of canonical forms is similar to that for other kinds of
combinatorial games, but some adjustments are needed to deal with
atoms. The notion of bypassing a reversible option must be adjusted,
and we need a new notion of passing option. As before, we consider
games over a fixed atom poset $A$.

\begin{definition}[Dominated option]
  Suppose that $H,K$ are distinct left options of $G$. We say that $H$
  {\em dominates} $K$ if $K\leq H$. Dually, for distinct right options
  $H,K$ of $G$, we say that $H$ dominates $K$ if $H\leq K$.
\end{definition}

\begin{lemma}[Removing dominated options]
  If $G$ has a dominated left option $K$, then $G\eq G'$, where $G'$
  is the result of removing the left option $K$ from $G$. The
  dual statement for right options also holds.
\end{lemma}

\begin{proof}
  Suppose the left option $K$ is dominated by another left option
  $H$. Since $H$ is a left option of $G'$, we have $H\tri G'$.  Since
  $K\leq H$, by Lemma~\ref{lem-transitive}, we have $K\tri
  G'$. Therefore, $K$ is a left gift horse for $G'$ and $G\eq G'$ by
  the Gift Horse Lemma~\ref{lem-gift-horse}. The proof for right
  options is dual.
\end{proof}
 
\begin{definition}[Reversible option]
  Suppose $H$ is a left option of $G$. We say that $H$ is {\em
    reversible} via $K$ if $H$ has a right option $K$ such that $K\leq
  G$. (Therefore both $G$ and $H$ are necessarily
  composite.) Reversible right options are defined dually.
\end{definition}

\begin{lemma}[Bypassing reversible options]\label{lem-bypassing}
  Suppose $G$ is a game with a left option $H$ that is reversible via
  $K$. Then $G\eq G'$, where
  \begin{itemize}
  \item $G'=\g{K^L,G^L|G^R}$, when $K$ is composite, and
  \item $G'=\g{K,G^L|G^R}$, when $K$ is atomic.
  \end{itemize}
  Here, $G^L$ denotes all left options of $G$ other than $H$, and
  $K^L$ denotes all left options of $K$.
\end{lemma}

\begin{proof}
  First, consider the case where $K$ is composite. Since $K\leq G$,
  any left option of $K$ is a gift horse for $G$, so that $G\eq
  \g{K^L,H,G^L|G^R}$. To show that $\g{K^L,H,G^L| G^R}\eq G'$, we need
  to show that $H$ is a left gift horse for $G'$, i.e., $H\tri
  G'$. Since $K$ is a right option of $H$, it suffices to show that
  $K\leq G'$. This requires showing three things:
  \begin{itemize}
  \item First, consider any left option $K^L$ of $K$. We must show
    $K^L\tri G'$. But $K^L$ is, by definition, a left option of $G'$,
    so $K^L\tri G'$ as desired.
  \item Second, consider any right option $G'^R$ of $G'$. We must show
    $K\tri G'^R$. But $G'$ has the same right options as $G$, so that
    $G'^R=G^R$ for some right option of $G$. By assumption, $K\leq G$,
    which implies $K\tri G^R$ as desired.
  \item Third, assume $K$ or $G'$ is atomic. But this is not the case
    as we had assumed otherwise.
  \end{itemize}
  This concludes the proof in the case where $K$ is composite.
  The case where $K$ is atomic can be reduced to the previous case,
  because $K\eq\g{K|K}$ by Lemma~\ref{lem-atomic-composite}.
\end{proof}

The two cases of Lemma~\ref{lem-bypassing} are called {\em non-atomic}
and {\em atomic reversibility}, respectively.

\begin{definition}[Passing option]
  Suppose $H$ is a left option of $G$. We say that $H$ is a {\em
    passing option} of $G$ if $H\eq G$. 
\end{definition}

\begin{lemma}[Simplifying passing options]
  If $H$ is a passing option of $G$, then $G\eq H$.
\end{lemma}

\begin{proof}
  This holds by definition.
\end{proof}

\begin{definition}[Canonical form]
  A game $G$ is in {\em canonical form} if $G$ has no dominated,
  reversible, or passing options, and all left and right options of
  $G$ are in canonical form.
\end{definition}

\begin{lemma}[Uniqueness of canonical form]
  If $G\eq G'$ and both $G$ and $G'$ are in canonical form, then $G=G'$.
\end{lemma}

\begin{proof}
  We prove this by induction. Assume $G\eq G'$ and both are in
  canonical form. To show $G=G'$, we must show that $G$ and $G'$ have
  exactly the same left options and right options, and moreover, if
  $G$ and $G'$ are atomic, then they are equal. The latter is clear,
  for if $G=[a]$ and $G'=[a']$, then $G\eq G'$, by definition, holds
  if and only if $a=a'$.

  Now consider any left option $H$ of $G$. We will first show that
  $H\leq H'$ for some left option $H'$ of $G'$. Since $H$ is a left
  option of $G$, we have $H\tri G$. Since $G\eq G'$, we have $H\tri
  G'$. By definition of $\tri$, there are three possibilities:
  \begin{itemize}
  \item Case 1: $H^R\leq G'$ for some right option $H^R$ of $H$. But
    since $G\eq G'$, we then have $H^R\leq G$, so that $H$ is
    reversible via $H^R$ in $G$, contradicting the assumption that $G$
    is in canonical form.
  \item Case 2: $H\leq G'^L$ for some left option $G'^L$ of $G'$. Then
    let $H'=G'^L$.
  \item Case 3: $H=[b]$ and $G'=[a']$ are atomic and $b\leq a'$.  From
    $G'\leq G$, since $G'$ is atomic, it follows that $G'\tri G$,
    i.e., $[a']\tri G$. Since $G$ is not atomic, it follows that
    $[a']\leq G^L$ for some left option $G^L$ of $G$. Since $b\leq
    a'$, we also have $[b]\leq G^L$. Since both $[b]$ and $G^L$ are
    left options of $G$ and therefore cannot be dominated, we have
    $G^L=[b]$. Since $[a']\leq G^L$, we have $a'\leq b$. But we also
    assumed $b\leq a'$, so $a'=b$, so $G'=H$, so $G\eq H$, so that $H$
    is a passing option of $G$, contradicting the assumption that $G$
    is in canonical form.
  \end{itemize}
  We showed that for every left option $H$ of $G$, there is some left
  option $H'$ of $G'$ such that $H\leq H'$. The same argument shows
  that there is some left option $H''$ of $G$ such that $H'\leq
  H''$. By transitivity, this implies $H\leq H''$. Since both $H$ and
  $H''$ are left options of $G$ and $G$ is in canonical form, $H$
  cannot be dominated, so it follows that $H=H''$. Then $H\leq H'\leq
  H''=H$, so that $H\eq H'$. Since $H$ and $H'$ are in canonical form,
  by the induction hypothesis, $H=H'$.

  We have shown that every left option of $G$ is a left option of
  $G'$. A symmetric argument shows that every left option of $G'$ is a
  left option of $G$, so that $G$ and $G'$ have exactly the same left
  options. The analogous property for right options holds by duality.
\end{proof}

\begin{definition}
  A game is {\em short} if it has only finitely many positions;
  equivalently, a game is short if it has finitely many left and right
  options, and all the left and right options are short.
\end{definition}

\begin{lemma}[Existence of canonical form]
  Every short game has a canonical form (i.e., for every $G$, there
  exists $G'$ such that $G\eq G'$ and such that $G'$ is in canonical
  form).
\end{lemma}

\begin{proof}
  We prove this by induction on the total size of the game, i.e., the
  total number of positions in $G$. Repeatedly remove a dominated
  option, bypass a reversible option, or simplify a passing option
  anywhere in $G$. Since each of these steps strictly decreases the
  size, the procedure terminates.
\end{proof}

\subsection{Monotone Games}

Monotone set coloring games, including Hex, are combinatorial games in
the sense of Definition~\ref{def-game}, but they have additional
properties. In particular, in a monotone set coloring game, an
additional move can never hurt a player: adding some stones of a
player's color to a position is never bad for that player. This
motivates the following definitions.

\begin{definition}[Good option, monotone game]
  \label{def-good}
  Let $G$ be a game over a poset $A$. A left option $G^L$ of $G$ is
  said to be {\em good} if $G\leq G^L$. Dually, a right option $G^R$
  of $G$ is {\em good} if $G^R\leq G$. A game $G$ is called {\em
    locally monotone} if all of its left and right options are good.
  It is said to be {\em monotone} if it is locally monotone and
  recursively, all of its options are monotone.
\end{definition}

Note that in this definition, an option is ``good'' if it is
beneficial for the player to whom the option belongs. This deviates
from our usual convention of always taking Left's point of view.

\begin{example}
  Over the booleans $\Bool$, the games $\top$, $\bot$, and
  $\g{\top|\bot}$ are monotone. The game $\g{\bot|\top}$ is not
  monotone.
\end{example}

Let $\Star=\g{\top|\bot}$. This game is called $\Star$ or ``star''
because it is a first-player win for both players, and thus is is
roughly analogous to the nimber $*=\g{0|0}$ from normal play games.
(However, although these two games share a name, they should only be
considered analogous, not actually equal, because they belong to
different classes of games, where wins are propagated in a different
way.) The game $\Star$ has sometimes appeared under this name in the
Hex literature; for example, it inspired the term {\em star
  decomposition} in {\cite{star-decomposition}}.

\begin{example}\label{exa-not-monotone}
  The game $G=\g{\Star|\Star}$ is not monotone. Indeed, since
  $\top\ntri\bot$, we have $\top\nleq\g{\top|\bot}=\Star$, and
  therefore $\top\ntri\g{\Star|\Star}=G$, and hence
  $\Star=\g{\top|\bot}\nleq G$. Since $\Star$ is a right option of
  $G$, this shows that $G$ is not monotone.
\end{example}

\begin{remark}
  If $G=\g{G^L|G^R}$ is monotone, then $G^L,G^R$ are monotone and
  $G^R\leq G^L$. This is obvious from the definition of monotonicity
  and by transitivity, since $G^R\leq G \leq G^L$. However, the
  converse is not true: if $G^R$ and $G^L$ are monotone, assuming
  $G^R\leq G^L$ is not sufficient to imply that $G=\g{G^L|G^R}$ is
  monotone. Example~\ref{exa-not-monotone} is a counterexample with
  $G^L=G^R=\Star$.
\end{remark}

\begin{lemma}\label{lem-tri-reflexive}
  If\/ $G$ is a monotone game, then $G\tri G$.
\end{lemma}

\begin{proof}
  If $G$ is atomic, this is trivial by definition. If $G$ is composite,
  pick any left option $G^L$ (there is at least one). Since $G$ is
  monotone, we have $G\leq G^L$. Thus, by definition of $\tri$, we
  have $G\tri G$.
\end{proof}

\subsection{The Problem with Monotonicity}\label{ssec-problem-monotonicity}

The problem with the concept of monotonicity is that it is a property
of games, not equivalence classes of games. As the following example
shows, the canonical form of a monotone game is not in general
monotone. The absence of canonical forms for monotone games makes it
awkward to work with monotone games up to equivalence.

\begin{example}[Canonical forms of monotone games need not be monotone]
  \label{exa-canonical-not-monotone}
  Consider the atom poset $A=\s{\bot,a,b,\top}$ with top element
  $\top$, bottom element $\bot$, and $a,b$ incomparable.  Consider the
  following games:
  \[
  J = \g{b|\bot}, \quad
  K = \g{\top|a}, \quad
  H = \g{\top|J}, \quad
  G = \g{K,H|\bot}, \quad
  G'=\g{K,b|\bot}.
  \]
  One can check that $G$ is monotone. Moreover, the left option $H$ is
  reversible via $J$ in $G$, so that $G\eq G'$. In fact, $G'$ is the
  canonical form of $G$. However, $G'$ is not monotone, because
  $G'\nleq b$.
\end{example}

Fortunately, this problem has a solution, which we develop in
Section~\ref{sec-fundamental}. We will define the notion of a {\em
  passable game}, which is closed under canonical forms, and we will
prove a game is equivalent to a monotone game if and only if its
canonical form is passable.

\section{Left and Right Equivalence}\label{sec-left-right}

In this section, we introduce some technical notions that will give us
more insight into the structure of games. This will be useful in the
characterization of monotone games in Section~\ref{sec-fundamental}.
The results of this section also have other applications, for example
to the efficient enumeration of games up to equivalence, which we
describe in greater detail in Section~\ref{sec-enumeration}.  Readers
who are in a hurry can skip directly to Section~\ref{sec-fundamental}
and come back here as necessary.

To better understand equivalence of games, we consider some coarser
equivalence relations, which we call left and right equivalence. For
example, we can say that games $H$ and $K$ are ``left equivalent'' if
they are equivalent when used as left options, i.e., if
$\g{H,G^L|G^R}\eq \g{K,G^L|G^R}$ for all $G^L,G^R$. In a similar vein,
it makes sense to consider left and right orderings; for example, we
define $H\leql K$ if $\g{H,G^L|G^R}\leq \g{K,G^L|G^R}$ for all
$G^L,G^R$.

\subsection{The Left and Right Orders}

In the following, we use the letters $S$ and $T$ to denote non-empty
sets of games. The case where $S$ and $T$ are single games is a
special case. For notational convenience, when $S$ consists of a
single game $H$, we write $S=H$ rather than $S=\s{H}$.

\begin{definition}
  Let $S$ and $T$ be non-empty sets of games. The {\em left order} is
  defined as follows: We say that $S\leql T$ if for every (possibly
  empty) set of games $L$ and every non-empty set of games $R$, we have
  $\g{S,L|R}\leq\g{T,L|R}$. The {\em right order} is defined dually:
  we say that $S\leqr T$ if $\g{L|R,S}\leq\g{L|R,T}$ for all
  $R$ and all non-empty $L$.
\end{definition}

By Lemma~\ref{lem-composite-monotone}, $H\leq K$ implies $H\leql
K$. However, the converse is not true. For example, we have
$\g{\top|\g{a|\bot}}\leql a$. This can be seen by noting that as a
left option, $\g{\top|\g{a|\bot}}$ is always reversible, and bypassing
it results in $a$. On the other hand, $\g{\top|\g{a|\bot}}\nleq a$.

The next two lemmas give useful characterizations of the left order.
If $S$ is a set of games and $G$ is a game, we write $S\tri G$ to mean
that $H\tri G$ holds for all $H\in S$. Similarly, we write $S\leq G$
to mean that $H\leq G$ holds for all $H\in S$.

\begin{lemma}\label{lem-102}
  Let $S$ and $T$ be non-empty sets of games over any poset $A$ (it
  does not need to have a top or bottom element). Then the following
  are equivalent:
  \begin{enumerate}\alphalabels
  \item $S\leql T$.
  \item For all $G$, $T\tri G$ implies $S\tri G$.
  \end{enumerate}
  In particular, $H\leql K$ if and only if every right gift horse for
  $K$ is a right gift horse for $H$.
\end{lemma}

\begin{proof}
  To show that (a) implies (b), assume $S\leql T$ and $T\tri G$. We
  must show $S\tri G$. Since each $K\in T$ is a left gift horse for
  $G$, we have $G \eq \g{T, G^L | G^R}$. Since $S\leql T$, we have
  $\g{S, G^L | G^R} \leq \g{T, G^L | G^R}$, and therefore $\g{S, G^L |
    G^R} \leq G$ by transitivity. Then by definition of $\leq$, we
  have $H\tri G$ for all $H\in S$, thus $S\tri G$ as desired.

  To show the opposite implication, assume (b) holds. We must show
  $S\leql T$. Consider any $L,R$. We must show
  $\g{S,L|R}\leq\g{T,L|R}$.  We clearly have $T\tri \g{T,L|R}$, so by
  assumption (b), we have $S\tri \g{T,L|R}$. Then
  $\g{S,L|R}\leq\g{T,L|R}$ follows directly from the definition of
  $\leq$.

  The final claim of the lemma is just the special case when
  $S=H$ and $T=K$ are singletons.
\end{proof}

\begin{lemma}\label{lem-101}
  Suppose the atom poset $A$ has a bottom element, and let $S,T$ be
  non-empty sets of games over $A$. Then the following are
  equivalent:
  \begin{enumerate}\alphalabels
  \item $S\leql T$,
  \item $\g{S|\bot} \leq \g{T|\bot}$,
  \item $S\tri \g{T|\bot}$.
  \end{enumerate}
\end{lemma}

\begin{proof}
  (a) $\imp$ (b) follows directly from the definition of $\leql$.  (b)
  $\imp$ (c) follows directly from the definition of $\leq$.  To prove
  (c) $\imp$ (a), assume $S\tri \g{T|\bot}$. Consider any $L$,
  $R$. We must show $\g{S,L|R} \leq \g{T,L|R}$. But we have
  $\g{T|\bot}\leq \g{T,L|R}$, so the assumption $S\tri \g{T|\bot}$
  implies $S\tri\g{T,L|R}$ by Lemma~\ref{lem-transitive}. Then
  $\g{S,L|R} \leq \g{T,L|R}$ follows directly from the
  definition of $\leq$. 
\end{proof}

Note that the condition in the definition of the left order says that
$\g{S,L|R}\leq\g{T,L|R}$ must hold for ``all'' $L,R$. One may wonder
whether we would wind up with a different definition of left order if
we restricted $L$ and $R$, for example to sets of monotone games.
Lemma~\ref{lem-101} shows that this makes no difference, since it
suffices to consider $L=\emptyset$ and $R=\bot$.

\begin{lemma}\label{lem-103}
  Let $A$ be any poset (it does not need to have a top or bottom
  element). Let $H,K$ be games over $A$ such that $H\leql K$. Then
  half of the conditions in the definition of $H\leq K$ are satisfied:
  \begin{enumerate}\alphalabels
  \item All right options $K^R$ satisfy $H\tri K^R$, and
  \item if $K$ is atomic, $H\tri K$.
  \end{enumerate}
\end{lemma}

\begin{proof}
  By assumption, $H\leql K$. To show (a), consider any right option
  $K^R$ of $K$. Then $K\tri K^R$, and therefore by
  Lemma~\ref{lem-102}, we have $H\tri K^R$, as desired. To show (b),
  assume $K$ is atomic. Then $K\tri K$, and again by
  Lemma~\ref{lem-102}, we have $H\tri K$, as desired.
\end{proof}

The converse implication of Lemma~\ref{lem-103} is not true. For
example, let $H=\g{\top|a}$ and $K=a$. Then conditions (a) and (b) are
satisfied, but $H\not\leql K$.

\begin{lemma}\label{lem-leq-leql-leqr}
  We have $H\leq K$ if and only if both $H\leql K$ and $H\leqr K$ hold.
\end{lemma}

\begin{proof}
  The left-to-right implication holds by
  Lemma~\ref{lem-composite-monotone}. For the right-to-left
  implication, assume $H\leql K$ and $H\leqr K$.  By
  Lemma~\ref{lem-103} applied to $H\leql K$, half of the conditions
  for $H\leq K$ hold. By the dual of Lemma~\ref{lem-103} applied to
  $H\leqr K$, the other half holds, therefore $H\leq K$.
\end{proof}

\begin{corollary}
  On the class of games over a given atom poset, the relations $\leq$,
  $\leql$, and $\leqr$ are all uniquely determined by the relation
  $\tri$ (i.e., without requiring knowledge of the internal structure
  of games).
\end{corollary}

\begin{proof}
  By Lemma~\ref{lem-102}, the relation $\leql$ is uniquely determined
  by $\tri$; the relation $\leqr$ is determined dually, and finally
  $\leq$ is determined by Lemma~\ref{lem-leq-leql-leqr}.
\end{proof}

\begin{definition}
  Let $A$ be a poset with top and bottom elements. Let $S$ be a
  non-empty set of games over $A$. We define $\upl S = \g{\top | \g{S |
      \bot}}$. Dually, define $\downr S = \g{\g{\top | S} | \bot}$.
\end{definition}

\begin{lemma}\label{lem-104}
  We have $S\leql T$ if and only if $S \leq \upl T$.
\end{lemma}

\begin{proof}
  For the right-to-left implication, note that $S \leq \g{\top | \g{ T
      | \bot}}$ implies $S \tri \g{T | \bot}$ by the definition of
  $\leq$. Hence $S\leql T$ by Lemma~\ref{lem-101}. For the
  left-to-right implication, assume $S\leql T$, and consider some
  $H\in S$. To show $H \leq \g{\top | \g{T | \bot}}$, we must show
  three things. First, clearly any left option $H^L$ satisfies
  $H^L\leq\top$ by Lemma~\ref{lem-top}. Second, we have $H\tri \g{T |
    \bot}$ by Lemma~\ref{lem-101}. Third, if $H$ is atomic, we must
  show that $H \tri \g{\top | \g{ T | \bot}}$. But this is true for
  all $H$ (atomic or not), because $H\leq\top$.
\end{proof}

\begin{corollary}
  We have $S\leql T$ if and only if $\upl S\leq \upl T$.
\end{corollary}

\begin{proof}
  Left-to-right: $S\leql T$ implies, by definition,
  $\g{S|\bot}\leq\g{T|\bot}$, which implies $\upl S\leq\upl T$ by
  Lemma~\ref{lem-composite-monotone}. Right-to-left: Assume $\upl
  S\leq \upl T$. From $S\leql S$, we get $S\leq\upl S$ by
  Lemma~\ref{lem-104}, hence $S\leq\upl T$ by transitivity, hence
  $S\leql T$ by Lemma~\ref{lem-104}.
\end{proof}

We remark that if the atom poset has top and bottom elements, both the
left order $\leql$ and the right order $\leqr$ admit least upper
bounds (and by duality, greatest lower bounds). Let $H\vee K = \g{H^L,
  K^L | \g{H, K | \bot}}$ denote the least upper bound of $H$ and $K$
with respect to the order $\leq$, as in
Proposition~\ref{prop-lattice}.  Then using Lemma~\ref{lem-104}, we
easily see that $H\vee K$ is also a least upper bound for
$\leql$. Namely, for any $G$, we have $H,K\leql G$ if and only if
$H,K\leq\upl G$ if and only if $H\vee K\leq\upl G$ if and only if
$H\vee K\leql G$. With respect to the right order, the least upper
bound of $H$ and $K$ is not $H\vee K$, but $\downr H\vee \downr K$. To
see why, first note that $\downr(\downr H\vee \downr K)\eq \downr
H\vee \downr K$. Then by the dual of Lemma~\ref{lem-104}, for any $G$,
we have $H,K\leqr G$ if and only if $\downr H,\downr K\leq G$ if and
only if $\downr H\vee \downr K\leq G$ if and only if $\downr(\downr
H\vee \downr K)\leq G$ if and only if $\downr H\vee \downr K\leqr
G$. The same argument of course also applies to least upper bounds of
more than two games.

\subsection{Left and Right Equivalence}\label{ssec-lr-equiv}

\begin{definition}
  Let $S$ and $T$ be non-empty sets of games. We say that $S,T$ are
  {\em left equivalent}, in symbols $S\eql T$, if $S\leql T$ and
  $T\leql S$. In other words, $S$ and $T$ are left equivalent if
  $\g{S,L|R}\eq\g{T,L|R}$ holds whenever $L$ and $R$ are a set of
  games and a non-empty set of games, respectively. Dually, we say
  that $S,T$ are {\em right equivalent}, in symbols $S\eqr T$, if
  $S\leqr T$ and $T\leqr S$.
\end{definition}

\begin{lemma}\label{lem-s-upl-s}
  Let $S$ be a set of games over an atom poset with top and
  bottom. Then $S$ is left equivalent to $\upl S$.
\end{lemma}

\begin{proof}
  Applying the left-to-right direction of Lemma~\ref{lem-104} to
  $S\leql S$, we get $S\leq \upl S$, which implies $S\leql\upl
  S$. Conversely, applying the right-to-left direction of
  Lemma~\ref{lem-104} to $\upl S\leq \upl S$, we get $\upl S\leql S$.
  Thus, we have $S\eql\upl S$, as claimed.
\end{proof}

\begin{corollary}\label{cor-upl-s}
  The game $\upl S$ is the maximal element (with respect to the order
  $\leq$) in the left equivalence class of $S$.
\end{corollary}

\begin{proof}
  By Lemma~\ref{lem-s-upl-s}, $\upl S$ is in the left equivalence
  class of $S$; if $T$ is any other member of this left equivalence
  class, then $T\eql S$, hence $T\leql S$, hence $T\leq\upl S$ by
  Lemma~\ref{lem-104}.
\end{proof}

\begin{lemma}\label{lem-eq-eql-eqr}
  We have $H\eq K$ if and only if both $H\eql K$ and $H\eqr K$ hold.
\end{lemma}

\begin{proof}
  This follows immediately from Lemma~\ref{lem-leq-leql-leqr}.
\end{proof}

The following lemma gives a remarkable equivalence between the
collection of all left equivalence classes and the right equivalence
class of $\top$.

\begin{lemma}\label{lem-right-eq-top}
  Let $A$ be a poset with top and bottom. Then for games over $A$,
  each left equivalence class has a unique maximal element (up to
  equivalence), and moreover, these maximal elements are precisely the
  members of the right equivalence class of $\top$.
\end{lemma}

\begin{proof}
  Each left equivalence class has a unique maximal element by
  Corollary~\ref{cor-upl-s}. Since this maximal element is of the form
  $\g{\top|G}$, it is reversible when used as a right option, and
  therefore easily seen to be right equivalent to $\top$. Conversely,
  let $G$ be any game that is right equivalent to $\top$. Then
  $G\eql\upl G$ (by Lemma~\ref{lem-s-upl-s}) and $G\eqr \upl G$
  (because they are both right equivalent to $\top$). Hence by
  Lemma~\ref{lem-eq-eql-eqr}, $G\eq\upl G$, i.e., $G$ is the maximal
  element of its left equivalence class, as claimed.
\end{proof}

For a visualization of Lemmas~\ref{lem-eq-eql-eqr} and
{\ref{lem-right-eq-top}}, see Figures~\ref{fig-l3} and {\ref{fig-l4}}
in Section~\ref{sec-enumeration} below.

\section{Passable Games and the Fundamental Theorem}\label{sec-fundamental}

\subsection{Passable Games}

Recall from Lemma~\ref{lem-tri-reflexive} that every monotone game
satisfies $G\tri G$. The converse is clearly not true; for example,
the game $G'=\g{\g{\top|a},b|\bot}$ from
Example~\ref{exa-canonical-not-monotone} satisfies $G'\tri G'$
(because $G'$ is equivalent to a monotone game), but it is not itself
monotone. It turns out that games satisfying $G\tri G$ are an
important class of games, which we will investigate in this section,
culminating in the fundamental theorem of monotone games.

We start with an obvious observation. Recall from
Definition~\ref{def-good} that a left option $G^L$ is called {\em
  good} if $G\leq G^L$, and dually, a right option $G^R$ is {\em good}
if $G^R\leq G$.

\begin{proposition}
  \label{prop-passable}
  Consider a game $G$ over some atom poset $A$. The following are
  equivalent:
  \begin{enumerate}\alphalabels
  \item $G\tri G$.
  \item $G$ is atomic or has at least one good (left or right) option.
  \item $G$ is a left gift horse for itself.
  \item $G$ is a right gift horse for itself.
  \item If Left had the option to pass in the position $G$, it would
    not benefit Left.
  \item If Right had the option to pass in the position $G$, it would
    not benefit Right.
  \end{enumerate}
\end{proposition}

\begin{proof}
  The equivalence of (a) and (b) follows directly from the definition
  of $\tri$. Indeed, by definition, for $G\tri G$ to be satisfied, one
  of three conditions must hold: either $G^R\leq G$ (in which case $G$
  has a good right option), or $G\leq G^L$ (in which case $G$ has a
  good left option), or $G$ is atomic. That is precisely what (b) is
  stating.  The equivalence of (a), (c), and (d) is obvious, since
  this is the definition of a gift horse. To see why (c) and (e) are
  equivalent, assume that $G=\g{G^L|G^R}$ is composite (this is
  without loss of generality by Lemma~\ref{lem-atomic-composite}).
  Consider the game $G'=\g{G,G^L|G^R}$. This game has the same moves
  as $G$, except that Left also has the option to pass (i.e., move to
  $G$).  Clearly $G\leq G'$ by Lemma~\ref{lem-more-options}.  To say
  that ``passing does not benefit Left'' means that $G\eq G'$, which
  holds, by the Gift Horse Lemma~\ref{lem-gift-horse}, if and only if
  $G$ is a left gift horse for itself. The equivalence between (d) and
  (f) is dual.
\end{proof}

This motivates the following definition.

\begin{definition}
  A game $G$ is {\em locally passable} if $G\tri G$. A game $G$ is
  {\em passable} if it is locally passable and recursively, all of its
  options are passable.
\end{definition}

The term ``passable'' should not be understood to mean that players
are allowed to pass in these games. Indeed, none of our games
literally allow passing, as that would potentially lead to infinite
plays. Rather, it means that players {\em do not want to pass} in
these games, and therefore, it does not matter whether passing is
allowed or not. Alternatively, we can understand a passable game to be
a game in which passing is ``almost'' allowed, in the sense that such
a game $G=\g{G^L|G^R}$ is {\em equivalent} to a game $\g{G,G^L|G^R,G}$
that has $G$ itself as a left option and a right option.

\begin{lemma}[Canonical form of passable game]
  \label{lem-passable-canonical}
  The canonical form of a passable game is passable.
\end{lemma}

\begin{proof}
  First, note that the class of locally passable games is closed under
  equivalence; indeed, if $G\tri G$ and $G'\eq G$, then $G'\tri G'$;
  this follows from Lemma~\ref{lem-transitive}. Now assume $G$ is some
  passable game and $G'$ is its canonical form. Because of the way
  canonical forms can be computed by repeatedly removing dominated
  options, simplifying passing options, and bypassing reversible
  options, every position occurring in $G'$ (including $G'$ itself) is
  equivalent to some position occurring in $G$. Therefore, every
  position in $G'$ is locally passable, and it follows that $G'$ is
  passable.
\end{proof}

We also note that the class of passable games is closed under least
upper bounds (and dually, under greatest lower bounds).

\begin{lemma}\label{lem-passable-lattice}
  If the atom poset has a bottom element, the least upper bound of any
  set of passable games is equivalent to a passable game.
\end{lemma}

\begin{proof}
  Note that the concept of a least upper bound is only well-defined up
  to equivalence, hence the lemma states ``equivalent to a passable
  game''. We will show that the particular least upper bound that was
  constructed in the proof of Proposition~\ref{prop-lattice} is passable.
  
  For simplicity, we show this in the case of two games, but the same
  proof applies to arbitrary non-empty sets of games. Also, the least
  upper bound of the empty set of games is of course $\bot$ and is
  passable. So consider passable games $H,K$.  As in
  Proposition~\ref{prop-lattice}, we may assume without loss of
  generality that $H,K$ are composite. Then the least upper bound is
  $G = \g{H^L, K^L | \g{H, K | \bot}}$. The options of $G$ are easily
  seen to be passable, so it suffices to show that $G$ is locally
  passable. From the upper bound property, we have $H,K\leq G$. Since
  $H$ and $K$ are passable, we have $H\tri H$ and $K\tri K$.  By
  transitivity, we get $H,K\tri G$. This implies $\g{H,K|\bot}\leq
  G$. Since $\g{H,K|\bot}$ is a right option of $G$, this implies
  $G\tri G$ as desired.
\end{proof}

\subsection{The Fundamental Theorem of Monotone Games}

By Lemma~\ref{lem-tri-reflexive}, we know that every monotone game is
passable. Perhaps surprisingly, if $A$ has top and bottom elements,
the converse is true up to equivalence of games. The purpose of the
rest of this section is to prove the following theorem.

\begin{theorem}[Fundamental theorem of monotone games]
  \label{thm-fundamental}
  Assume $A$ has top and bottom elements. Then every passable game
  over $A$ is equivalent to a monotone game over $A$.
\end{theorem}

Together with Lemma~\ref{lem-passable-canonical}, we immediately get
the following characterization of the canonical form of monotone games:

\begin{corollary}
  Assume the atom poset has top and bottom elements and $G$ has a
  canonical form. Then $G$ is equivalent to a monotone game if and
  only if the canonical form of $G$ is passable.
\end{corollary}

The corollary gives an easy method for enumerating equivalence classes
of short monotone games: we can simply enumerate all short passable
games in canonical form.

\subsection{Proof of the Fundamental Theorem}

Recall from Definition~\ref{def-good} and
Proposition~\ref{prop-passable} that a composite game is:
\begin{itemize}
\item {\em locally monotone} if all of its left and right options are
  good, and
\item {\em locally passable} if it has at least one good left or right
  option.
\end{itemize}
To prove the fundamental theorem, we introduce an intermediate notion:
we say that a composite game is
\begin{itemize}
\item {\em locally semi-monotone} if it has at least one good left
  option and at least one good right option.
\end{itemize}
By convention, atomic games are locally semi-monotone as well.
Naturally, we then say that a game is {\em semi-monotone} if it
locally semi-monotone and recursively, all of its options are
semi-monotone. Our strategy for proving the fundamental theorem is: we
will first show that every passable game is equivalent to a
semi-monotone game, and then that every semi-monotone game is
equivalent to a monotone game. Throughout this section, we assume that
all games are over a fixed poset $A$ of atoms with top and bottom
elements.

\begin{lemma}\label{lem-upl-monotone}
  Let $S$ be a set of monotone (respectively semi-monotone, passable)
  games and assume that $S$ forms a $\tri$-clique in the sense that
  $H\tri K$ for all $H,K\in S$. Then $\g{\top|S}$, $\g{S|\bot}$, $\upl
  S$, and $\downr S$ are monotone (respectively semi-monotone,
  passable).  In particular, if $H$ is a monotone game, then so are
  $\g{\top|H}$, $\g{H|\bot}$, $\upl H$, and $\downr H$.
\end{lemma}

\begin{proof}
  Let $G=\g{\top|S}$. Since all options of $G$ are monotone
  (resp.\@ semi-monotone, passable) by assumption, it suffices to show
  that $G$ is locally monotone. Clearly, we have $G\leq\top$ by
  Lemma~\ref{lem-top}. So all that is left to show for monotonicity is
  that $H\leq G$ for all $H\in S$. So consider some $H\in S$. To show
  $H\leq G$, first consider any left option $H^L$. We must show
  $H^L\tri G$, but this holds because $H^L\leq G^L=\top$. Second,
  consider any right option $G^R$. We must show $H\tri G^R$. But both
  $H,G^R$ are members of $S$, so the claim follows from the assumption
  that $S$ is a $\tri$-clique. Finally, in case $H$ is atomic, we must
  show $H\tri G$. But this is once again obvious since $\top$ is a
  left option of $G$. Therefore we have proved that $\g{\top|S}$ is
  monotone (resp.\@ semi-monotone, passable).
  
  The remaining claims of the lemma are easy: the claim for
  $\g{S|\bot}$ holds by duality, and the claims for $\upl
  S=\g{\top|\g{S|\bot}}$ and $\downr S$ follow by applying the
  previous claims twice. The claims in the lemma's final sentence
  follow because if $H$ is a single monotone game, it is automatically
  a $\tri$-clique.
\end{proof}

\begin{lemma}\label{lem-passable-is-semi-monotone}
  Every passable game is equivalent to a semi-monotone game.
\end{lemma}

\begin{proof}
  We prove this by induction. Suppose $G$ is a passable game. If $G$
  is atomic, then $G$ is semi-monotone by definition, and there is
  nothing to show. If $G$ is composite, then each of its left and
  right options is equivalent to a semi-monotone game by the induction
  hypothesis; therefore, $G$ is equivalent to some composite game $G'$
  all of whose options are semi-monotone. Since $G'$ is passable, $G'$
  either has a good left option or a good right option. We assume
  without loss of generality that $G'$ has a good left option $H$ (the
  case where $G'$ has a good right option is dual).

  Because $G'$ may not already have a good right option, we will add
  one. We are therefore looking for a game $K$ that is (1) a right
  gift horse for $G'$, so that adding it as a right option to $G'$
  will not change the value of $G'$, (2) semi-monotone, and (3) good,
  i.e., $K\leq G'$. If $K$ satisfies all three conditions, then
  $G''=\g{G'^L|G'^R,K}$ is semi-monotone and equivalent to $G'$, hence
  to $G$. So the only thing left to do is to find a game $K$
  satisfying the above three conditions.

  We claim that $K=\g{H|\bot}$ is the desired game. First, since $H$
  is a good left option of $G'$, we have $G'\leq H$ and therefore
  $G'\tri\g{H|\bot}=K$, so that $K$ is a right gift horse for $G'$.
  Second, $K$ is semi-monotone by Lemma~\ref{lem-upl-monotone}. Third,
  we have $K\leq G'$ since the only left option of $K$ is also a left
  option of $G'$, and $\bot$ is below any right option of $G'$.  So
  indeed, $K$ satisfies properties (1)--(3), which finishes the proof.
\end{proof}

\begin{lemma}\label{lem-semi-monotone-is-monotone}
  Every semi-monotone game is equivalent to a monotone game.
\end{lemma}

\begin{proof}
  We prove this by induction. Let $G$ be a semi-monotone game. If $G$
  is atomic, then $G$ is monotone and there is nothing to show. If $G$
  is composite, then each of its left and right options is equivalent
  to a monotone game by the induction hypothesis; therefore, $G$ is
  equivalent to some composite semi-monotone game $G'$ all of whose
  options are monotone.  Let $L$ and $R$ be the set of left and right
  options of $G'$, respectively. Define $L'=\s{\upl H \mid H \in L}$
  and $R'=\s{\downr K\mid K\in R}$. Let $G''=\g{\upl L'|\downr
    R'}$. We claim that $G''$ is the desired game, i.e., we claim that
  $G''\eq G$ and $G''$ is monotone.

  To show $G''\eq G$, recall that by Lemma~\ref{lem-s-upl-s} and its
  dual, each $\upl H$ is left equivalent to $H$ and each $\downr K$ is
  right equivalent to $K$. We therefore have $G'\eq \g{L'|R'}$. Using
  Lemma~\ref{lem-s-upl-s} a second time, we also know that $L'$ is
  left equivalent to $\upl L'$ and $R'$ is right equivalent to $\downr
  R'$. It follows that $G\eq G'\eq \g{\upl L'|\downr R'} = G''$.

  To show that $G''$ is monotone, first note that each $\upl H$ and
  each $\downr K$ is monotone by
  Lemma~\ref{lem-upl-monotone}. Moreover, since $\upl H\tri \upl H'$
  for all $H,H'$, the set $L'$, and dually $R'$, forms a
  $\tri$-clique. Using Lemma~\ref{lem-upl-monotone} a second time, it
  follows that $\upl L'$ and $\downr R'$ are monotone. So all the
  options of $G''$ are monotone. The only thing left to prove is that
  $G''$ is locally monotone, i.e., $\downr R' \leq G''\leq \upl L'$.

  To show $G''\leq\upl L'$, first consider any left option $G''^L$. We
  must show $G''^L\tri \upl L'$, but this is plainly true since
  $G''^L=\upl L'$ (and $\upl L'$ is monotone, hence passable).  Next,
  consider the only right option $\g{L'|\bot}$ of $\upl L'$. We must
  show $G''\tri \g{L'|\bot}$. Since $G'$ is semi-monotone, $G'$ has at
  least one good left option, say $H\in L$. By goodness, $G'\leq H$,
  which implies $G''\leq \upl H$ since $G''\eq G'$ and $H\leq \upl
  H$. Since $\upl H\in L'$, it follows that $G''\tri \g{L'|\bot}$ as
  claimed. Since neither $G''$ nor $\upl L'$ is atomic, this finishes
  the proof of $G''\leq\upl L'$.

  The proof of $\downr R' \leq G''$ is dual.
\end{proof}

\begin{proof}[Proof of Theorem~\ref{thm-fundamental}]
  Theorem~\ref{thm-fundamental} follows directly from
  Lemmas~\ref{lem-passable-is-semi-monotone} and
  {\ref{lem-semi-monotone-is-monotone}}.
\end{proof}

\begin{example}\label{exa-thm-fundamental}
  Consider the poset $A=\s{\bot,a,b,\top}$, with $a$ and $b$
  incomparable. Let $G=\g{a,b|a}$. One can easily check that $a\leq
  G$, and hence $G\tri G$, so $G$ is passable. However, $G$ is
  certainly not monotone; for example, we neither have $G\leq a$ nor
  $G\leq b$. By Theorem~\ref{thm-fundamental}, $G$ is equivalent to a
  monotone game. The particular game constructed in the proof of
  Theorem~\ref{thm-fundamental} is
  \[ \begin{array}{r@{~}c@{~}l}
    G'' &=& \g{\upl(\upl a, \upl b, \upl\g{\top|a}) | \downr\downr a}
    \\ &=& \{\uplx{\uplx a, \uplx b, \uplx{\g{\top|a}}}
    \\ && \qquad\qquad\qquad \mid \downrx{\downrx a}\}.
  \end{array}
  \]
  Indeed, one can check that $G''$ is monotone and equivalent to $G$.
  Note that $G''$ is not the simplest monotone game equivalent to
  $G$. We can obtain a simpler one by removing dominated options and
  bypassing only those reversible options for which bypassing does not
  break monotonicity. By doing so, we obtain the following simpler
  monotone game equivalent to $G$:
  \[ G''' = \g{\g{\top|\g{\g{\top|a},\g{\top|\g{b|\bot}}|\bot}}|a}.
  \]
  Although $G'''$ is (perhaps) the simplest monotone game equivalent
  to $G$, it is not in canonical form, since it still has some
  reversible options. The canonical form of $G'''$ is of course $G =
  \g{a,b|a}$.
\end{example}

\begin{remark}
  In the proof of Lemma~\ref{lem-semi-monotone-is-monotone}, we
  applied two levels of $\upl$ operations: first to each option of
  $G'$, and then to the sets of all left options and of all right
  options. For example, if $G'=\g{G_1,G_2|G_3,G_4}$, then
  $G''=\g{\upl(\upl G_1, \upl G_2)|\downr(\downr G_3, \downr
    G_4)}$. Both steps are necessary; in general, neither $\g{\upl
    G_1, \upl G_2|\downr G_3, \downr G_4}$ nor $\g{\upl (G_1,
    G_2)|\downr (G_3, G_4)}$ is monotone. Indeed, for the game from
  Example~\ref{exa-thm-fundamental}, one can check that neither
  $\g{\upl a, \upl b, \upl\g{\top|a} | \downr a}$ nor $\g{\upl(a, b,
    \g{\top|a}) | \downr a}$ is monotone, but $\g{\upl(\upl a, \upl b,
    \upl\g{\top|a}) | \downr\downr a}$ is.
\end{remark}

\subsection{A Useful Reasoning Principle for Passable Games}

We will prove the following property of passable games, which is
sometimes useful: if all left options satisfy $G^L\leq a$, then $G\leq
a$. Informally, this is true because if it is Left's turn, Left can
achieve at most outcome $a$. If it were possible for Left to do better
when it is Right's turn, then Right would prefer to pass,
contradicting the fact that Right does not prefer to pass in a
passable game.

The formal proof requires a lemma.

\begin{lemma}\label{lem-passable-a}
  Let $G$ be a passable composite game, $a$ an atom, and assume that
  all left options $G^L$ satisfy $G^L\leq a$. Then the following hold
  for all $H$:
  \begin{enumerate}
  \item[(a)] If $H\tri G$ then $H\tri a$.
  \item[(b)] If $H\leq G$ then $H\leq a$.
  \end{enumerate}
\end{lemma}

\begin{proof}
  We prove (a) and (b) by joint induction on $H$. In each inductive
  case, we prove (a) before (b). To prove (a), assume $H\tri G$. By
  definition of $\tri$, we either have $H\leq G^L$ for some $G^L$, or
  $H^R\leq G$ for some $H^R$. If $H\leq G^L$, we use the assumption
  $G^L\leq a$ to conclude $H\leq a$ by transitivity, and therefore
  also $H\tri a$ by definition of $\leq$. If $H^R\leq G$, then by
  the induction hypothesis (b) we have $H^R\leq a$, therefore $H\tri a$ as
  claimed.

  To prove (b), assume $H\leq G$. We must show $H\leq a$. The
  definition of $\leq$ requires us to prove two things: all $H^L\tri
  a$, and $H\tri a$.  To prove the first claim, consider any
  $H^L$. Since $H\leq G$, we have $H^L\tri G$, and therefore $H^L\tri
  a$ by the induction hypothesis (a). To prove the second claim, note that
  $G$ is passable, so $G\tri G$. From the assumption $H\leq G$, we get
  $H\tri G$ by Lemma~\ref{lem-transitive}. Therefore, by part (a),
  which we already proved for $H$, we have $H\tri a$ as claimed.
\end{proof}

\begin{proposition}\label{prop-passable-a}
  Let $G$ be a passable composite game, $a$ an atom, and assume all
  left options $G^L$ satisfy $G^L\leq a$. Then $G\leq a$.
\end{proposition}

\begin{proof}
  This follows by Lemma~\ref{lem-passable-a}(b), with $H=G$.
\end{proof}

Note: not only does Lemma~\ref{lem-passable-a} imply
Proposition~\ref{prop-passable-a}, but also the other way
round. However, since the lemma is proved by induction on $H$, and not
by induction on $G$, we cannot prove the proposition directly without
the lemma.

\section{Games over Linearly Ordered Sets of Atoms}\label{sec-linear}

As we have seen in Section~\ref{ssec-problem-monotonicity}, the
canonical form of a monotone game is not in general monotone; in fact,
this was the reason we introduced the more general class of passable
games. However, perhaps surprisingly, when the set $A$ of atoms
happens to be linearly ordered, the class of monotone games does turn
out to be closed under canonical forms. The purpose of this section is
to prove it. We start with two lemmas that state another useful
property of games over linearly ordered posets.

\begin{lemma}\label{lem-linear-atomic}
  Suppose $A$ is a linearly ordered poset, $a\in A$ is an atom, and
  $H$ is a passable game over $A$. Then we have:
  \begin{enumerate}\alphalabels
  \item $[a]\tri H$ or $H\leq [a]$.
  \item $H\tri [a]$ or $[a]\leq H$.
  \end{enumerate}
\end{lemma}

\begin{proof}
  We prove this by induction. We first prove (a). If $H=[b]$ is
  atomic, we have $[a]\tri H$ if and only if $a\leq b$ and $H\leq [a]$
  if and only if $b\leq a$, so the claim follows from the fact that
  $A$ is linearly ordered.

  If $H$ is not atomic, assume $[a]\ntri H$. To show $H\leq [a]$,
  first consider any left option $H^L$ of $H$. Since by assumption,
  $[a]\ntri H$, we have $[a]\nleq H^L$, so by the induction hypothesis,
  $H^L\tri [a]$. Therefore, the first part of the definition of $H\leq
  [a]$ is satisfied. The second part is trivially satisfied, since
  $[a]$ is atomic. For the third part, we must show $H\tri [a]$. Since
  $H$ is passable, $H$ either has a good left option or a good right
  option. If $H$ has a good left option $H^L$, then $H\leq H^L$. Since
  we already showed $H^L\tri [a]$ above, it follows by
  Lemma~\ref{lem-transitive} that $H\tri [a]$. On the other hand, if
  $H$ has a good right option $H^R$, then $H^R\leq H$. Since by
  assumption, $[a]\ntri H$, by Lemma~\ref{lem-transitive}, we have
  $[a]\ntri H^R$, therefore by the induction hypothesis, $H^R\leq [a]$,
  which implies $H\tri [a]$ as desired. It follows that $H\leq [a]$.

  The proof of (b) is dual.
\end{proof}

\begin{lemma}\label{lem-linear}
  Suppose $A$ is a linearly ordered poset. Then for all passable games
  $G,H$ over $A$, we have $G\tri H$ or $H\leq G$.
\end{lemma}

\begin{proof}
  We prove this by induction. If $G$ or $H$ is atomic, then the result
  holds by Lemma~\ref{lem-linear-atomic}. Therefore, assume both $G$
  and $H$ are composite. Suppose that $G\ntri H$. To show $H\leq G$,
  first consider an arbitrary left option $H^L$ of $H$. We must show
  $H^L\tri G$. From the definition of $G\ntri H$, we know that $G\nleq
  H^L$, therefore by the induction hypothesis, $H^L\tri G$ as desired.
  The dual argument shows that $H\tri G^R$ holds for all $G^R$. Since
  $G$ and $H$ are composite, this proves that $H\leq G$, as desired.
\end{proof}

\begin{remark}
  If the games are not passable and $A$ has at least 3 elements,
  Lemmas~\ref{lem-linear-atomic} and {\ref{lem-linear}} are not
  true. The simplest counterexample is $A=\s{\bot,a,\top}$, $G=[a]$,
  $H=\g{\bot|\top}$. We neither have $[a]\tri\g{\bot|\top}$ nor
  $\g{\bot|\top}\leq [a]$.

  If $A=\s{\bot,\top}$, then Lemmas~\ref{lem-linear-atomic} and
  {\ref{lem-linear}} are valid even if the games are not
  passable. Because in that case, Lemma~\ref{lem-linear-atomic} only
  has the two cases $a=\bot$ and $a=\top$, and both happen to be true
  because $\bot\tri H\leq\top $ and $\bot\leq H\tri\top$. Note that
  the proof of Lemma~\ref{lem-linear} goes through in that case as
  well.
\end{remark}

We now come to the main result of this section.

\begin{theorem}\label{thm-passable-linear-monotone}
  Suppose $A$ is a linearly ordered poset, that $G$ is a passable game
  over $A$, and that $G$ is in canonical form. Then $G$ is monotone.
\end{theorem}

\begin{proof}
  We prove this by induction. Since all options of $G$ are passable
  and in canonical form, they are monotone by the induction
  hypothesis. Therefore, all we have to show is that $G$ is locally
  monotone, i.e., all options of $G$ are good. We note that, since $G$
  is in canonical form, none of its options are dominated or
  reversible.  We prove a number of claims in turn.
  \begin{enumerate}
  \item[(a)] For any two left options $H,H'$ of $G$, we have $H\tri H'$.

    Proof: If $H=H'$, then the claim holds because $H$ is passable. If
    $H\neq H'$, then since $H'$ is not dominated, we have $H'\nleq
    H$. Then $H\tri H'$ holds by Lemma~\ref{lem-linear}.
  \item[(b)] For any left option $H$ of $G$ and any right option $H^R$
    of $H$, we have $G\tri H^R$.

    Proof: Since $H$ is not reversible, we must have $H^R\nleq
    G$. Then $G\tri H^R$ holds by Lemma~\ref{lem-linear}.

  \item[(c)] All composite left options of $G$ are good.

    Proof: Let $H$ be a left option of $G$ and assume that $H$ is
    composite. We must show $G\leq H$. From (a), we know that $G^L\tri
    H$ for all $G^L$, and from (b), we know that $G\tri H^R$ for all
    $H^R$. Since both $G$ and $H$ are composite, the third condition
    in the definition of $\leq$ does not apply, and we have $G\leq H$
    as claimed.

  \item[(d)] $G$ has at most one atomic left option.

    Proof: If $a,b$ are distinct atomic left options of $G$, then we
    must either have $a\leq b$ or $b\leq a$ since the poset $A$ is
    linearly ordered. But then either $a$ or $b$ would be dominated,
    contradicting the fact that $G$ is in canonical form.

  \item[(e)] All atomic left options of $G$ are good.

    Proof: Let $a$ be an atomic left option of $G$. We must show
    $G\leq a$. We distinguish two cases: either $G$ has some composite
    left option or it does not. In case $G$ has some composite left
    option $H$, then by (c), we have $G\leq H$. Since $a$ is not
    dominated, we have $a\nleq H$, and therefore $a\nleq G$. Then
    $G\tri a$ holds by Lemma~\ref{lem-linear}. The other condition in
    the definition of $G\leq a$ already holds by (a). Therefore $G\leq
    a$ as claimed. In case $G$ does not have any composite left
    options, then all left options are atomic. By (d), $a$ is the only
    left option of $G$. Then $G\leq a$ holds by
    Proposition~\ref{prop-passable-a}.
  \end{enumerate}
  Together, (c) and (e) imply that all left options of $G$ are good.
  The proof for right options is dual. This finishes the proof of the
  monotonicity of $G$.  
\end{proof}

\begin{corollary}
  For games over a linearly ordered poset $A$, the canonical form of a
  monotone game is monotone.
\end{corollary}

\section{Operations on Games}\label{sec-operations}

\subsection{Sum}\label{ssec-sum}

As already mentioned in Section~\ref{ssec-order}, in the standard
treatment of combinatorial game theory, it is common to define the
negation $-G$ and sum $G+H$ of games before defining the order,
because $G\leq H$ can then be conveniently defined to mean that $H-G$
is a second-player win for Left. We did not take this route here,
because in the games we consider in this paper, the sum of games works
a bit differently than usual. Indeed, as we will see, the sum is
well-behaved only for passable games, so it was necessary to define
passable games, and therefore the ordering on games, before we could
consider sums.

Unsurprisingly, the sum of games is defined in the same way as in
other branches of combinatorial game theory when the games are
composite. As usual, the difference lies in the treatment of
atoms. When $G$ and $H$ are atomic games with respective outcomes $a$
and $b$, we define their sum to be an atomic game whose outcome is the
pair $(a,b)$. This appropriately reflects what goes on when one plays
in multiple disconnected regions in a monotone set coloring game such
as Hex. This leads us to the following definition.

\begin{definition}[The sum of games]\label{def-sum}
  Suppose $G$ and $H$ are games over outcome posets $A$ and $B$,
  respectively. The {\em sum} of $G$ and $H$, written $G+H$, is a game
  over $A\times B$. It is defined as follows.
  \begin{itemize}
    \item $G+H = \g{G^L + H, G + H^L |G^R + H, G + H^R}$, when at
      least one of $G,H$ is composite, and
    \item $[a]+[b] = [(a,b)]$, when $G=[a]$ and $H=[b]$ are both atomic.
  \end{itemize}
\end{definition}

We note that in this definition, we have used our usual convention
that atomic games have no left and right options. Thus, in the cases
where one of $G,H$ is atomic and the other is composite, the
definition specializes to:
\begin{itemize}
\item $[a]+H = \g{[a] + H^L |[a] + H^R}$,
\item $G+[b] = \g{G^L + [b] |G^R + [b]}$.
\end{itemize}

Next, one would expect that we prove that the sum is a monotone
operation, i.e., that $G\leq G'$ and $H\leq H'$ imply $G+H\leq
G'+H'$. However, we will not prove this, because it is not true. In a
nutshell, we built into the definition of the order that
$a\eq\g{a|a}$, so that a player can always, up to equivalence, pass in
any atomic component of a larger game. This is justified when all
games are passable, because in this case, passing is not to any
player's advantage. But it is not justified for non-passable games,
where passing may actually be advantageous. Since we are only
interested in passable games, defining the order in this way is the
right thing to do. But the price to pay is that monotonicity of sum
does not hold for arbitrary games.

\begin{example}[Non-monotonicity of sum]\label{exa-sum-not-monotone}
  Let $G=[a]$ and $H=\g{\bot|\top}$. Also consider the game
  $G'=\g{a|a}$. Then
  \[
  \begin{array}{l@{~}l@{~}l@{}}
    G+H &=& \g{(a,\bot)|(a,\top)}, \\
    G'+H &=& \g{\g{(a,\bot)|(a,\bot)}, \g{(a,\bot)|(a,\top)}|
      \g{(a,\top)|(a,\top)},\g{(a,\bot)|(a,\top)}}.
  \end{array}
  \]
  If we consider the outcome $(a,\top)$ to be winning for Left and
  $(a,\bot)$ to be winning for Right, the game $G+H$ is a
  second-player win for both players, whereas $G'+H$ is a first-player
  win for both players. The example shows that although $G\leq G'$ and
  $G'\leq G$, we have $G+H\nleq G'+H$ and $G'+H\nleq G+H$.  In
  particular, the sum operation is not monotone on these games.
\end{example}

Since the game $H$ is not passable, Example~\ref{exa-sum-not-monotone}
is nothing to worry about. In Corollary~\ref{cor-sum-monotone}, we
will show that the sum operation is well-behaved on passable
games. The following lemma holds for all games (passable or not). The
proof is straightforward and we omit it.

\begin{lemma}\label{lem-pre-monotone-sum}
  Let $a,b\in A$ be atoms, and let $G,H$ by any games over $B$.
  \begin{enumerate}\alphalabels
  \item If $a\leq b$ and $G\tri H$, then $[a]+G\tri [b]+H$.
  \item If $a\leq b$ and $G\leq H$ then $[a]+G\leq [b]+H$.
  \end{enumerate}
  The symmetric properties, about $G+[a]$ and $H+[b]$, also hold.
\end{lemma}

\begin{proposition}[Monotonicity of sum on passable games]\label{prop-sum-monotone}
  Let $G,G'$ be games over a poset $A$, and let $H$ be a passable game
  over a poset $B$. Then:
  \begin{enumerate}\alphalabels
  \item $G\tri G'$ implies $G+H\tri G'+H$.
  \item $G\leq G'$ implies $G+H\leq G'+H$.
  \end{enumerate}
\end{proposition}

\begin{proof}
  We prove (a) and (b) by joint induction. In each inductive case, we
  prove (a) before (b). Assume the proposition is true for all smaller
  triples of games $(G,G',H)$.

  To prove (a), assume $G\tri G'$. We must show $G+H\tri G'+H$. By the
  definition of $G\tri G'$, there are three cases.  Case 1: $G^R\leq
  G'$ for some right option $G^R$ of $G$.  By the induction
  hypothesis, $G^R+H\leq G'+H$. Since $G^R+H$ is a right option of
  $G+H$, it follows that $G+H\tri G'+H$, as desired.  Case 2: $G\leq
  G'^L$ for some left option $G'^L$ of $G'$. This case is analogous.
  Case 3: $G=[a]$ and $G'=[a']$ are atomic and $a\leq a'$.  Here we
  use the assumption that $H$ is passable (and this is the only place
  in the proof where it is used, apart from inductively using the fact
  that the options of $H$ are passable). Since $H$ is passable, we
  have $H\tri H$, and therefore by Lemma~\ref{lem-pre-monotone-sum},
  we have $[a]+H\tri [a']+H$, i.e., $G+H\tri G'+H$, as claimed.
    
  To prove (b), assume $G\leq G'$. To show $G+H\leq G'+H$, first
  consider any left option $K$ of $G+H$. We must show $K\tri G'+H$.
  There are two cases, depending on what kind of left option $K$ is.
  Case 1: $K=G^L+H$ for some left option $G^L$ of $G$. From $G\leq
  G'$, we get $G^L\tri G'$, and therefore by the induction hypothesis,
  $G^L+H\tri G'+H$ as desired.  Case 2: $K=G+H^L$ for some left option
  $H^L$ of $H$. By the induction hypothesis, since $H^L$ is passable,
  $G+H^L\leq G'+H^L$. Since $G'+H^L$ is a left option of $G'+H$, by
  definition of $\tri$, we have $G+H^L\tri G'+H$, as desired.  This
  proves the first property required for $G+H\leq G'+H$.  Second, we
  must show that $G+H\tri J$ for every right option $J$ of $G'+H$.  This
  case is analogous to the previous one. The remaining thing to show
  is that if $G+H$ or $G'+H$ is atomic, then $G+H\tri G'+H$. But by
  definition of $+$, it follows that $G$ or $G'$ is atomic. Since we
  assumed $G\leq G'$, it follows that $G\tri G'$. Then by (a), we have
  $G+H\tri G'+H$, as desired.
\end{proof}

Of course, the symmetric version of
Proposition~\ref{prop-sum-monotone} also holds, i.e., if $G$ is
passable, then $H\tri H'$ implies $G+H\tri G+H'$ and $H\leq H'$
implies $G+H\leq G+H'$. Using transitivity, we thus obtain the
monotonicity properties of the sum of passable games that are
summarized in the following corollary.

\begin{corollary}\label{cor-sum-monotone}
  Let $G,G'$ be passable games over $A$, and let $H,H'$ be passable
  games over $B$. Then
  \begin{enumerate}\alphalabels
  \item $G\tri G'$ and $H\leq H'$ imply $G+H\tri G'+H'$.
  \item $G\leq G'$ and $H\tri H'$ imply $G+H\tri G'+H'$.
  \item $G\leq G'$ and $H\leq H'$ imply $G+H\leq G'+H'$.
  \end{enumerate}
  In particular, on passable games, the sum is well-defined up to
  equivalence: if $G\eq G'$ and $H\eq H'$, then $G+H\eq G'+H'$.
\end{corollary}

Note that Corollary~\ref{cor-sum-monotone} also implies that the
sum of passable games is passable. The sum also enjoys other expected
properties: up to obvious isomorphisms of atom sets, the sum is
symmetric and associative, and it has a unit, which is the unique game
over one atom. These properties are straightforward and we omit the
proofs.

From here on for the rest of the paper, all games will be assumed to
be passable. We will usually state this explicitly (for example for
the benefit of readers who skipped the current paragraph), but should
it ever not be stated, the games are assumed to be passable anyway.

\subsection{The Opposite Game}

In standard combinatorial game theory, the negation of a game $G$,
usually written $-G$, is the game obtained by exchanging the roles of
the players. We have a similar operation here, but in addition to
exchanging the roles of the players, we must also invert the outcome
poset. To emphasize that this operation is not an additive inverse for
the sum operation, we call it the opposite game, rather than the
negation, and we denote it by $G\opp$.

If $A$ is a poset, let $A\opp$ denote the opposite poset, which has
the same elements as $A$, but with the opposite order, i.e.,
$a\leq_{A\opp} b$ if and only if $b\leq_{A} a$.

\begin{definition}[Opposite game]
  Let $G$ be a game over a poset $A$. The {\em opposite game} of $G$,
  denoted by $G\opp$, is a game over $A\opp$. It is obtained by
  exchanging the roles of Left and Right. More formally, we define
  $[a]\opp = [a]$ and $\g{G^L|G^R}\opp =
  \g{(G^R)\opp|(G^L)\opp}$.
\end{definition}

Note that $G$ is monotone (respectively, passable) if and only if
$G\opp$ is monotone (respectively, passable). All notions dualize; for
example, we have $G\leq H$ if and only if $H\opp\leq G\opp$, $G\eql H$
if and only if $G\opp\eqr H\opp$, and so on.

\subsection{The Map Operation}\label{ssec-map}

We can use a monotone function $f:A\to B$ to turn a game over $A$ into
a game over $B$.

\begin{definition}[Map operation]
  Let $A,B$ be posets and $f:A\to B$ a monotone function. Given a game
  $G$ over $A$, we define a game $f(G)$ over $B$ as follows:
  \begin{itemize}
  \item $f([a]) = [f(a)]$ for atomic games;
  \item $f(G) = \g{f(G^L) |f(G^R)}$ when $G$ is composite.
  \end{itemize}
  We say that the function $f$ {\em maps} the game $G$ to the game
  $f(G)$. We will often apply maps to sums of games, in which case we
  write $G+_fH$ instead of $f(G+H)$.
\end{definition}

The map operation satisfies some obvious properties; for example, it
is functorial: if $f:A\to B$ and $g:B\to C$, we have $g(f(G)) =
(g\circ f)(G)$, and if $\id:A\to A$ is the identity function, we have
$\id(G) = G$. Also, if $f:A\times A\to A$ is an associative operation,
then $(G+_f H)+_f K = G+_f(H+_f K)$. Finally, as expected, the map
operation is monotone in both arguments, i.e., $f\leq f'$ and $G\leq
G'$ imply $f(G)\leq f'(G')$, and similarly for $\tri$. Also, if $G$ is
passable, then so is $f(G)$.

\begin{remark}
  Any monotone set coloring game $(X,\pi)$ is actually $\pi(G)$, where
  $G$ is a direct sum of $|X|$ copies of the game $\g{\top|\bot}$, and
  $\pi:\Bool^X\to A$ is the payoff function.
\end{remark}

\subsection{Copy-Cat Strategies}\label{ssec-copy-cat}

One difference between the sum operation in our games and in other
kinds of combinatorial games is that there are actually many different
sum operations, depending on what winning condition we want to impose
on the combined game. In a sense, each different monotone function
$f:A\times B\to C$ defines a different sum operation $G\pf H$. Many of
these are not self-dual. As a simple example, consider the boolean
functions ``and'' and ``or''${}:\Bool\times\Bool\to\Bool$. If $G,H$
are games over $\Bool$, then to win the game $G+_{\cand}H$, Left must
win both components, but Right only needs to win one component. The
game $G+_{\cor}H$ is its dual: here Left only needs to win one
component, but Right must win both components.  In Hex, such games
might look like this:
\[
\begin{array}{c@{\hspace{1cm}}c}
  \begin{hexboard}[scale=0.8]
    \shadows
    \board(3,7)
    \foreach\i in {1,...,3} {\black(\i,4)}
    \cell(2,2)\label{$G$}
    \cell(2,6)\label{$H$}
  \end{hexboard}
  &
  \begin{hexboard}[scale=0.8]
    \shadows
    \board(7,3)
    \foreach\i in {1,...,3} {\white(4,\i)}
    \cell(2,2)\label{$G$}
    \cell(6,2)\label{$H$}
  \end{hexboard}
  \\
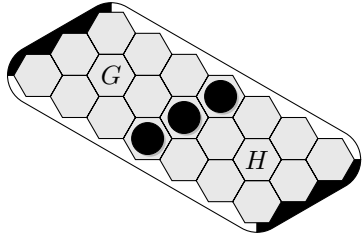
  \mbox{Black must win both components.}
  &
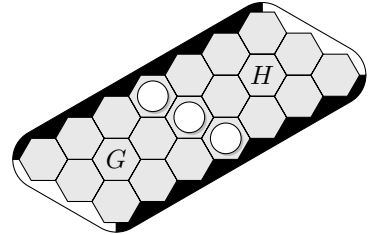
  \mbox{Black must win one component.}
\end{array}
\]
Games like these were called {\em conjunction} and {\em disjunction}
games in {\cite{VanRijswijck}}.

We can generalize this idea to an arbitrary outcome poset. Given a
poset $A$, define the monotone functions $\lambda,\rho:A\times
A\opp\to\Bool$ as follows:
\[
\lambda(a,b) = \begin{choices}
  \top & \mbox{if $b\leq_A a$,} \\
  \bot & \mbox{otherwise.}
\end{choices}
\qquad\qquad
\rho(a,b) = \begin{choices}
  \top & \mbox{if $a\nleq_A b$,} \\
  \bot & \mbox{otherwise.}
\end{choices}
\]
These functions are each other's duals, in the sense that $(G\pl
H)\opp = H\opp\pr G\opp$. They are called ``\tu{l}ambda'' and
``\tu{r}ho'' because they favor the \tu{L}eft and \tu{R}ight player,
respectively. We have the following properties:

\begin{lemma}\label{lem-copy-cat}
  For all passable games $G,H$:
  \begin{enumerate}\alphalabels
  \item $G\pl G\opp \eq \top$ is a second-player win for Left.
  \item $G\pr G\opp \eq \bot$ is a second-player win for Right.
  \item $G\leq H$ $\iff$ $\top\leq H\pl G\opp$ $\iff$ $G\pr H\opp\leq\bot$.
  \item $G\tri H$ $\iff$ $\top\tri H\pl G\opp$ $\iff$ $G\pr H\opp\tri\bot$.
  \end{enumerate}
\end{lemma}

\noindent
We omit the proofs, which are straightforward inductions on $G$ and
$H$, and basically the same as the usual proofs of $G\leq H$ if and
only if $0\leq H+(-G)$ and related properties in normal play games. In
particular, the second-player winning strategies employed by Left in
(a) and Right in (b) are the copy-cat strategies, which consist of
always copying the other player's last move in the opposite component
of the game.

\section{Contextual Order and Global Decisiveness}\label{sec-contextual}

\subsection{Contexual Order}\label{ssec-contextual}

As motivated in Section~\ref{sec-outcomes}, our games over an outcome
poset $A$ are intended to represent ``local'' play, such as play in a
particular region of a Hex board. This local play is a component of a
larger ``global'' game. In the global game, success or failure is
measured by winning or losing, i.e., global games are always over the
outcome poset $\Bool=\s{\bot,\top}$. Now that we have introduced sums
and the map operation, we can briefly describe how the local and
global notions hang together. This is not unlike what happens in other
branches of combinatorial game theory, but appropriately adjusted to
accommodate games over outcome posets.

\begin{definition}[Context]
  Let $A$ be a poset. A {\em context} for $A$ is a triple $(B,f,K)$,
  where $B$ is a poset, $f:A\times B\to\Bool$ is a monotone function,
  and $K$ is a passable game over $B$.
\end{definition}

\begin{definition}[Contextual order and equivalence]\label{def-contextual}
  Fix a poset $A$. We define several different {\em contextual order}
  relations on games over $A$, as follows. For passable games $G,H$
  over $A$, we say:
  \begin{itemize}
  \item $G\leqcI H$ if for all contexts $(B,f,K)$, if Left has a
    first-player winning strategy in $G\pf K$, then Left has a
    first-player winning strategy in $H\pf K$.
  \item $G\leqcII H$ if for all contexts $(B,f,K)$, if Left has a
    second-player winning strategy in $G\pf K$, then Left has a
    second-player winning strategy in $H\pf K$.
  \item $G\leqcIII H$ if for all contexts $(B,f,K)$, if Left has a
    second-player winning strategy in $G\pf K$, then Left has a
    first-player winning strategy in $H\pf K$.
  \end{itemize}
  We also write $G\leqc H$ if $G\leqcI H$ and $G\leqcII H$, and $G\eqc
  H$ if $G\leqc H$ and $H\leqc G$. The latter relation is called {\em
    contextual equivalence}.
\end{definition}

We note that up to equivalence, there are only three passable games
over $\Bool$: $\top$, which is a first- and second-player win for
Left; $\Star$, which is a first-player win but not a second-player win
for Left; and $\bot$, which is neither a first- nor a second-player
win for Left. Thus, Left has a first-player winning strategy in some
game $X$ over $\Bool$ if and only if $\Star\leq X$, or equivalently,
$\top\tri X$, and Left has a second-player winning strategy if and
only if $\top\leq X$.

The following proposition shows that the contextual order relations
coincide with the relations we defined in Section~\ref{ssec-order}.

\begin{proposition}\label{prop-context}
  For passable games $G,H$ over $A$, we have:
  \begin{enumerate}\alphalabels
  \item $G\leqcI H$ $\iff$ $G\leq H$.
  \item $G\leqcII H$ $\iff$ $G\leq H$.
  \item $G\leqcIII H$ $\iff$ $G\tri H$.
  \end{enumerate}
\end{proposition}

\begin{proof}
  In all three cases, the right-to-left implication follows from the
  monotonicity properties of the sum and map operations. For example,
  $G\leq H$ implies $G\pf K\leq H\pf K$ for all contexts. We must show
  the left-to-right implications. For (a), assume $G\leqcI H$.
  Consider the context $(A\opp,\rho,H\opp)$. We know from
  Lemma~\ref{lem-copy-cat}(b) that Right has a second-player win in
  the game $H\pr H\opp$; therefore Left does not have a first-player
  win in that game. Since $G\leqcI H$, it follows that Left does not
  have first-player win in $G\pr H\opp$, so we must have $G\pr
  H\opp\eq \bot$. Then by Lemma~\ref{lem-copy-cat}(c), we get $G\leq
  H$ as claimed. The proof of (b) is dual. Here we assume $G\leqcII H$
  and use the context $(A\opp,\lambda,G\opp)$. By
  Lemma~\ref{lem-copy-cat}(a), Left has a second-player winning
  strategy in $G\pl G\opp$, and since $G\leqcII H$, Left also has a
  second-player winning strategy in $H\pl G\opp$, so $H\pl
  G\opp\eq\top$, so $G\leq H$ by Lemma~\ref{lem-copy-cat}(c).  The
  proof of (c) is also very similar. Here, we assume $G\leqcIII H$ and
  use the context $(A\opp,\lambda,G\opp)$. By
  Lemma~\ref{lem-copy-cat}(a), Left has a second-player winning
  strategy in $G\pl G\opp$. Therefore, since $G\leqcIII H$, Left has a
  first-player winning strategy in $H\pl G\opp$. This means that
  $\top\tri H\pl G\opp$. Thus, by Lemma~\ref{lem-copy-cat}(d), $G\tri
  H$.
\end{proof}

Proposition~\ref{prop-context} is perhaps not very surprising; similar
results hold in other branches of combinatorial game theory, by and
large with the same proofs. The fact that our a priori notions of
order and equivalence coincide with the contextual order and
equivalence does provide an additional measure of evidence that our
definitions are good. (Note that the order and equivalence of games
over $\Bool$ is likely uncontroversial, since there are only three
passable game values $\s{\bot,\Star,\top}$, and they are completely
determined by the existence of winning strategies for the players.)

But the real reason we stated Definition~\ref{def-contextual} is to be
able to generalize it to a setting where the analog of
Proposition~\ref{prop-context} is false. We do so in
Section~\ref{ssec-global}.

\subsection{Games with Globally Decisive Moves}\label{ssec-global}

Consider the Hex region in Figure~\ref{fig-globally-decisive}. It is a
3-terminal region in which two of Black's terminals are opposing board
edges, and two of White's terminals are opposing board edges. We call
such a region a {\em one-sided fork}, because it is a fork
(Figure~\ref{fig-region-edges}) with the additional property that
connecting terminal 3 to terminal 1 is strictly better than connecting
it to terminal 2, and in fact as good as connecting all three terminals.

Because four of the one-sided fork's terminals are edges of the global
Hex board, it is clear that if Black connects her two edges within the
region, then she immediately wins the global game, rather than just
winning in the local region. In other words, there is no move
whatsoever that White could play outside the region that would trump
such a connection by Black. In this case, we say that the outcome
$\top$ of the local region is ``globally decisive''. Similarly, if
White connects his two edges within the region, he immediately wins
the global game and trumps anything that Black could do outside the
region. Therefore, the outcome $\bot$ is also globally decisive. It
turns out that the theory of regions with globally decisive $\top$ and
$\bot$ is slightly different than that with ordinary $\top$ and
$\bot$; it sometimes (but not always) happens that some games that
would otherwise be inequivalent can become equivalent due to global
decisiveness.
\begin{figure}
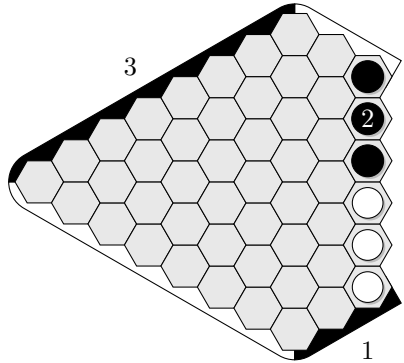

  \[
  \begin{hexboard}[scale=0.8]
    \shadows
    \foreach\i in {1,...,8} {\hex(1,\i)}
    \foreach\i in {1,...,8} {\hex(2,\i)}
    \foreach\i in {1,...,8} {\hex(3,\i)}
    \foreach\i in {1,...,7} {\hex(4,\i)}
    \foreach\i in {1,...,6} {\hex(5,\i)}
    \foreach\i in {1,...,5} {\hex(6,\i)}
    \foreach\i in {1,...,4} {\hex(7,\i)}
    \foreach\i in {1,...,3} {\hex(8,\i)}
    \edge[\sw](1,1)(1,8)
    \edge[\ne\noacutecorner](8,3)(8,1)
    \edge[\nw](1,1)(8,1)
    \edge[\se\noacutecorner](3,8)(1,8)
    \black(8,3)
    \black(7,4)\label{2}
    \black(6,5)
    \white(5,6)
    \white(4,7)
    \white(3,8)
    \cell(5,-0.5)\label{3}
    \cell(1.5,9.5)\label{1}
  \end{hexboard}
  \]
  \caption{A 3-terminal region with globally decisive $\top$ and
    $\bot$, or ``one-sided fork''.}
  \label{fig-globally-decisive}
\end{figure}

In this section, we give a precise definition of global decisiveness
and some examples of how globally decisive equivalence does not
coincide with ordinary equivalence. Much of the theory of global
decisiveness has not yet been worked out and is left for future work.

\begin{definition}
  Let $A$ and $B$ be posets, and assume $A$ has top and bottom
  elements. A monotone function $f:A\times B\to \Bool$ is called {\em
    (left) strict} if $f(\top,b)=\top$ and $f(\bot,b)=\bot$ for all
  $b\in B$.  A context $(B,f,K)$ for $A$ is called {\em strict} if $f$
  is left strict.
\end{definition}

\begin{lemma}\label{lem-strict-game}
  Let $(B,f,K)$ be a strict context over $A$. Then $\top\pf K\eq\top$
  and $\bot\pf K\eq\bot$.
\end{lemma}

\begin{proof}
  By the strictness of $f$, every atom occurring in $\top\pf K$ is
  $\top$, and therefore $\top\pf K\eq\top$ by repeated application of
  Lemma~\ref{lem-atomic-composite}. The second claim is dual.
\end{proof}

We say that a region of a game has {\em globally decisive} $\top$ and
$\bot$ if it can only be played in strict contexts. An example of this
is the one-sided fork in Figure~\ref{fig-globally-decisive}. It then
makes sense to define an ordering and equivalence on games by taking
only strict contexts into account. The following definition does this.

\begin{definition}[Globally decisive order and equivalence]\label{def-global-decisive}
  Let $A$ be a poset with top and bottom elements. The {\em globally
    decisive order} relations $\leqgI$, $\leqgII$, and $\leqgIII$ are
  relations on games over $A$ that are defined in exactly the same way
  as the relations $\leqcI$, $\leqcII$, and $\leqcIII$ in
  Definition~\ref{def-contextual}, except that all contexts are
  restricted to {\em strict} contexts. Specifically:
  \begin{itemize}
  \item $G\leqgI H$ if for all strict contexts $(B,f,K)$, if Left has a
    first-player winning strategy in $G\pf K$, then Left has a
    first-player winning strategy in $H\pf K$.
  \item $G\leqgII H$ if for all strict contexts $(B,f,K)$, if Left has a
    second-player winning strategy in $G\pf K$, then Left has a
    second-player winning strategy in $H\pf K$.
  \item $G\leqgIII H$ if for all strict contexts $(B,f,K)$, if Left has a
    second-player winning strategy in $G\pf K$, then Left has a
    first-player winning strategy in $H\pf K$.
  \end{itemize}
  We also write $G\leqg H$ if $G\leqgI H$ and $G\leqgII H$, and $G\eqg
  H$ if $G\leqg H$ and $H\leqg G$. The latter relation is called {\em
    globally decisive equivalence}.
\end{definition}

We note that these relations are at least as coarse as $\leq$ and
$\tri$; in other words, $G\leq H$ implies $G\leqgI H$ and $G\leqgII
H$, and $G\tri H$ implies $G\leqgIII H$. This is trivial, because if
some condition holds for {\em all} contexts, then it certainly holds
for all {\em strict} contexts.  But in contrast to the situation in
Section~\ref{ssec-contextual}, the globally decisive order does not
usually coincide with $\leq$. In fact, the following example shows
that $\leqgI$ and $\leqgII$ do not even coincide with each other.

\begin{example}
  Let $G=\top$ and $H=\Star=\g{\top|\bot}$. Consider any strict
  context $(B,f,K)$. By Lemma~\ref{lem-strict-game}, $G\pf K\eq \top$,
  so that Left has first- and second-player winning strategies in
  $G\pf K$, regardless of $K$ (effectively, Left has already won $G\pf
  K$). Also, the game $H\pf K$ has $\top\pf K\eq \top$ as a left
  option, and it has $\bot\pf K\eq\bot$ as a right option. Therefore,
  Left has a first-player winning strategy, but no second-player
  winning strategy, in $H\pf K$, again regardless of $K$. It follows
  that $G\leqgI H$ but $G\not\leqgII H$. Also, clearly $G\nleq H$,
  which shows that $\leqgI$ and $\leq$ do not coincide.  A dual
  argument shows that $\leqgII$ implies neither $\leq$ nor $\leqgI$.
\end{example}

The previous example shows that $\leqgI$ and $\leqgII$ do not coincide
with $\leq$, but one may wonder if the relation $\leqg$, defined as
the intersection of $\leqgI$ and $\leqgII$, coincides with $\leq$. The
following example shows that this is not the case. In other words,
even $G\leqgI H$ and $G\leqgII H$ together do not imply $G\leq H$.

\begin{example}\label{exa-g-counterex}
  Consider the linearly ordered poset $A=\s{\bot,a,b,\top}$, with
  $\bot<a<b<\top$. Let $G=\g{\g{\top|\g{b|\g{a|\bot}}} | \bot}$ and
  $H=\g{b,\g{\top|a}|\bot}$. One can check that $G\nleq H$. However,
  we have $G\leqgI H$ and $G\leqgII H$. To see $G\leqgII H$, note that
  Left never has a second-player win in $G$ in any strict context,
  because Right can always move to $\bot$. So $G\leqgII H$ is
  vacuously true.

  Proving $G\leqgI H$ is more tricky and more fun, because it requires
  actually playing the game. Let $G'=\g{b|\g{a|\bot}}$. We first claim
  that $G'\leqgI H$. To that end, let $(B,f,K)$ be some strict context
  and assume that Left has a first-player winning strategy in $G'\pf
  K$. We must show that Left has a first-player winning strategy in
  $H\pf K$. Left's first-player winning strategy in $G'\pf K$ must
  start with a winning move. There are two cases, depending on what
  this move is. Case~1: Left's winning move is $b\pf K$. But since
  $b\pf K$ is also a left option of $H\pf K$, Left then has a
  first-player win in $H\pf K$, which is what we had to prove. Case~2:
  Left's winning move is $\g{b|\g{a|\bot}}\pf K^L$ for some left
  option $K^L$ of $K$.  Since $\g{b|\g{a|\bot}}\pf K^L$ is a
  second-player win for Left, Left must have a first-player win in all
  of its right options, including in $\g{a|\bot}\pf K^L$. Left's
  winning move in $\g{a|\bot}\pf K^L$ can certainly not be in $K^L$,
  because then Right can play $\bot$ and win by strictness. Therefore,
  Left's winning move must be $a\pf K^L$, which means that Left has a
  second-player win in $a\pf K^L$. Now we will show that Left has a
  first-player win in $H\pf K$. We claim that the left option
  $\g{\top|a}\pf K$ is Left's winning move, i.e., that the game
  $\g{\top|a}\pf K$ is a second-player win for Left. So consider any
  right move. If Right moves in $K$, Left wins immediately by playing
  $\top$ and by strictness. If Right plays in the other component, the
  move is $a\pf K$, from which Left has the move $a\pf K^L$, which we
  already showed is winning for Left. Therefore, Left indeed has a
  first-player win in $H\pf K$. This concludes the proof that
  $G'\leqgI H$. Finally, it is easy to check that $G\leq G'$, which
  implies $G\leqgI G'\leqgI H$. The relation $\leqgI$ is transitive
  almost by definition, so $G\leqgI H$ holds.
\end{example}

\begin{remark}\label{rem-g-counterex}
  The games $G$ and $H$ in Example~\ref{exa-g-counterex} actually
  satisfy $G\eqg H$, but not $G\eq H$. Indeed, it is easy to check
  that $H\leq G$, hence $H\leqg G$. Since we proved $G\leqg H$, we
  have $G\eqg H$. Thus, the example shows that globally decisive
  equivalence does not coincide with ordinary equivalence of games.
  We will see a Hex example of this in Section~\ref{ssec-realizable}.
\end{remark}

\begin{remark}
  In Example~\ref{exa-g-counterex}, the proof that $G\leqg H$ does not
  imply $G\leq H$ uses the global decisiveness of both $\top$ and
  $\bot$. Indeed, it is necessary to use them both, because we can
  show: if $\leqg$ were defined with respect to contexts that are
  strict only for $\top$, or only for $\bot$, then $\leqg$ would
  coincide with $\leq$. To see why, it suffices to note that the
  function $\lambda$ defined in Section~\ref{ssec-copy-cat} is left
  strict with respect to $\top$ (but not $\bot$), and the function
  $\rho$ is left strict with respect to $\bot$ (but not $\top$). Now
  if $G\pf K\leq H\pf K$ holds for all $f$ that are strict for $\top$,
  then $G\pl G\opp\leq H\pl G\opp$ holds, which implies $G\leq H$ as
  in Proposition~\ref{prop-context}. Dually, if $G\pf K\leq H\pf K$
  holds for all $f$ that are strict for $\bot$, then $G\pr H\opp\leq
  H\pr H\opp$ holds, which also implies $G\leq H$.
\end{remark}

Having given a definition of global decisiveness and some examples, it
would be nice to know more of its properties. For example, it is not
clear from the definition whether the relation $\leqg$ is decidable,
since it potentially requires looking at all possible contexts, of
which there can be infinitely many. Even better would be to have a
recursive definition for $\leqg$, along similar lines as
Definition~\ref{def-order}, and a canonical form for games up to
globally decisive equivalence. However, this is left for future work.

\section{Enumeration of Game Values}\label{sec-enumeration}

\subsection{An Efficient Algorithm for Enumerating Game Values}
\label{ssec-enumeration-method}

Given a finite poset $A$, there is an easy, but inefficient, method
for enumerating all canonical forms of passable games of finite depth
over $A$. First we enumerate all games of depth 0; these are the
atomic games. Now given the set $\Gg_n$ of all canonical passable
games of depth up to $n$, which is finite, it is easy to construct the
set of all passable games of depth up to $n+1$; each such game is
either atomic, or it has a set of left options and a set of right
options that are subsets of $\Gg_n$. Since there are only finitely
many such subsets of $\Gg_n$, there are only finitely many potential
games at depth $n+1$ to consider. We can disregard all games that are
not passable or are not in canonical form. What is left is the set
$\Gg_{n+1}$ of all canonical passable games of depth up to $n+1$.

The enumeration method described in the previous paragraph is
extremely inefficient. Suppose there are 100 canonical passable games
at depth $n$. Then we must consider $2^{100}-1$ possible sets of left
options and $2^{100}-1$ possible sets of right options, given an
astronomical set of games to consider (most of which turn out not to
be in canonical form). The enumeration can be made far more efficient
using the concepts of left and right equivalence from
Section~\ref{sec-left-right}. The crucial insight is that we do not
need to consider {\em all} games $\g{L|R}$, where $L$ and $R$ are sets
of games of smaller depth. Instead, it suffices to consider just one
representative $L$ of each {\em left equivalence class} of sets of
games, and one representative $R$ of each right equivalence
class. When $L$ ranges over such representatives, all of the games
$\g{L|\bot}$ are passable and distinct. Therefore, the number of left
equivalence classes at each depth is not larger than the number of
games to be enumerated at the next depth. The dual statement holds for
right equivalence classes as well. This means that the number of games
$\g{L|R}$ that must be considered at depth $n+1$ is at most the square
of the number of distinct games that will ultimately be output at
depth $n+1$. The enumeration is therefore reasonably efficient
(relative to the amount of output produced), provided that we can
efficiently enumerate left and right equivalence classes of games.

The enumeration of left (or dually, right) equivalence classes can be
done efficiently as follows. Given a set $X=\s{G_1,\ldots,G_k}$ of
games, we first recursively compute representatives $L_1,\ldots,L_m$
for all the left equivalence classes of non-empty subsets of
$\s{G_1,\ldots,G_{k-1}}$. Then we consider the $2m+1$ sets
$L_1,\ldots,L_m$ and
$\s{G_k},L_1\cup\s{G_k},\ldots,L_m\cup\s{G_k}$. By eliminating
duplicates, we obtain the set of all representatives of left
equivalence classes of subsets of $X$. If the set $X$ has $k$ members,
we need to repeat this step $k$ times, and in each step, the number of
representatives potentially doubles. But since we eliminate duplicates
after each step, rather than only at the end, $m$ never exceeds the
final number of representatives computed. While this number may (or
may not) be exponential as a function of $k$, the runtime is
polynomial as a function of the amount of output produced.

\subsection{Game Values over Specific Atom Sets}\label{ssec-specific-values}

We compute the set of passable game values over various atom sets. Let
us write $\Lin_n$ for the linearly ordered set with $n$ elements.

\begin{itemize}
\item Over the set $\Lin_1=\s{\bot}$, there is only one passable game
  value, and it is $\bot$.
\item Over the set $\Lin_2=\s{\bot,\top}$, there are exactly 3 distinct
  passable game values, and they are $\bot$, $\Star=\g{\top|\bot}$, and
  $\top$.
\item Over the set $\Lin_3=\s{\bot,a,\top}$, there are exactly 8 distinct
  passable game values. They are shown, along with a Hasse diagram for
  the partial order $\leq$, in Figure~\ref{fig-l3}.
  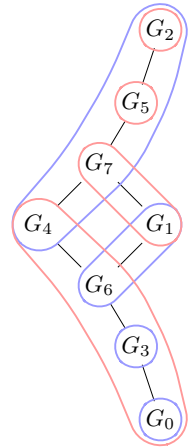
\begin{figure}
    \[
    \m{
      $\begin{array}{lll}
        \multicolumn{3}{l}{\mbox{Depth 0:}}\\
        G_0 &=& \bot\\
        G_1 &=& a\\
        G_2 &=& \top\\[1em]
        \multicolumn{3}{l}{\mbox{Depth 1:}}\\
        G_3 &=& \g{a\mid \bot}\\
        G_4 &=& \g{\top\mid \bot} ~~=~~ \Star\\
        G_5 &=& \g{\top\mid a}\\[1em]
        \multicolumn{3}{l}{\mbox{Depth 2:}}\\
        G_6 &=& \g{G_5\mid \bot}\\
        G_7 &=& \g{\top\mid G_3}\\
      \end{array}$
    }
    \hspace{2cm}
    \m{
      \scalebox{0.9}{$
      \begin{tikzpicture}[scale=0.9]
        \def\r{0.5}
        \def\s{0.6}
        \node(4) at (-1,0) {$G_4$};
        \node(1) at (1,0) {$G_1$};
        \node(6) at (0,-1) {$G_6$};
        \node(3) at (0.6,-2) {$G_3$};
        \node(0) at (1,-3.2) {$G_0$};
        \node(7) at (0,1) {$G_7$};
        \node(5) at (0.6,2) {$G_5$};
        \node(2) at (1,3.2) {$G_2$};
        \draw (0)--(3)--(6)--(4)--(7)--(5)--(2);
        \draw (6)--(1)--(7);
        \draw[color=blue!40,thick,rounded corners=3.4mm] ($(0)+(-0:\r)$) -- ($(0)+(-90:\r)$) -- ($(0)+(-180:\r)$) -- ($(0)+(-270:\r)$) -- cycle;
        \draw[color=blue!40,thick,rounded corners=3.4mm] ($(3)+(-0:\r)$) -- ($(3)+(-90:\r)$) -- ($(3)+(-180:\r)$) -- ($(3)+(-270:\r)$) -- cycle;
        \draw[color=blue!40,thick,rounded corners=3.4mm] ($(6)+(-90:\r)$)
        -- ($(6)+(-180:\r)$)
        -- ($(1)+(-270:\r)$)
        -- ($(1)+(-0:\r)$)
        -- cycle;
        \draw[color=blue!40,thick,rounded corners=3.9mm] ($(4)+(-90:\s)$)
        -- ($(4)+(-180:\s)$)
        -- ($(7)+(-200:\s)$)
        -- ($(5)+(-220:\s)$)
        -- ($(2)+(-240:\s)$)
        -- ($(2)+(-330:\s)$)
        -- ($(5)+(-350:\s)$)
        -- ($(7)+(-370:\s)$)
        -- cycle;
        \draw[color=red!40,thick,rounded corners=3.4mm] ($(2)+(0:\r)$) -- ($(2)+(90:\r)$) -- ($(2)+(180:\r)$) -- ($(2)+(270:\r)$) -- cycle;
        \draw[color=red!40,thick,rounded corners=3.4mm] ($(5)+(0:\r)$) -- ($(5)+(90:\r)$) -- ($(5)+(180:\r)$) -- ($(5)+(270:\r)$) -- cycle;
        \draw[color=red!40,thick,rounded corners=3.4mm] ($(7)+(90:\r)$)
        -- ($(7)+(180:\r)$)
        -- ($(1)+(270:\r)$)
        -- ($(1)+(0:\r)$)
        -- cycle;
        \draw[color=red!40,thick,rounded corners=3.9mm] ($(4)+(90:\s)$)
        -- ($(4)+(180:\s)$)
        -- ($(6)+(200:\s)$)
        -- ($(3)+(220:\s)$)
        -- ($(0)+(240:\s)$)
        -- ($(0)+(330:\s)$)
        -- ($(3)+(350:\s)$)
        -- ($(6)+(370:\s)$)
        -- cycle;
      \end{tikzpicture}
      $}
    }
    \]
    \caption{The passable game values over $\Lin_3=\s{\bot,a,\top}$ and
      their partial order. Left equivalence classes are indicated in
      red and right equivalence classes in blue.}
    \label{fig-l3}
  \end{figure}

\item Over the set $\Lin_4=\s{\bot,a,b,\top}$, with $\bot<a<b<\top$,
  there are exactly 31 distinct passable game values, and they are
  shown, along with a Hasse diagram for the partial order $\leq$, in
  Figure~\ref{fig-l4}.
  \begin{figure}
    \[
    \m{
      $\begin{array}{lll}
        \multicolumn{3}{l}{\mbox{Depth 0:}}\\
        G_{0} &=& \bot \\
        G_{1} &=& a \\
        G_{2} &=& b \\
        G_{3} &=& \top \\[1em]
        \multicolumn{3}{l}{\mbox{Depth 1:}}\\
        G_{4} &=& \g{a\mid \bot} \\
        G_{5} &=& \g{b\mid \bot} \\
        G_{6} &=& \g{\top\mid \bot} \\
        G_{7} &=& \g{b\mid a} \\
        G_{8} &=& \g{\top\mid a} \\
        G_{9} &=& \g{\top\mid b} \\[1em]
        \multicolumn{3}{l}{\mbox{Depth 2:}}\\
        G_{10} &=& \g{b,G_{8}\mid \bot} \\
        G_{11} &=& \g{G_{7}\mid \bot} \\
        G_{12} &=& \g{G_{8}\mid \bot} \\
        G_{13} &=& \g{G_{9}\mid \bot} \\
        G_{14} &=& \g{G_{9}\mid a} \\
        G_{15} &=& \g{\top\mid a,G_{5}} \\
        G_{16} &=& \g{b\mid G_{4}} \\
        G_{17} &=& \g{\top\mid G_{4}} \\
        G_{18} &=& \g{G_{9}\mid G_{4}} \\
        G_{19} &=& \g{\top\mid G_{5}} \\
        G_{20} &=& \g{\top\mid G_{7}} \\[1em]
        \multicolumn{3}{l}{\mbox{Depth 3:}}\\
        G_{21} &=& \g{G_{14}\mid \bot} \\
        G_{22} &=& \g{G_{15}\mid \bot} \\
        G_{23} &=& \g{G_{20}\mid \bot} \\
        G_{24} &=& \g{G_{20}\mid G_{4}} \\
        G_{25} &=& \g{\top\mid G_{16}} \\
        G_{26} &=& \g{\top\mid G_{10}} \\
        G_{27} &=& \g{\top\mid G_{11}} \\
        G_{28} &=& \g{G_{9}\mid G_{11}} \\[1em]
        \multicolumn{3}{l}{\mbox{Depth 4:}}\\
        G_{29} &=& \g{G_{25}\mid \bot} \\
        G_{30} &=& \g{\top\mid G_{21}} \\
      \end{array}$
    }
    \hspace{2cm}
    \m{
      \scalebox{0.9}{$
      \begin{tikzpicture}[scale=0.9]
        \def\r{0.5}
        \def\s{0.6}
        \node(6) at (-1,0) {$G_{6}$};
        \node(18) at (1,0) {$G_{18}$};
        \node(7) at (3,0) {$G_{7}$};
        \node(17) at (0,1) {$G_{17}$};
        \node(28) at (2,1) {$G_{28}$};
        \node(14) at (2.5,2) {$G_{14}$};
        \node(2) at (3,4) {$G_{2}$};
        \node(27) at (0.6,2) {$G_{27}$};
        \node(30) at (1,3) {$G_{30}$};
        \node(15) at (1.2,4) {$G_{15}$};
        \node(19) at (1.8,5) {$G_{19}$};
        \node(8) at (0.8,5) {$G_{8}$};
        \node(26) at (1.4,6) {$G_{26}$};
        \node(25) at (1.4,7) {$G_{25}$};
        \node(20) at (1.4,8) {$G_{20}$};
        \node(9) at (1.4,9) {$G_{9}$};
        \node(3) at (1.4,10) {$G_{3}$};
        \node(13) at (0,-1) {$G_{13}$};
        \node(24) at (2,-1) {$G_{24}$};
        \node(16) at (2.5,-2) {$G_{16}$};
        \node(1) at (3,-4) {$G_{1}$};
        \node(23) at (0.6,-2) {$G_{23}$};
        \node(29) at (1,-3) {$G_{29}$};
        \node(10) at (1.2,-4) {$G_{10}$};
        \node(12) at (1.8,-5) {$G_{12}$};
        \node(5) at (0.8,-5) {$G_{5}$};
        \node(22) at (1.4,-6) {$G_{22}$};
        \node(21) at (1.4,-7) {$G_{21}$};
        \node(11) at (1.4,-8) {$G_{11}$};
        \node(4) at (1.4,-9) {$G_{4}$};
        \node(0) at (1.4,-10) {$G_{0}$};
        \draw (0)--(4)--(11)--(21)--(22)--(5)--(10)--(29)--(23)--(13)
        --(6)--(17)--(27)--(30)--(15)--(8)--(26)--(25)--(20)--(9)--(3);
        \draw (22)--(12)--(1)--(16)--(24)--(18)--(28)--(14)--(2)--(19)--(26);
        \draw (12)--(10);
        \draw (15)--(19);
        \draw (29)--(16);
        \draw (14)--(30);
        \draw (23)--(24)--(7)--(28)--(27);
        \draw (13)--(18)--(17);
        \draw[color=blue!40,thick,rounded corners=3.4mm] ($(0)+(-0:\r)$) -- ($(0)+(-90:\r)$) -- ($(0)+(-180:\r)$) -- ($(0)+(-270:\r)$) -- cycle;
        \draw[color=blue!40,thick,rounded corners=3.4mm] ($(4)+(-0:\r)$) -- ($(4)+(-90:\r)$) -- ($(4)+(-180:\r)$) -- ($(4)+(-270:\r)$) -- cycle;
        \draw[color=blue!40,thick,rounded corners=3.4mm] ($(11)+(-0:\r)$) -- ($(11)+(-90:\r)$) -- ($(11)+(-180:\r)$) -- ($(11)+(-270:\r)$) -- cycle;
        \draw[color=blue!40,thick,rounded corners=3.4mm] ($(21)+(-0:\r)$) -- ($(21)+(-90:\r)$) -- ($(21)+(-180:\r)$) -- ($(21)+(-270:\r)$) -- cycle;
        \draw[color=blue!40,thick,rounded corners=3.4mm] ($(22)+(-0:\r)$) -- ($(22)+(-90:\r)$) -- ($(22)+(-180:\r)$) -- ($(22)+(-270:\r)$) -- cycle;
        \draw[color=blue!40,thick,rounded corners=3.4mm] ($(5)+(-0:\r)$) -- ($(5)+(-90:\r)$) -- ($(5)+(-180:\r)$) -- ($(5)+(-270:\r)$) -- cycle;
        \draw[color=blue!40,thick,rounded corners=3.4mm] ($(10)+(-0:\r)$) -- ($(10)+(-90:\r)$) -- ($(10)+(-180:\r)$) -- ($(10)+(-270:\r)$) -- cycle;
        \draw[color=blue!40,thick,rounded corners=3.4mm] ($(12)+(-90:\r)$)
        -- ($(12)+(-180:\r)$)
        -- ($(1)+(-270:\r)$)
        -- ($(1)+(-0:\r)$)
        -- cycle;
        \draw[color=blue!40,thick,rounded corners=3.4mm] ($(29)+(-100:\r)$)
        -- ($(29)+(-190:\r)$)
        -- ($(16)+(-280:\r)$)
        -- ($(16)+(-10:\r)$)
        -- cycle;
        \draw[color=blue!40,thick,rounded corners=3.9mm] ($(23)+(-100:\s)$)
        -- ($(23)+(-190:\s)$)
        -- ($(24)+(-270:\s)$)
        -- ($(7)+(-270:\s)$)
        -- ($(7)+(-0:\s)$)
        -- ($(24)+(-90:\s)$)
        -- cycle;
        \draw[color=blue!40,thick,rounded corners=3.9mm] ($(13)+(-90:\s)$)
        -- ($(13)+(-180:\s)$)
        -- ($(18)+(-180:\s)$)
        -- ($(28)+(-200:\s)$)
        -- ($(14)+(-240:\s)$)
        -- ($(2)+(-240:\s)$)
        -- ($(2)+(-330:\s)$)
        -- ($(14)+(-350:\s)$)
        -- ($(28)+(-70:\s)$)
        -- ($(13)+(-370:\s)$)
        -- cycle;
        \draw[color=blue!40,thick,rounded corners=3.9mm] ($(6)+(-90:\s)$)
        -- ($(6)+(-180:\s)$)
        -- ($(17)+(-180:\s)$)
        -- ($(27)+(-170:\s)$)
        -- ($(30)+(-180:0.8*\s)$)
        -- ($(15)+(-180:\s)$)
        -- ($(8)+(-180:0.9*\s)$)
        -- ($(26)+(-225:\s)$)
        -- ($(3)+(-225:\s)$)
        -- ($(3)+(-315:\s)$)
        -- ($(26)+(-315:\s)$)
        -- ($(19)+(-0:0.8*\s)$)
        -- ($(15)+(-45:\s)$)
        -- ($(30)+(-20:0.8*\s)$)
        -- ($(27)+(-60:\s)$)
        -- ($(17)+(-80:\s)$)
        -- cycle;
        \draw[color=red!40,thick,rounded corners=3.4mm] ($(3)+(0:\r)$) -- ($(3)+(90:\r)$) -- ($(3)+(180:\r)$) -- ($(3)+(270:\r)$) -- cycle;
        \draw[color=red!40,thick,rounded corners=3.4mm] ($(9)+(0:\r)$) -- ($(9)+(90:\r)$) -- ($(9)+(180:\r)$) -- ($(9)+(270:\r)$) -- cycle;
        \draw[color=red!40,thick,rounded corners=3.4mm] ($(20)+(0:\r)$) -- ($(20)+(90:\r)$) -- ($(20)+(180:\r)$) -- ($(20)+(270:\r)$) -- cycle;
        \draw[color=red!40,thick,rounded corners=3.4mm] ($(25)+(0:\r)$) -- ($(25)+(90:\r)$) -- ($(25)+(180:\r)$) -- ($(25)+(270:\r)$) -- cycle;
        \draw[color=red!40,thick,rounded corners=3.4mm] ($(26)+(0:\r)$) -- ($(26)+(90:\r)$) -- ($(26)+(180:\r)$) -- ($(26)+(270:\r)$) -- cycle;
        \draw[color=red!40,thick,rounded corners=3.4mm] ($(8)+(0:\r)$) -- ($(8)+(90:\r)$) -- ($(8)+(180:\r)$) -- ($(8)+(270:\r)$) -- cycle;
        \draw[color=red!40,thick,rounded corners=3.4mm] ($(15)+(0:\r)$) -- ($(15)+(90:\r)$) -- ($(15)+(180:\r)$) -- ($(15)+(270:\r)$) -- cycle;
        \draw[color=red!40,thick,rounded corners=3.4mm] ($(19)+(90:\r)$)
        -- ($(19)+(180:\r)$)
        -- ($(2)+(270:\r)$)
        -- ($(2)+(0:\r)$)
        -- cycle;
        \draw[color=red!40,thick,rounded corners=3.4mm] ($(30)+(100:\r)$)
        -- ($(30)+(190:\r)$)
        -- ($(14)+(280:\r)$)
        -- ($(14)+(10:\r)$)
        -- cycle;
        \draw[color=red!40,thick,rounded corners=3.9mm] ($(27)+(100:\s)$)
        -- ($(27)+(190:\s)$)
        -- ($(28)+(270:\s)$)
        -- ($(7)+(270:\s)$)
        -- ($(7)+(0:\s)$)
        -- ($(28)+(90:\s)$)
        -- cycle;
        \draw[color=red!40,thick,rounded corners=3.9mm] ($(17)+(90:\s)$)
        -- ($(17)+(180:\s)$)
        -- ($(18)+(180:\s)$)
        -- ($(24)+(200:\s)$)
        -- ($(16)+(240:\s)$)
        -- ($(1)+(240:\s)$)
        -- ($(1)+(330:\s)$)
        -- ($(16)+(350:\s)$)
        -- ($(24)+(70:\s)$)
        -- ($(17)+(370:\s)$)
        -- cycle;
        \draw[color=red!40,thick,rounded corners=3.9mm] ($(6)+(90:\s)$)
        -- ($(6)+(180:\s)$)
        -- ($(13)+(180:\s)$)
        -- ($(23)+(170:\s)$)
        -- ($(29)+(180:0.8*\s)$)
        -- ($(10)+(180:\s)$)
        -- ($(5)+(180:0.9*\s)$)
        -- ($(22)+(225:\s)$)
        -- ($(0)+(225:\s)$)
        -- ($(0)+(315:\s)$)
        -- ($(22)+(315:\s)$)
        -- ($(12)+(0:0.8*\s)$)
        -- ($(10)+(45:\s)$)
        -- ($(29)+(20:0.8*\s)$)
        -- ($(23)+(60:\s)$)
        -- ($(13)+(80:\s)$)
        -- cycle;
      \end{tikzpicture}
      $}
    }
    \]
    \caption{The passable game values over $\Lin_4=\s{\bot,a,b,\top}$ and
      their partial order. Left equivalence classes are indicated in
      red and right equivalence classes in blue.}
    \label{fig-l4}
  \end{figure}
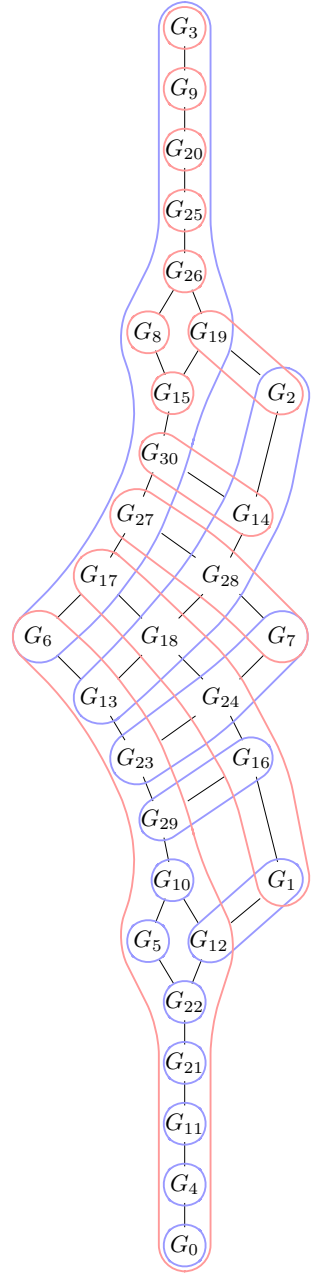
\end{itemize}

Note that by Theorem~\ref{thm-passable-linear-monotone}, since all of
the games enumerated above are over linear atom sets and in canonical
form, these games are not just passable but also monotone.

\begin{remark}
  Figures~\ref{fig-l3} and {\ref{fig-l4}} also show the left and right
  equivalence classes of games. These figures confirm that the
  intersection of each left equivalence class with each right
  equivalence class is at most a singleton, as proved in
  Lemma~\ref{lem-eq-eql-eqr}. They also illustrate that each left
  equivalence class has a unique maximal element, and that these
  elements are exactly the members of the right equivalence class of
  $\top$. This was shown in Lemma~\ref{lem-right-eq-top} for arbitrary
  games, but the same proof also applies to passable games.  We
  further note the set of game values in these figures forms a lattice
  under $\leq$; this was shown in Proposition~\ref{prop-lattice} and
  Lemma~\ref{lem-passable-lattice}.
\end{remark}

From the above, one may wonder whether the collection of passable game
values over a finite atom poset is always finite. This is not the
case. In fact, the following proposition, due to Eric Demer, shows
that the collection of game values is always infinite when the poset $A$
is not linearly ordered.

\begin{proposition}[Demer]\label{prop-infinite}
  Let $A$ be a poset with two incomparable elements $a$ and
  $b$. Define a sequence of games by $G_0=a$ and $G_{n} =
  \g{a,b|G_{n-1}}$ for all $n\geq 1$. Then $G_0,G_1,G_2,\ldots$ is an
  infinite, strictly increasing sequence of passable games over
  $A$. In particular, there are infinitely many non-equivalent
  passable games over $A$.
\end{proposition}

\begin{proof}
  We first show, by induction on $n$, that $G_n\leq G_{n+1}$ for all
  $n\geq 0$. For $n=0$, this is clear because $G_0=a\eq \g{a|a}\leq
  \g{a,b|a}=G_1$.  For the induction step, assume $G_n\leq
  G_{n+1}$. Then $G_{n+1}=\g{a,b|G_n}\leq \g{a,b|G_{n+1}}=G_{n+2}$ by
  Lemma~\ref{lem-composite-monotone}. Thus, $G_0,G_1,G_2,\ldots$ is an
  increasing sequence of games. Next, we show that all $G_n$ are
  passable. It suffices to show that $G_n\tri G_n$ for all $n\geq
  0$. For $n=0$, this is trivial since $G_0=a$ is atomic. For $n\geq
  1$, we have $G_n\tri G_{n-1}$ since $G_{n-1}$ is a right option of
  $G_n$. Since we have already proved $G_{n-1}\leq G_n$, this implies
  $G_n\tri G_n$ by Lemma~\ref{lem-transitive}. Finally, we must show
  that the sequence $(G_n)_{n\in\N}$ is strictly increasing, i.e., we
  must show that $G_{n+1}\nleq G_{n}$, for all $n\geq 0$. We show this
  by induction. The base case $n=0$ holds because $\g{a,b|a}\nleq
  a$. For the induction step, assume $G_{n+1}\nleq G_{n}$. We must
  show $G_{n+2}\nleq G_{n+1}$, i.e., $\g{a,b|G_{n+1}}\nleq
  \g{a,b|G_n}$. Assume, on the contrary, that $\g{a,b|G_{n+1}}\leq
  \g{a,b|G_n}$. Then by definition of $\leq$, we have
  $\g{a,b|G_{n+1}}\tri G_n$. By definition of $\tri$, there are two
  cases. Case 1: $\g{a,b|G_{n+1}}\leq G_n^L$ for some left option
  $G_n^L$ of $G_n$. But this is plainly not the case, since the only
  left options of $G_n$ are $a$ and $b$, and neither is greater than
  $\g{a,b|G_{n+1}}$. Case 2: $G_{n+1}\leq G_n$, but this contradicts
  our induction hypothesis. Therefore $G_{n+2}\nleq G_{n+1}$ as
  claimed. This completes the proof.  
\end{proof}

Next, one may wonder whether it is at least the case that the
collection of passable game values is finite whenever the atom poset
is {\em linearly} ordered and finite.  This is also false: as shown in
{\cite{DS2022-hex-l5}}, there exist infinitely many non-equivalent
passable games over $\Lin_5$. It follows that the atom sets
$\Lin_1,\ldots,\Lin_4$ (and trivially, the empty atom set) are the
only ones over which the collection of passable game values is finite.

Over atom sets where the collection of passable game values is
infinite, we can still enumerate passable games of small depth.

\begin{itemize}
\item Over the set $\Lin_5$, there are exactly 5 canonical-form passable
  games of depth 0, 10 such games of depth 1, 40 games of depth 2, 178
  games of depth 3, and 2962 games of depth 4, for a total of 3195
  canonical-form passable games of depth up to 4.
\item Over the poset $A=\s{\bot,a,b,\top}$, with $a$ and $b$
  incomparable (see Figure~\ref{fig-region-edges}(b)), there are
  exactly 4 canonical-form passable games of depth 0, 11 such games of
  depth 1, and 291 games of depth 2, for a total of 306 games of depth
  up to 2. We enumerated more than 43000 distinct passable games of
  depth 3 before running out of memory.
\item Over the poset $A=\s{\bot,a,b,c,\top}$, with $a$, $b$, and $c$
  incomparable (see Figure~\ref{fig-region}(b)), there are exactly 5
  canonical-form passable games of depth 0 and 33 such games of depth
  1. We enumerated more than 1.8 million distinct passable games of
  depth 2 before running out of memory.
\end{itemize}

\section{Realizability as Hex Positions}\label{sec-realizable}

As we saw, the value of a Hex position is always a passable game.
However, it is not a priori clear whether {\em every} passable
game arises as the value of some Hex position --- and indeed, we will
show in Proposition~\ref{prop-abstract-not-hex} that this is not in
general the case.

\subsection{Realizable Game Values}\label{ssec-realizable}

We start by reporting some positive results on Hex-realizable game
values over small atom posets.

\begin{itemize}
  \item For 2-terminal regions in Hex, the outcome poset is
    $\Bool=\s{\bot,\top}$, and there are three passable game values:
    $\bot$, $\Star$, and $\top$. All of these passable game values are
    realizable as Hex positions, as shown in
    Figure~\ref{fig-realizable-2-terminal}.
  \begin{figure}
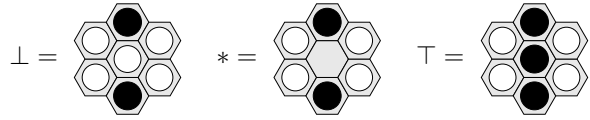

    \def\scale{0.7}
    \def\lift{\mp{0.45}}
    \[
    \bot = \lift{
      \scalebox{\scale}{$\begin{hexboard}
          \noshadows
          \foreach\i in {2,...,3} {\hex(1,\i)}
          \foreach\i in {1,...,3} {\hex(2,\i)}
          \foreach\i in {1,...,2} {\hex(3,\i)}
          \black(3,1)
          \black(1,3)
          \white(2,1)
          \white(1,2)
          \white(2,3)
          \white(3,2)
          \white(2,2)
        \end{hexboard}$}
    }
    \quad
    \Star = \lift{
      \scalebox{\scale}{$\begin{hexboard}
          \noshadows
          \foreach\i in {2,...,3} {\hex(1,\i)}
          \foreach\i in {1,...,3} {\hex(2,\i)}
          \foreach\i in {1,...,2} {\hex(3,\i)}
          \black(3,1)
          \black(1,3)
          \white(2,1)
          \white(1,2)
          \white(2,3)
          \white(3,2)
        \end{hexboard}$}
    }
    \quad
    \top = \lift{
      \scalebox{\scale}{$\begin{hexboard}
          \noshadows
          \foreach\i in {2,...,3} {\hex(1,\i)}
          \foreach\i in {1,...,3} {\hex(2,\i)}
          \foreach\i in {1,...,2} {\hex(3,\i)}
          \black(3,1)
          \black(1,3)
          \white(2,1)
          \white(1,2)
          \white(2,3)
          \white(3,2)
          \black(2,2)
        \end{hexboard}$}
    }
    \]
    \caption{2-terminal Hex positions realizing all passable
      game values over $\Lin_2=\Bool$.}
    \label{fig-realizable-2-terminal}
  \end{figure}

\item For 2-terminal Hex regions with gap (see
  Figure~\ref{fig-region-gap}), the outcome poset is
  $\Lin_3=\s{\bot,a,\top}$. Recall from Figure~\ref{fig-l3} that there
  are 8 passable game values over this poset. All of them are
  realizable as Hex positions, as shown in
  Figure~\ref{fig-realizable-2-terminal-gap}.
  \begin{figure}
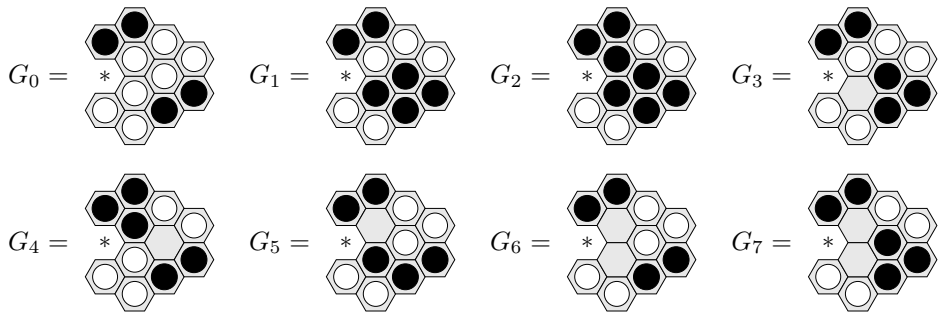

    \def\scale{0.65}
    \def\lift{\mp{0.45}}
    \[
    G_0 = \lift{
      \scalebox{\scale}{$\begin{hexboard}
          \noshadows
          \foreach\i in {3,...,4} {\hex(1,\i)}
          \foreach\i in {3,...,4} {\hex(2,\i)}
          \foreach\i in {1,...,4} {\hex(3,\i)}
          \foreach\i in {1,...,3} {\hex(4,\i)}
          \black(3,1)
          \black(4,1)
          \black(2,4)
          \black(3,4)
          \white(1,3)
          \white(1,4)
          \white(4,2)
          \white(4,3)
          \cell(2,2)\label{\Large$\ast$}
          \white(2,3)
          \white(3,2)
          \white(3,3)
        \end{hexboard}$}
    }
    \quad
    G_1 = \lift{
      \scalebox{\scale}{$\begin{hexboard}
          \noshadows
          \foreach\i in {3,...,4} {\hex(1,\i)}
          \foreach\i in {3,...,4} {\hex(2,\i)}
          \foreach\i in {1,...,4} {\hex(3,\i)}
          \foreach\i in {1,...,3} {\hex(4,\i)}
          \black(3,1)
          \black(4,1)
          \black(2,4)
          \black(3,4)
          \white(1,3)
          \white(1,4)
          \white(4,2)
          \white(4,3)
          \cell(2,2)\label{\Large$\ast$}
          \black(2,3)
          \white(3,2)
          \black(3,3)
        \end{hexboard}$}
    }
    \quad
    G_2 = \lift{
      \scalebox{\scale}{$\begin{hexboard}
          \noshadows
          \foreach\i in {3,...,4} {\hex(1,\i)}
          \foreach\i in {3,...,4} {\hex(2,\i)}
          \foreach\i in {1,...,4} {\hex(3,\i)}
          \foreach\i in {1,...,3} {\hex(4,\i)}
          \black(3,1)
          \black(4,1)
          \black(2,4)
          \black(3,4)
          \white(1,3)
          \white(1,4)
          \white(4,2)
          \white(4,3)
          \cell(2,2)\label{\Large$\ast$}
          \black(2,3)
          \black(3,2)
          \black(3,3)
        \end{hexboard}$}
    }
    \quad
    G_3 = \lift{
      \scalebox{\scale}{$\begin{hexboard}
          \noshadows
          \foreach\i in {3,...,4} {\hex(1,\i)}
          \foreach\i in {3,...,4} {\hex(2,\i)}
          \foreach\i in {1,...,4} {\hex(3,\i)}
          \foreach\i in {1,...,3} {\hex(4,\i)}
          \black(3,1)
          \black(4,1)
          \black(2,4)
          \black(3,4)
          \white(1,3)
          \white(1,4)
          \white(4,2)
          \white(4,3)
          \cell(2,2)\label{\Large$\ast$}
          \white(3,2)
          \black(3,3)
        \end{hexboard}$}
    }
    \]
    \[
    G_4 = \lift{
      \scalebox{\scale}{$\begin{hexboard}
          \noshadows
          \foreach\i in {3,...,4} {\hex(1,\i)}
          \foreach\i in {3,...,4} {\hex(2,\i)}
          \foreach\i in {1,...,4} {\hex(3,\i)}
          \foreach\i in {1,...,3} {\hex(4,\i)}
          \black(3,1)
          \black(4,1)
          \black(2,4)
          \black(3,4)
          \white(1,3)
          \white(1,4)
          \white(4,2)
          \white(4,3)
          \cell(2,2)\label{\Large$\ast$}
          \white(2,3)
          \black(3,2)
        \end{hexboard}$}
    }
    \quad
    G_5 = \lift{
      \scalebox{\scale}{$\begin{hexboard}
          \noshadows
          \foreach\i in {3,...,4} {\hex(1,\i)}
          \foreach\i in {3,...,4} {\hex(2,\i)}
          \foreach\i in {1,...,4} {\hex(3,\i)}
          \foreach\i in {1,...,3} {\hex(4,\i)}
          \black(3,1)
          \black(4,1)
          \black(2,4)
          \black(3,4)
          \white(1,3)
          \white(1,4)
          \white(4,2)
          \white(4,3)
          \cell(2,2)\label{\Large$\ast$}
          \black(2,3)
          \white(3,3)
        \end{hexboard}$}
    }
    \quad
    G_6 = \lift{
      \scalebox{\scale}{$\begin{hexboard}
          \noshadows
          \foreach\i in {3,...,4} {\hex(1,\i)}
          \foreach\i in {3,...,4} {\hex(2,\i)}
          \foreach\i in {1,...,4} {\hex(3,\i)}
          \foreach\i in {1,...,3} {\hex(4,\i)}
          \black(3,1)
          \black(4,1)
          \black(2,4)
          \black(3,4)
          \white(1,3)
          \white(1,4)
          \white(4,2)
          \white(4,3)
          \cell(2,2)\label{\Large$\ast$}
          \white(3,3)
        \end{hexboard}$}
    }
    \quad
    G_7 = \lift{
      \scalebox{\scale}{$\begin{hexboard}
          \noshadows
          \foreach\i in {3,...,4} {\hex(1,\i)}
          \foreach\i in {3,...,4} {\hex(2,\i)}
          \foreach\i in {1,...,4} {\hex(3,\i)}
          \foreach\i in {1,...,3} {\hex(4,\i)}
          \black(3,1)
          \black(4,1)
          \black(2,4)
          \black(3,4)
          \white(1,3)
          \white(1,4)
          \white(4,2)
          \white(4,3)
          \cell(2,2)\label{\Large$\ast$}
          \black(3,3)
        \end{hexboard}$}
    }
    \]
    \caption{2-terminal Hex positions with gap, realizing all 
      passable game values over $\Lin_3$.}
    \label{fig-realizable-2-terminal-gap}
  \end{figure}

\item For one-sided forks (see Figure~\ref{fig-globally-decisive}),
  the outcome poset is also $\Lin_3=\s{\bot,a,\top}$. In this setting
  too, all 8 passable game values are realizable as Hex positions, as
  shown in Figure~\ref{fig-realizable-one-sided-fork}.
  \begin{figure}
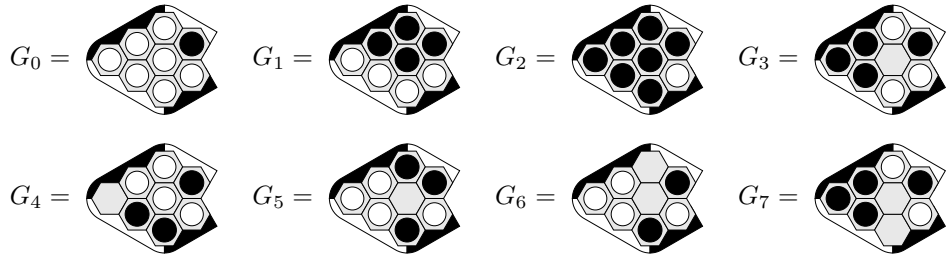

    \def\scale{0.6}
    \def\lift{\mp{0.45}}
    \[
    G_0 = \lift{
      \scalebox{\scale}{$\begin{hexboard}
          \noshadows
          \foreach\i in {1,...,3} {\hex(1,\i)}
          \foreach\i in {1,...,3} {\hex(2,\i)}
          \foreach\i in {1,...,2} {\hex(3,\i)}
          \edge[\sw](1,1)(1,3)
          \edge[\ne\noacutecorner](3,2)(3,1)
          \edge[\nw](1,1)(3,1)
          \edge[\se\noacutecorner](2,3)(1,3)
          \black(3,2)
          \white(2,3)
          \white(1,1)
          \white(1,2)
          \white(1,3)
          \white(2,1)
          \white(2,2)
          \white(3,1)
        \end{hexboard}$}
    }
    \quad
    G_1 = \lift{
      \scalebox{\scale}{$\begin{hexboard}
          \noshadows
          \foreach\i in {1,...,3} {\hex(1,\i)}
          \foreach\i in {1,...,3} {\hex(2,\i)}
          \foreach\i in {1,...,2} {\hex(3,\i)}
          \edge[\sw](1,1)(1,3)
          \edge[\ne\noacutecorner](3,2)(3,1)
          \edge[\nw](1,1)(3,1)
          \edge[\se\noacutecorner](2,3)(1,3)
          \black(3,2)
          \white(2,3)
          \black(3,1)
          \white(1,3)
          \black(2,1)
          \black(2,2)
          \white(1,1)
          \white(1,2)
        \end{hexboard}$}
    }
    \quad
    G_2 = \lift{
      \scalebox{\scale}{$\begin{hexboard}
          \noshadows
          \foreach\i in {1,...,3} {\hex(1,\i)}
          \foreach\i in {1,...,3} {\hex(2,\i)}
          \foreach\i in {1,...,2} {\hex(3,\i)}
          \edge[\sw](1,1)(1,3)
          \edge[\ne\noacutecorner](3,2)(3,1)
          \edge[\nw](1,1)(3,1)
          \edge[\se\noacutecorner](2,3)(1,3)
          \black(3,2)
          \white(2,3)
          \black(1,1)
          \black(1,2)
          \black(1,3)
          \black(2,1)
          \black(2,2)
          \black(3,1)
        \end{hexboard}$}
    }
    \quad
    G_3 = \lift{
      \scalebox{\scale}{$\begin{hexboard}
          \noshadows
          \foreach\i in {1,...,3} {\hex(1,\i)}
          \foreach\i in {1,...,3} {\hex(2,\i)}
          \foreach\i in {1,...,2} {\hex(3,\i)}
          \edge[\sw](1,1)(1,3)
          \edge[\ne\noacutecorner](3,2)(3,1)
          \edge[\nw](1,1)(3,1)
          \edge[\se\noacutecorner](2,3)(1,3)
          \black(3,2)
          \white(2,3)
          \black(1,1)
          \black(1,2)
          \white(1,3)
          \black(2,1)
          \white(3,1)
        \end{hexboard}$}
    }
    \]
    \[
    G_4 = \lift{
      \scalebox{\scale}{$\begin{hexboard}
          \noshadows
          \foreach\i in {1,...,3} {\hex(1,\i)}
          \foreach\i in {1,...,3} {\hex(2,\i)}
          \foreach\i in {1,...,2} {\hex(3,\i)}
          \edge[\sw](1,1)(1,3)
          \edge[\ne\noacutecorner](3,2)(3,1)
          \edge[\nw](1,1)(3,1)
          \edge[\se\noacutecorner](2,3)(1,3)
          \black(3,2)
          \white(2,3)
          \black(1,2)
          \black(1,3)
          \white(2,1)
          \white(2,2)
          \white(3,1)
        \end{hexboard}$}
    }
    \quad
    G_5 = \lift{
      \scalebox{\scale}{$\begin{hexboard}
          \noshadows
          \foreach\i in {1,...,3} {\hex(1,\i)}
          \foreach\i in {1,...,3} {\hex(2,\i)}
          \foreach\i in {1,...,2} {\hex(3,\i)}
          \edge[\sw](1,1)(1,3)
          \edge[\ne\noacutecorner](3,2)(3,1)
          \edge[\nw](1,1)(3,1)
          \edge[\se\noacutecorner](2,3)(1,3)
          \black(3,2)
          \white(2,3)
          \white(1,1)
          \white(2,1)
          \black(3,1)
          \white(1,2)
          \black(1,3)
        \end{hexboard}$}
    }
    \quad
    G_6 = \lift{
      \scalebox{\scale}{$\begin{hexboard}
          \noshadows
          \foreach\i in {1,...,3} {\hex(1,\i)}
          \foreach\i in {1,...,3} {\hex(2,\i)}
          \foreach\i in {1,...,2} {\hex(3,\i)}
          \edge[\sw](1,1)(1,3)
          \edge[\ne\noacutecorner](3,2)(3,1)
          \edge[\nw](1,1)(3,1)
          \edge[\se\noacutecorner](2,3)(1,3)
          \black(3,2)
          \white(2,3)
          \white(1,1)
          \white(2,1)
          \white(1,2)
          \black(1,3)
        \end{hexboard}$}
    }
    \quad
    G_7 = \lift{
      \scalebox{\scale}{$\begin{hexboard}
          \noshadows
          \foreach\i in {1,...,3} {\hex(1,\i)}
          \foreach\i in {1,...,3} {\hex(2,\i)}
          \foreach\i in {1,...,2} {\hex(3,\i)}
          \edge[\sw](1,1)(1,3)
          \edge[\ne\noacutecorner](3,2)(3,1)
          \edge[\nw](1,1)(3,1)
          \edge[\se\noacutecorner](2,3)(1,3)
          \black(3,2)
          \white(2,3)
          \black(1,1)
          \black(1,2)
          \black(2,1)
          \white(3,1)
        \end{hexboard}$}
    }
    \]
    \caption{One-sided forks, realizing all passable game values over
      $\Lin_3$.}
    \label{fig-realizable-one-sided-fork}
  \end{figure}

\item Figure~\ref{fig-3-terminal-decisive-gap} shows a one-sided fork
  with an additional gap marked ``$\ast$'' between the two non-edge
  terminals. We consider the gap to be part of the region. The outcome
  poset for this type of region is the 4-element linearly ordered set
  $\Lin_4=\s{\bot,a,b,\top}$, with $\bot<a<b<\top$. The four outcomes
  are:
  \begin{itemize}
  \item $\bot$: White's board edges are connected.
  \item $a$: Neither player's board edges are connected, and White
    occupies the gap.
  \item $b$: Neither player's board edges are connected, and Black
    occupies the gap.
  \item $\top$: Black's board edges are connected.
  \end{itemize}
  \begin{figure}
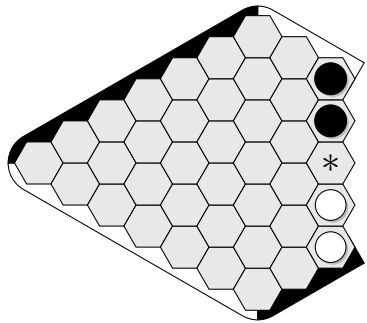

    \[
    \begin{hexboard}[scale=0.8]
      \shadows
      \foreach\i in {1,...,7} {\hex(1,\i)}
      \foreach\i in {1,...,7} {\hex(2,\i)}
      \foreach\i in {1,...,7} {\hex(3,\i)}
      \foreach\i in {1,...,6} {\hex(4,\i)}
      \foreach\i in {1,...,5} {\hex(5,\i)}
      \foreach\i in {1,...,4} {\hex(6,\i)}
      \foreach\i in {1,...,3} {\hex(7,\i)}
      \edge[\sw](1,1)(1,7)
      \edge[\ne\noacutecorner](7,3)(7,1)
      \edge[\nw](1,1)(7,1)
      \edge[\se\noacutecorner](3,7)(1,7)
      \black(7,3)
      \black(6,4)
      \cell(5,5)\label{\Large$\ast$}
      \white(4,6)
      \white(3,7)
    \end{hexboard}
    \]
    \caption{A 3-terminal region with four board edges and a gap, or
      ``one-sided fork with gap''.}
    \label{fig-3-terminal-decisive-gap}
  \end{figure}

  Recall from Figure~\ref{fig-l4} that there are 31 passable game
  values over $\Lin_4$. All of them are realizable in Hex by one-sided
  forks with gap, as shown in
  Figure~\ref{fig-realizable-3-terminal-decisive-gap}.
  \begin{figure}
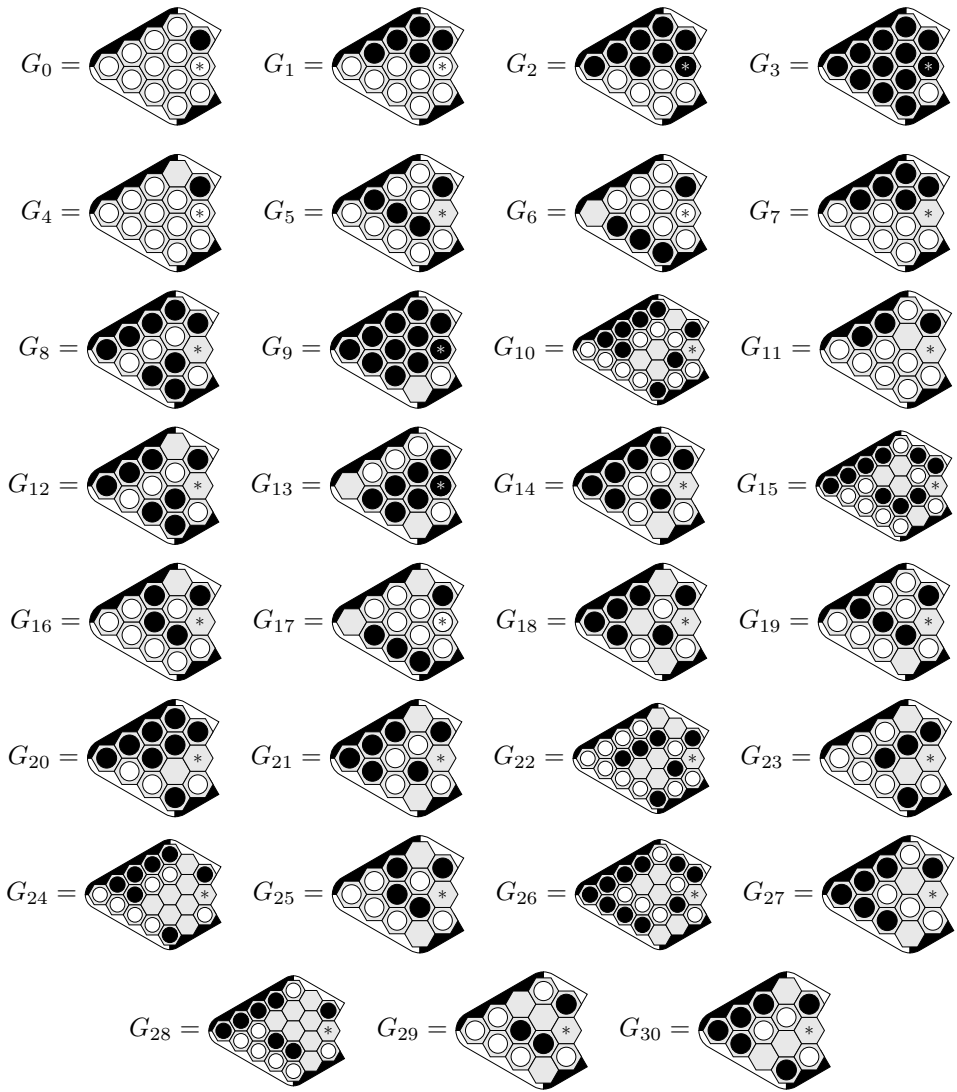

    \def\scale{0.5}
    \def\scalel{0.7667}
    \def\b#1{\makebox[1.5em][r]{$#1$}}
    \def\xspace{\hspace{-0.7ex}}
    \def\squad{\quad}
    \def\lift{\mp{0.45}}
    \[
    \b{G_{0}} =\xspace \lift{
      \scalebox{\scale}{$\begin{hexboard}
          \noshadows
      \foreach\i in {1,...,4} {\hex(1,\i)}
      \foreach\i in {1,...,4} {\hex(2,\i)}
      \foreach\i in {1,...,3} {\hex(3,\i)}
      \foreach\i in {1,...,2} {\hex(4,\i)}
      \edge[\sw](1,1)(1,4)
      \edge[\ne\noacutecorner](4,2)(4,1)
      \edge[\nw](1,1)(4,1)
      \edge[\se\noacutecorner](2,4)(1,4)
      \black(4,2)
      \white(2,4)
      \cell(3,3)\label{\Large$\ast$}
      \foreach\i in {1,...,4} {\white(1,\i)}
      \foreach\i in {1,...,3} {\white(2,\i)}
      \foreach\i in {1,...,3} {\white(3,\i)}
      \foreach\i in {1,...,1} {\white(4,\i)}
    \end{hexboard}$}
    }
    \squad
    \b{G_{1}} =\xspace \lift{
      \scalebox{\scale}{$\begin{hexboard}
          \noshadows
      \foreach\i in {1,...,4} {\hex(1,\i)}
      \foreach\i in {1,...,4} {\hex(2,\i)}
      \foreach\i in {1,...,3} {\hex(3,\i)}
      \foreach\i in {1,...,2} {\hex(4,\i)}
      \edge[\sw](1,1)(1,4)
      \edge[\ne\noacutecorner](4,2)(4,1)
      \edge[\nw](1,1)(4,1)
      \edge[\se\noacutecorner](2,4)(1,4)
      \black(4,2)
      \white(2,4)
      \cell(3,3)\label{\Large$\ast$}
      \foreach\i in {1,...,4} {\white(1,\i)}
      \foreach\i in {1,...,1} {\black(2,\i)}
      \foreach\i in {2,...,3} {\white(2,\i)}
      \foreach\i in {1,...,2} {\black(3,\i)}
      \foreach\i in {3,...,3} {\white(3,\i)}
      \foreach\i in {1,...,1} {\black(4,\i)}
    \end{hexboard}$}
    }
    \squad
    \b{G_{2}} =\xspace \lift{
      \scalebox{\scale}{$\begin{hexboard}
          \noshadows
      \foreach\i in {1,...,4} {\hex(1,\i)}
      \foreach\i in {1,...,4} {\hex(2,\i)}
      \foreach\i in {1,...,3} {\hex(3,\i)}
      \foreach\i in {1,...,2} {\hex(4,\i)}
      \edge[\sw](1,1)(1,4)
      \edge[\ne\noacutecorner](4,2)(4,1)
      \edge[\nw](1,1)(4,1)
      \edge[\se\noacutecorner](2,4)(1,4)
      \black(4,2)
      \white(2,4)
      \black(3,3)\label{\Large$\ast$}
      \foreach\i in {1,...,4} {\black(\i,1)}
      \foreach\i in {1,...,1} {\white(\i,2)}
      \foreach\i in {2,...,3} {\black(\i,2)}
      \foreach\i in {1,...,2} {\white(\i,3)}
      \foreach\i in {3,...,3} {\black(\i,3)}
      \foreach\i in {1,...,1} {\white(\i,4)}
    \end{hexboard}$}
    }
    \squad
    \b{G_{3}} =\xspace \lift{
      \scalebox{\scale}{$\begin{hexboard}
          \noshadows
      \foreach\i in {1,...,4} {\hex(1,\i)}
      \foreach\i in {1,...,4} {\hex(2,\i)}
      \foreach\i in {1,...,3} {\hex(3,\i)}
      \foreach\i in {1,...,2} {\hex(4,\i)}
      \edge[\sw](1,1)(1,4)
      \edge[\ne\noacutecorner](4,2)(4,1)
      \edge[\nw](1,1)(4,1)
      \edge[\se\noacutecorner](2,4)(1,4)
      \black(4,2)
      \white(2,4)
      \black(3,3)\label{\Large$\ast$}
      \foreach\i in {1,...,4} {\black(\i,1)}
      \foreach\i in {1,...,3} {\black(\i,2)}
      \foreach\i in {1,...,3} {\black(\i,3)}
      \foreach\i in {1,...,1} {\black(\i,4)}
    \end{hexboard}$}
    }
    \]
    \[
    \b{G_{4}} =\xspace \lift{
      \scalebox{\scale}{$\begin{hexboard}
          \noshadows
      \foreach\i in {1,...,4} {\hex(1,\i)}
      \foreach\i in {1,...,4} {\hex(2,\i)}
      \foreach\i in {1,...,3} {\hex(3,\i)}
      \foreach\i in {1,...,2} {\hex(4,\i)}
      \edge[\sw](1,1)(1,4)
      \edge[\ne\noacutecorner](4,2)(4,1)
      \edge[\nw](1,1)(4,1)
      \edge[\se\noacutecorner](2,4)(1,4)
      \black(4,2)
      \white(2,4)
      \cell(3,3)\label{\Large$\ast$}
      \foreach\i in {1,...,4} {\white(1,\i)}
      \foreach\i in {1,...,3} {\white(2,\i)}
      \foreach\i in {1,...,3} {\white(3,\i)}
    \end{hexboard}$}
    }
    \squad
    \b{G_{5}} =\xspace \lift{
      \scalebox{\scale}{$\begin{hexboard}
          \noshadows
      \foreach\i in {1,...,4} {\hex(1,\i)}
      \foreach\i in {1,...,4} {\hex(2,\i)}
      \foreach\i in {1,...,3} {\hex(3,\i)}
      \foreach\i in {1,...,2} {\hex(4,\i)}
      \edge[\sw](1,1)(1,4)
      \edge[\ne\noacutecorner](4,2)(4,1)
      \edge[\nw](1,1)(4,1)
      \edge[\se\noacutecorner](2,4)(1,4)
      \black(4,2)
      \white(2,4)
      \cell(3,3)\label{\Large$\ast$}
      \foreach\i in {1,...,4} {\white(1,\i)}
      \foreach\i in {1,...,3} {\black(2,\i)}
      \foreach\i in {1,...,2} {\white(3,\i)}
      \foreach\i in {1,...,1} {\white(4,\i)}
    \end{hexboard}$}
    }
    \squad
    \b{G_{6}} =\xspace \lift{
      \scalebox{\scale}{$\begin{hexboard}
          \noshadows
      \foreach\i in {1,...,4} {\hex(1,\i)}
      \foreach\i in {1,...,4} {\hex(2,\i)}
      \foreach\i in {1,...,3} {\hex(3,\i)}
      \foreach\i in {1,...,2} {\hex(4,\i)}
      \edge[\sw](1,1)(1,4)
      \edge[\ne\noacutecorner](4,2)(4,1)
      \edge[\nw](1,1)(4,1)
      \edge[\se\noacutecorner](2,4)(1,4)
      \black(4,2)
      \white(2,4)
      \white(3,3)\label{\Large$\ast$}
      \foreach\i in {2,...,4} {\black(1,\i)}
      \foreach\i in {1,...,2} {\white(2,\i)}
      \foreach\i in {1,...,2} {\white(3,\i)}
      \foreach\i in {1,...,1} {\white(4,\i)}
      \white(2,3)
    \end{hexboard}$}
    }
    \squad
    \b{G_{7}} =\xspace \lift{
      \scalebox{\scale}{$\begin{hexboard}
          \noshadows
      \foreach\i in {1,...,4} {\hex(1,\i)}
      \foreach\i in {1,...,4} {\hex(2,\i)}
      \foreach\i in {1,...,3} {\hex(3,\i)}
      \foreach\i in {1,...,2} {\hex(4,\i)}
      \edge[\sw](1,1)(1,4)
      \edge[\ne\noacutecorner](4,2)(4,1)
      \edge[\nw](1,1)(4,1)
      \edge[\se\noacutecorner](2,4)(1,4)
      \black(4,2)
      \white(2,4)
      \cell(3,3)\label{\Large$\ast$}
      \foreach\i in {1,...,4} {\white(1,\i)}
      \foreach\i in {1,...,1} {\black(2,\i)}
      \foreach\i in {2,...,3} {\white(2,\i)}
      \foreach\i in {1,...,2} {\black(3,\i)}
      \foreach\i in {1,...,1} {\black(4,\i)}
    \end{hexboard}$}
    }
    \]
    \[
    \b{G_{8}} =\xspace \lift{
      \scalebox{\scale}{$\begin{hexboard}
          \noshadows
      \foreach\i in {1,...,4} {\hex(1,\i)}
      \foreach\i in {1,...,4} {\hex(2,\i)}
      \foreach\i in {1,...,3} {\hex(3,\i)}
      \foreach\i in {1,...,2} {\hex(4,\i)}
      \edge[\sw](1,1)(1,4)
      \edge[\ne\noacutecorner](4,2)(4,1)
      \edge[\nw](1,1)(4,1)
      \edge[\se\noacutecorner](2,4)(1,4)
      \black(4,2)
      \white(2,4)
      \cell(3,3)\label{\Large$\ast$}
      \foreach\i in {1,...,4} {\black(\i,1)}
      \foreach\i in {1,...,3} {\white(\i,2)}
      \foreach\i in {1,...,2} {\black(\i,3)}
      \foreach\i in {1,...,1} {\black(\i,4)}
    \end{hexboard}$}
    }
    \squad
    \b{G_{9}} =\xspace \lift{
      \scalebox{\scale}{$\begin{hexboard}
          \noshadows
      \foreach\i in {1,...,4} {\hex(1,\i)}
      \foreach\i in {1,...,4} {\hex(2,\i)}
      \foreach\i in {1,...,3} {\hex(3,\i)}
      \foreach\i in {1,...,2} {\hex(4,\i)}
      \edge[\sw](1,1)(1,4)
      \edge[\ne\noacutecorner](4,2)(4,1)
      \edge[\nw](1,1)(4,1)
      \edge[\se\noacutecorner](2,4)(1,4)
      \black(4,2)
      \white(2,4)
      \black(3,3)\label{\Large$\ast$}
      \foreach\i in {1,...,4} {\black(\i,1)}
      \foreach\i in {1,...,3} {\black(\i,2)}
      \foreach\i in {1,...,3} {\black(\i,3)}
    \end{hexboard}$}
    }
    \squad
    \b{G_{10}} =\xspace \lift{
      \scalebox{\scale}{$\begin{hexboard}[scale=\scalel]
          \noshadows
      \foreach\i in {1,...,5} {\hex(1,\i)}
      \foreach\i in {1,...,5} {\hex(2,\i)}
      \foreach\i in {1,...,5} {\hex(3,\i)}
      \foreach\i in {1,...,4} {\hex(4,\i)}
      \foreach\i in {1,...,3} {\hex(5,\i)}
      \edge[\sw](1,1)(1,5)
      \edge[\ne\noacutecorner](5,3)(5,1)
      \edge[\nw](1,1)(5,1)
      \edge[\se\noacutecorner](3,5)(1,5)
      \black(5,3)
      \white(3,5)
      \cell(4,4)\label{\Large$\ast$}
      \foreach\i in {1,...,4} {\white(1,\i)}
      \black(1,5)
      \foreach\i in {1,...,2} {\black(2,\i)}
      \white(2,5)
      \black(3,1)
      \white(3,2)
      \black(3,4)
      \black(4,1)
      \white(4,2)
      \white(4,3)
      \black(5,1)
    \end{hexboard}$}
    }
    \squad
    \b{G_{11}} =\xspace \lift{
      \scalebox{\scale}{$\begin{hexboard}
          \noshadows
      \foreach\i in {1,...,4} {\hex(1,\i)}
      \foreach\i in {1,...,4} {\hex(2,\i)}
      \foreach\i in {1,...,3} {\hex(3,\i)}
      \foreach\i in {1,...,2} {\hex(4,\i)}
      \edge[\sw](1,1)(1,4)
      \edge[\ne\noacutecorner](4,2)(4,1)
      \edge[\nw](1,1)(4,1)
      \edge[\se\noacutecorner](2,4)(1,4)
      \black(4,2)
      \white(2,4)
      \cell(3,3)\label{\Large$\ast$}
      \foreach\i in {1,...,4} {\white(1,\i)}
      \foreach\i in {1,...,1} {\black(2,\i)}
      \foreach\i in {2,...,3} {\white(2,\i)}
      \foreach\i in {1,...,1} {\black(3,\i)}
      \foreach\i in {1,...,1} {\white(4,\i)}
    \end{hexboard}$}
    }
    \]
    \[
    \b{G_{12}} =\xspace \lift{
      \scalebox{\scale}{$\begin{hexboard}
          \noshadows
      \foreach\i in {1,...,4} {\hex(1,\i)}
      \foreach\i in {1,...,4} {\hex(2,\i)}
      \foreach\i in {1,...,3} {\hex(3,\i)}
      \foreach\i in {1,...,2} {\hex(4,\i)}
      \edge[\sw](1,1)(1,4)
      \edge[\ne\noacutecorner](4,2)(4,1)
      \edge[\nw](1,1)(4,1)
      \edge[\se\noacutecorner](2,4)(1,4)
      \black(4,2)
      \white(2,4)
      \cell(3,3)\label{\Large$\ast$}
      \foreach\i in {1,...,3} {\black(\i,1)}
      \foreach\i in {1,...,3} {\white(\i,2)}
      \foreach\i in {1,...,2} {\black(\i,3)}
      \foreach\i in {1,...,1} {\black(\i,4)}
    \end{hexboard}$}
    }
    \squad
    \b{G_{13}} =\xspace \lift{
      \scalebox{\scale}{$\begin{hexboard}
          \noshadows
      \foreach\i in {1,...,4} {\hex(1,\i)}
      \foreach\i in {1,...,4} {\hex(2,\i)}
      \foreach\i in {1,...,3} {\hex(3,\i)}
      \foreach\i in {1,...,2} {\hex(4,\i)}
      \edge[\sw](1,1)(1,4)
      \edge[\ne\noacutecorner](4,2)(4,1)
      \edge[\nw](1,1)(4,1)
      \edge[\se\noacutecorner](2,4)(1,4)
      \black(4,2)
      \white(2,4)
      \black(3,3)\label{\Large$\ast$}
      \foreach\i in {2,...,4} {\white(\i,1)}
      \foreach\i in {1,...,3} {\black(\i,2)}
      \foreach\i in {1,...,3} {\black(\i,3)}
    \end{hexboard}$}
    }
    \squad
    \b{G_{14}} =\xspace \lift{
      \scalebox{\scale}{$\begin{hexboard}
          \noshadows
      \foreach\i in {1,...,4} {\hex(1,\i)}
      \foreach\i in {1,...,4} {\hex(2,\i)}
      \foreach\i in {1,...,3} {\hex(3,\i)}
      \foreach\i in {1,...,2} {\hex(4,\i)}
      \edge[\sw](1,1)(1,4)
      \edge[\ne\noacutecorner](4,2)(4,1)
      \edge[\nw](1,1)(4,1)
      \edge[\se\noacutecorner](2,4)(1,4)
      \black(4,2)
      \white(2,4)
      \cell(3,3)\label{\Large$\ast$}
      \foreach\i in {1,...,4} {\black(\i,1)}
      \black(1,2)
      \white(2,2)
      \white(3,2)
      \white(1,3)
      \black(2,3)
    \end{hexboard}$}
    }
    \squad
    \b{G_{15}} =\xspace \lift{
      \scalebox{\scale}{$\begin{hexboard}[scale=\scalel]
          \noshadows
      \foreach\i in {1,...,5} {\hex(1,\i)}
      \foreach\i in {1,...,5} {\hex(2,\i)}
      \foreach\i in {1,...,5} {\hex(3,\i)}
      \foreach\i in {1,...,4} {\hex(4,\i)}
      \foreach\i in {1,...,3} {\hex(5,\i)}
      \edge[\sw](1,1)(1,5)
      \edge[\ne\noacutecorner](5,3)(5,1)
      \edge[\nw](1,1)(5,1)
      \edge[\se\noacutecorner](3,5)(1,5)
      \black(5,3)
      \white(3,5)
      \cell(4,4)\label{\Large$\ast$}
      \foreach\i in {1,...,4} {\black(\i,1)}
      \white(5,1)
      \foreach\i in {1,...,2} {\white(\i,2)}
      \black(5,2)
      \white(1,3)
      \black(2,3)
      \white(4,3)
      \white(1,4)
      \black(2,4)
      \black(3,4)
      \white(1,5)
    \end{hexboard}$}
    }
    \]
    \[
    \b{G_{16}} =\xspace \lift{
      \scalebox{\scale}{$\begin{hexboard}
          \noshadows
      \foreach\i in {1,...,4} {\hex(1,\i)}
      \foreach\i in {1,...,4} {\hex(2,\i)}
      \foreach\i in {1,...,3} {\hex(3,\i)}
      \foreach\i in {1,...,2} {\hex(4,\i)}
      \edge[\sw](1,1)(1,4)
      \edge[\ne\noacutecorner](4,2)(4,1)
      \edge[\nw](1,1)(4,1)
      \edge[\se\noacutecorner](2,4)(1,4)
      \black(4,2)
      \white(2,4)
      \cell(3,3)\label{\Large$\ast$}
      \foreach\i in {1,...,4} {\white(1,\i)}
      \white(2,1)
      \black(2,2)
      \black(2,3)
      \black(3,1)
      \white(3,2)
    \end{hexboard}$}
    }
    \squad
    \b{G_{17}} =\xspace \lift{
      \scalebox{\scale}{$\begin{hexboard}
          \noshadows
      \foreach\i in {1,...,4} {\hex(1,\i)}
      \foreach\i in {1,...,4} {\hex(2,\i)}
      \foreach\i in {1,...,3} {\hex(3,\i)}
      \foreach\i in {1,...,2} {\hex(4,\i)}
      \edge[\sw](1,1)(1,4)
      \edge[\ne\noacutecorner](4,2)(4,1)
      \edge[\nw](1,1)(4,1)
      \edge[\se\noacutecorner](2,4)(1,4)
      \black(4,2)
      \white(2,4)
      \white(3,3)\label{\Large$\ast$}
      \foreach\i in {2,...,4} {\black(1,\i)}
      \foreach\i in {1,...,3} {\white(2,\i)}
      \foreach\i in {1,...,3} {\white(3,\i)}
    \end{hexboard}$}
    }
    \squad
    \b{G_{18}} =\xspace \lift{
      \scalebox{\scale}{$\begin{hexboard}
          \noshadows
      \foreach\i in {1,...,4} {\hex(1,\i)}
      \foreach\i in {1,...,4} {\hex(2,\i)}
      \foreach\i in {1,...,3} {\hex(3,\i)}
      \foreach\i in {1,...,2} {\hex(4,\i)}
      \edge[\sw](1,1)(1,4)
      \edge[\ne\noacutecorner](4,2)(4,1)
      \edge[\nw](1,1)(4,1)
      \edge[\se\noacutecorner](2,4)(1,4)
      \black(4,2)
      \white(2,4)
      \cell(3,3)\label{\Large$\ast$}
      \black(1,1)
      \black(1,2)
      \white(1,3)
      \black(2,1)
      \black(2,3)
      \black(3,1)
      \white(3,2)
    \end{hexboard}$}
    }
    \squad
    \b{G_{19}} =\xspace \lift{
      \scalebox{\scale}{$\begin{hexboard}
          \noshadows
      \foreach\i in {1,...,4} {\hex(1,\i)}
      \foreach\i in {1,...,4} {\hex(2,\i)}
      \foreach\i in {1,...,3} {\hex(3,\i)}
      \foreach\i in {1,...,2} {\hex(4,\i)}
      \edge[\sw](1,1)(1,4)
      \edge[\ne\noacutecorner](4,2)(4,1)
      \edge[\nw](1,1)(4,1)
      \edge[\se\noacutecorner](2,4)(1,4)
      \black(4,2)
      \white(2,4)
      \cell(3,3)\label{\Large$\ast$}
      \foreach\i in {1,...,3} {\white(1,\i)}
      \foreach\i in {1,...,3} {\black(2,\i)}
      \foreach\i in {1,...,2} {\white(3,\i)}
      \foreach\i in {1,...,1} {\white(4,\i)}
    \end{hexboard}$}
    }
    \]
    \[
    \b{G_{20}} =\xspace \lift{
      \scalebox{\scale}{$\begin{hexboard}
          \noshadows
      \foreach\i in {1,...,4} {\hex(1,\i)}
      \foreach\i in {1,...,4} {\hex(2,\i)}
      \foreach\i in {1,...,3} {\hex(3,\i)}
      \foreach\i in {1,...,2} {\hex(4,\i)}
      \edge[\sw](1,1)(1,4)
      \edge[\ne\noacutecorner](4,2)(4,1)
      \edge[\nw](1,1)(4,1)
      \edge[\se\noacutecorner](2,4)(1,4)
      \black(4,2)
      \white(2,4)
      \cell(3,3)\label{\Large$\ast$}
      \foreach\i in {1,...,4} {\black(\i,1)}
      \foreach\i in {1,...,1} {\white(\i,2)}
      \foreach\i in {2,...,3} {\black(\i,2)}
      \foreach\i in {1,...,1} {\white(\i,3)}
      \foreach\i in {1,...,1} {\black(\i,4)}
    \end{hexboard}$}
    }
    \squad
    \b{G_{21}} =\xspace \lift{
      \scalebox{\scale}{$\begin{hexboard}
          \noshadows
      \foreach\i in {1,...,4} {\hex(1,\i)}
      \foreach\i in {1,...,4} {\hex(2,\i)}
      \foreach\i in {1,...,3} {\hex(3,\i)}
      \foreach\i in {1,...,2} {\hex(4,\i)}
      \edge[\sw](1,1)(1,4)
      \edge[\ne\noacutecorner](4,2)(4,1)
      \edge[\nw](1,1)(4,1)
      \edge[\se\noacutecorner](2,4)(1,4)
      \black(4,2)
      \white(2,4)
      \cell(3,3)\label{\Large$\ast$}
      \black(1,1)
      \black(1,2)
      \white(1,3)
      \black(2,1)
      \white(2,2)
      \black(2,3)
      \black(3,1)
      \white(3,2)
    \end{hexboard}$}
    }
    \squad
    \b{G_{22}} =\xspace \lift{
      \scalebox{\scale}{$\begin{hexboard}[scale=\scalel]
          \noshadows
      \foreach\i in {1,...,5} {\hex(1,\i)}
      \foreach\i in {1,...,5} {\hex(2,\i)}
      \foreach\i in {1,...,5} {\hex(3,\i)}
      \foreach\i in {1,...,4} {\hex(4,\i)}
      \foreach\i in {1,...,3} {\hex(5,\i)}
      \edge[\sw](1,1)(1,5)
      \edge[\ne\noacutecorner](5,3)(5,1)
      \edge[\nw](1,1)(5,1)
      \edge[\se\noacutecorner](3,5)(1,5)
      \black(5,3)
      \white(3,5)
      \cell(4,4)\label{\Large$\ast$}
      \foreach\i in {1,...,4} {\white(1,\i)}
      \black(1,5)
      \white(2,1)
      \black(2,2)
      \white(2,5)
      \white(3,1)
      \black(3,2)
      \black(3,4)
      \white(4,1)
      \black(4,2)
      \white(4,3)
    \end{hexboard}$}
    }
    \squad
    \b{G_{23}} =\xspace \lift{
      \scalebox{\scale}{$\begin{hexboard}
          \noshadows
      \foreach\i in {1,...,4} {\hex(1,\i)}
      \foreach\i in {1,...,4} {\hex(2,\i)}
      \foreach\i in {1,...,3} {\hex(3,\i)}
      \foreach\i in {1,...,2} {\hex(4,\i)}
      \edge[\sw](1,1)(1,4)
      \edge[\ne\noacutecorner](4,2)(4,1)
      \edge[\nw](1,1)(4,1)
      \edge[\se\noacutecorner](2,4)(1,4)
      \black(4,2)
      \white(2,4)
      \cell(3,3)\label{\Large$\ast$}
      \foreach\i in {1,...,3} {\white(\i,1)}
      \foreach\i in {1,...,1} {\white(\i,2)}
      \foreach\i in {2,...,3} {\black(\i,2)}
      \foreach\i in {1,...,1} {\white(\i,3)}
      \foreach\i in {1,...,1} {\black(\i,4)}
    \end{hexboard}$}
    }
    \]
    \[
    \b{G_{24}} =\xspace \lift{
      \scalebox{\scale}{$\begin{hexboard}[scale=\scalel]
          \noshadows
      \foreach\i in {1,...,5} {\hex(1,\i)}
      \foreach\i in {1,...,5} {\hex(2,\i)}
      \foreach\i in {1,...,5} {\hex(3,\i)}
      \foreach\i in {1,...,4} {\hex(4,\i)}
      \foreach\i in {1,...,3} {\hex(5,\i)}
      \edge[\sw](1,1)(1,5)
      \edge[\ne\noacutecorner](5,3)(5,1)
      \edge[\nw](1,1)(5,1)
      \edge[\se\noacutecorner](3,5)(1,5)
      \black(5,3)
      \white(3,5)
      \cell(4,4)\label{\Large$\ast$}
      \foreach\i in {1,...,4} {\white(1,\i)}
      \black(1,5)
      \black(2,1)
      \black(2,2)
      \black(3,1)
      \white(3,2)
      \black(4,1)
      \white(4,2)
      \black(5,1)
    \end{hexboard}$}
    }
    \squad
    \b{G_{25}} =\xspace \lift{
      \scalebox{\scale}{$\begin{hexboard}
          \noshadows
      \foreach\i in {1,...,4} {\hex(1,\i)}
      \foreach\i in {1,...,4} {\hex(2,\i)}
      \foreach\i in {1,...,3} {\hex(3,\i)}
      \foreach\i in {1,...,2} {\hex(4,\i)}
      \edge[\sw](1,1)(1,4)
      \edge[\ne\noacutecorner](4,2)(4,1)
      \edge[\nw](1,1)(4,1)
      \edge[\se\noacutecorner](2,4)(1,4)
      \black(4,2)
      \white(2,4)
      \cell(3,3)\label{\Large$\ast$}
      \white(1,1)
      \white(2,1)
      \black(3,1)
      \white(1,2)
      \black(2,2)
      \white(3,2)
      \white(1,3)
      \black(2,3)
    \end{hexboard}$}
    }
    \squad
    \b{G_{26}} =\xspace \lift{
      \scalebox{\scale}{$\begin{hexboard}[scale=\scalel]
          \noshadows
      \foreach\i in {1,...,5} {\hex(1,\i)}
      \foreach\i in {1,...,5} {\hex(2,\i)}
      \foreach\i in {1,...,5} {\hex(3,\i)}
      \foreach\i in {1,...,4} {\hex(4,\i)}
      \foreach\i in {1,...,3} {\hex(5,\i)}
      \edge[\sw](1,1)(1,5)
      \edge[\ne\noacutecorner](5,3)(5,1)
      \edge[\nw](1,1)(5,1)
      \edge[\se\noacutecorner](3,5)(1,5)
      \black(5,3)
      \white(3,5)
      \cell(4,4)\label{\Large$\ast$}
      \foreach\i in {1,...,4} {\black(\i,1)}
      \white(5,1)
      \black(1,2)
      \white(2,2)
      \black(5,2)
      \black(1,3)
      \white(2,3)
      \white(4,3)
      \black(1,4)
      \white(2,4)
      \black(3,4)
    \end{hexboard}$}
    }
    \squad
    \b{G_{27}} =\xspace \lift{
      \scalebox{\scale}{$\begin{hexboard}
          \noshadows
      \foreach\i in {1,...,4} {\hex(1,\i)}
      \foreach\i in {1,...,4} {\hex(2,\i)}
      \foreach\i in {1,...,3} {\hex(3,\i)}
      \foreach\i in {1,...,2} {\hex(4,\i)}
      \edge[\sw](1,1)(1,4)
      \edge[\ne\noacutecorner](4,2)(4,1)
      \edge[\nw](1,1)(4,1)
      \edge[\se\noacutecorner](2,4)(1,4)
      \black(4,2)
      \white(2,4)
      \cell(3,3)\label{\Large$\ast$}
      \foreach\i in {1,...,3} {\black(1,\i)}
      \foreach\i in {1,...,1} {\black(2,\i)}
      \foreach\i in {2,...,3} {\white(2,\i)}
      \foreach\i in {1,...,1} {\black(3,\i)}
      \foreach\i in {1,...,1} {\white(4,\i)}
    \end{hexboard}$}
    }
    \]
    \[
    \b{G_{28}} =\xspace \lift{
      \scalebox{\scale}{$\begin{hexboard}[scale=\scalel]
          \noshadows
      \foreach\i in {1,...,5} {\hex(1,\i)}
      \foreach\i in {1,...,5} {\hex(2,\i)}
      \foreach\i in {1,...,5} {\hex(3,\i)}
      \foreach\i in {1,...,4} {\hex(4,\i)}
      \foreach\i in {1,...,3} {\hex(5,\i)}
      \edge[\sw](1,1)(1,5)
      \edge[\ne\noacutecorner](5,3)(5,1)
      \edge[\nw](1,1)(5,1)
      \edge[\se\noacutecorner](3,5)(1,5)
      \black(5,3)
      \white(3,5)
      \cell(4,4)\label{\Large$\ast$}
      \foreach\i in {1,...,4} {\black(\i,1)}
      \white(5,1)
      \white(1,2)
      \white(2,2)
      \white(1,3)
      \black(2,3)
      \white(1,4)
      \black(2,4)
      \white(1,5)
    \end{hexboard}$}
    }
    \squad
    \b{G_{29}} =\xspace \lift{
      \scalebox{\scale}{$\begin{hexboard}
          \noshadows
      \foreach\i in {1,...,4} {\hex(1,\i)}
      \foreach\i in {1,...,4} {\hex(2,\i)}
      \foreach\i in {1,...,3} {\hex(3,\i)}
      \foreach\i in {1,...,2} {\hex(4,\i)}
      \edge[\sw](1,1)(1,4)
      \edge[\ne\noacutecorner](4,2)(4,1)
      \edge[\nw](1,1)(4,1)
      \edge[\se\noacutecorner](2,4)(1,4)
      \black(4,2)
      \white(2,4)
      \cell(3,3)\label{\Large$\ast$}
      \white(1,1)
      \white(2,1)
      \white(4,1)
      \white(1,2)
      \black(2,2)
      \white(1,3)
      \black(2,3)
    \end{hexboard}$}
    }
    \squad
    \b{G_{30}} =\xspace \lift{
      \scalebox{\scale}{$\begin{hexboard}
          \noshadows
      \foreach\i in {1,...,4} {\hex(1,\i)}
      \foreach\i in {1,...,4} {\hex(2,\i)}
      \foreach\i in {1,...,3} {\hex(3,\i)}
      \foreach\i in {1,...,2} {\hex(4,\i)}
      \edge[\sw](1,1)(1,4)
      \edge[\ne\noacutecorner](4,2)(4,1)
      \edge[\nw](1,1)(4,1)
      \edge[\se\noacutecorner](2,4)(1,4)
      \black(4,2)
      \white(2,4)
      \cell(3,3)\label{\Large$\ast$}
      \black(1,1)
      \black(1,2)
      \black(1,4)
      \black(2,1)
      \white(2,2)
      \black(3,1)
      \white(3,2)
    \end{hexboard}$}
    }
    \]
    \caption{One-sided forks with gap, realizing all passable game
      values over $\Lin_4$.}
    \label{fig-realizable-3-terminal-decisive-gap}
  \end{figure}
\end{itemize}

\begin{remark}
  While the values of all 31 Hex positions shown in
  Figure~\ref{fig-realizable-3-terminal-decisive-gap} are distinct
  when viewed as abstract combinatorial games, Demer pointed out that
  several of them are actually equivalent {\em as Hex positions}, due
  to the global decisiveness of $\top$ and $\bot$ in these positions
  (see Section~\ref{ssec-global}). Specifically, we have
  \[ G_{21}\eqg G_{22},
  \quad
  G_{10}\eqg G_{29},
  \quad
  G_{30}\eqg G_{15},
  \quad
  G_{26}\eqg G_{25}.
  \]
  For example, the values of $G_{10}$ and $G_{29}$ are
  $\g{b,\g{\top|a}|\bot}$ and $\g{\g{\top|\g{b|\g{a|\bot}}}|\bot}$,
  respectively. These were shown to be equivalent under global
  decisiveness in Example~\ref{exa-g-counterex} and
  Remark~\ref{rem-g-counterex}. The games $G_{15}$ and $G_{30}$ are
  their duals. The remaining equivalences can be shown by a similar
  argument. Perhaps surprisingly, apart from the four equivalences
  shown above, there are no additional collapses of the order under
  global decisiveness. In other words, among the 31 games of
  Figure~\ref{fig-l4} or
  Figure~\ref{fig-realizable-3-terminal-decisive-gap}, there are no
  pairs such that $G_i\leqg G_j$ but $G_i\nleq G_j$, except for
  $G_{22}\leqg G_{21}$, $G_{29}\leqg G_{10}$, $G_{15}\leqg G_{30}$,
  and $G_{25}\leqg G_{26}$.
\end{remark}

We also note that if two Hex positions realize combinatorial game
values that are inequivalent, or even inequivalent up to global
decisiveness, it does not necessarily follow that they are inequivalent
as Hex positions. Since equivalence of Hex positions in a region is
defined relative to all possible ways of embedding the region in a
larger {\em Hex} position, it is possible that equivalence of Hex
positions is a strictly coarser relation than equivalence of abstract
game values. However, apart from the phenomenon of global
decisiveness, which is not specific to Hex, no examples of this are
currently known.

\subsection{Unrealizable Game Values}\label{ssec-unrealizable}

So far, we have seen several settings in which all abstract passable
games over a given outcome poset were realizable as Hex positions. The
following proposition shows that this is not true in general.

\begin{proposition}\label{prop-abstract-not-hex}
  Consider 4-terminal Hex positions. As usual, let $\top$ denote the
  outcome where Black connects all terminals, and let $\bot$ denote
  the outcome where Black connects no terminals. The abstract game
  value $\Star=\g{\top|\bot}$, which is passable, cannot be realized
  by any 4-terminal Hex position.
\end{proposition}

The proof relies on the following lemma.

\begin{lemma}\label{lem-connected-components}
  Consider a Hex region that is completely filled with black and white
  stones. Changing the color of a single stone from black to white can
  increase the number of Black's connected components by at most 2.
\end{lemma}

\begin{proof}
  Let $x$ be the cell whose color is being changed. Before the color
  change, $x$ belongs to one black connected component. Since $x$ is a
  hexagon, it has at most 6 neighbors; therefore, after the color
  change, at most 3 black connected components can be adjacent to
  $x$. Since all black connected components that are not adjacent to
  $x$ are unaffected by the color change, this shows that the total
  number of black connected components increases by at most 2.
\end{proof}

\begin{proof}[Proof of Proposition~\ref{prop-abstract-not-hex}]
  Suppose, for the sake of contradiction, that there is some
  4-terminal position with value $\g{\top|\bot}$. Then Black has a
  first-player strategy that allows Black to connect all 4 terminals.
  Now suppose White goes first and plays in some cell $x$. Black can
  simply ignore White's move and follow Black's original strategy. (If
  the strategy ever calls for Black to play at $x$, Black can simply
  pass, or arbitrarily play elsewhere.) This will result in an outcome
  where all 4 terminals would be connected if $x$ was occupied by a
  black stone. By Lemma~\ref{lem-connected-components}, changing a
  single black stone to white can break this connected component into
  at most 3 parts, which implies that at least 2 of Black's terminals
  are still connected. Therefore, White cannot achieve outcome $\bot$,
  contradicting our initial assumption.
\end{proof}

Using essentially the same proof, we can easily derive generalizations
of Proposition~\ref{prop-abstract-not-hex}, such as the following.

\begin{proposition}\label{prop-n-terminal-limit}
  Suppose Black has a strategy that allows Black to connect all $n$
  terminals in an $n$-terminal position. If White is given $k$ free
  moves, then Black still has a strategy for connecting at least
  $n-2k$ terminals.
\end{proposition}

Results like Propositions~\ref{prop-abstract-not-hex} and
{\ref{prop-n-terminal-limit}} suggest that in Hex, there is some kind
of limit on the amount of advantage a player can gain with a single
move. In combinatorial game theory terms, while Hex is a ``hot'' game
(players never prefer passing to making a move), it is not a ``very
hot'' game (there is an inherent limit on how much a single move can
achieve).

In Proposition~\ref{prop-infinite}, we saw that the class of passable
abstract games over most outcome posets is infinite. However, in light
of Proposition~\ref{prop-abstract-not-hex}, we know that in general,
not all abstract game values are realizable as Hex positions. This
leaves open the question whether certain types of Hex regions (such as
3-terminal regions) admit finitely or infinitely many non-equivalent
Hex-realizable combinatorial values. In fact, since the counterexample
of Proposition~\ref{prop-abstract-not-hex} requires at least four
terminals, it is not even currently known whether there exists an
abstract passable 3-terminal value that is not realizable as a
3-terminal Hex position. One of the simplest passable 3-terminal
values for which no Hex realization is currently known is
$\g{a,\g{\top|b}|\g{a|\bot},b}$.

Note that Lemma~\ref{lem-connected-components}, and therefore
Propositions~\ref{prop-abstract-not-hex} and
{\ref{prop-n-terminal-limit}}, are specific to Hex, and do not apply
to other monotone set coloring games, or even planar connection
games. For example, on a game board that contains an octagonal cell,
it is trivially possible to find a 4-terminal position with value
$\g{\top|\bot}$, namely, a single empty octagon surrounded by
alternating black and white stones.

This leaves open another question, namely, whether all passable game
values can be realized by planar connection games. Failing that, we
may ask whether they are at least realizable by vertex Shannon games,
or failing that, by arbitrary monotone set coloring games. While the
first two questions remain open, the third question was recently
answered in the affirmative: all (short) passable game values are
realizable as monotone set coloring games {\cite{DSW2021-gadgets}}.

\section{Application: Minimal Connecting Sets in \texorpdfstring{$k\times n$}{k^^c3^^97n} Hex}
\label{sec-k-by-n}

The theory developed in this paper has several potential applications
to Hex, and some of these will be discussed in subsequent papers.
Here, we will briefly discuss one such application.

\subsection{Problem Statement}

Consider a Hex board of size $k\times n$, as shown for $k=4$ and
$n=12$ in Figure~\ref{fig-4xn}(a).
\begin{figure}
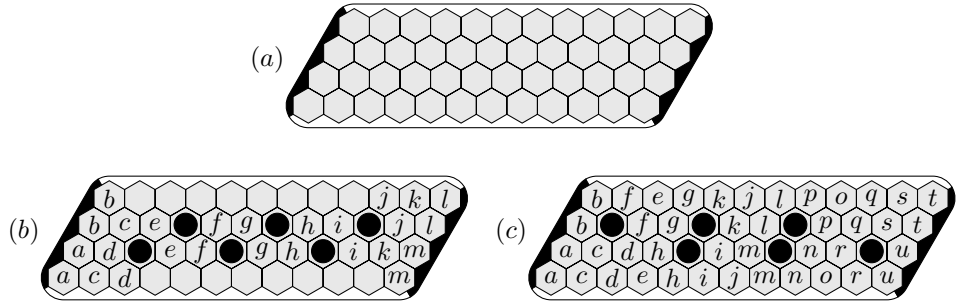

  \def\scale{0.58}
  \[
  (a) \m{$
    \begin{hexboard}[scale=\scale]
      \rotation{30}
      \board(4,12)
    \end{hexboard}
    $}
  \]
  \vspace{1em}
  \[
  (b) \m{$
    \begin{hexboard}[scale=\scale]
      \rotation{30}
      \board(4,12)
      \black(2,3)
      \black(3,4)
      \black(2,6)
      \black(3,7)
      \black(2,9)
      \black(3,10)
      \cell(1,1)\label{$a$}
      \cell(2,1)\label{$a$}
      \cell(3,1)\label{$b$}
      \cell(4,1)\label{$b$}
      \cell(1,2)\label{$c$}
      \cell(3,2)\label{$c$}
      \cell(1,3)\label{$d$}
      \cell(2,2)\label{$d$}
      \cell(3,3)\label{$e$}
      \cell(2,4)\label{$e$}
      \cell(2,5)\label{$f$}
      \cell(3,5)\label{$f$}
      \cell(3,6)\label{$g$}
      \cell(2,7)\label{$g$}
      \cell(2,8)\label{$h$}
      \cell(3,8)\label{$h$}
      \cell(3,9)\label{$i$}
      \cell(2,10)\label{$i$}
      \cell(4,10)\label{$j$}
      \cell(3,11)\label{$j$}
      \cell(2,11)\label{$k$}
      \cell(4,11)\label{$k$}
      \cell(3,12)\label{$l$}
      \cell(4,12)\label{$l$}
      \cell(1,12)\label{$m$}
      \cell(2,12)\label{$m$}
    \end{hexboard}
    $}
  \quad
  (c) \m{$
    \begin{hexboard}[scale=\scale]
      \rotation{30}
      \board(4,12)
      \black(3,2)
      \black(2,5)
      \black(3,5)
      \black(2,8)
      \black(3,8)
      \black(2,11)
      \cell(1,1)\label{$a$}
      \cell(2,1)\label{$a$}
      \cell(3,1)\label{$b$}
      \cell(4,1)\label{$b$}
      \cell(1,2)\label{$c$}
      \cell(2,2)\label{$c$}
      \cell(1,3)\label{$d$}
      \cell(2,3)\label{$d$}
      \cell(4,3)\label{$e$}
      \cell(1,4)\label{$e$}
      \cell(4,2)\label{$f$}
      \cell(3,3)\label{$f$}
      \cell(4,4)\label{$g$}
      \cell(3,4)\label{$g$}
      \cell(1,5)\label{$h$}
      \cell(2,4)\label{$h$}
      \cell(1,6)\label{$i$}
      \cell(2,6)\label{$i$}
      \cell(4,6)\label{$j$}
      \cell(1,7)\label{$j$}
      \cell(4,5)\label{$k$}
      \cell(3,6)\label{$k$}
      \cell(4,7)\label{$l$}
      \cell(3,7)\label{$l$}
      \cell(1,8)\label{$m$}
      \cell(2,7)\label{$m$}
      \cell(1,9)\label{$n$}
      \cell(2,9)\label{$n$}
      \cell(4,9)\label{$o$}
      \cell(1,10)\label{$o$}
      \cell(4,8)\label{$p$}
      \cell(3,9)\label{$p$}
      \cell(4,10)\label{$q$}
      \cell(3,10)\label{$q$}
      \cell(1,11)\label{$r$}
      \cell(2,10)\label{$r$}
      \cell(4,11)\label{$s$}
      \cell(3,11)\label{$s$}
      \cell(4,12)\label{$t$}
      \cell(3,12)\label{$t$}
      \cell(2,12)\label{$u$}
      \cell(1,12)\label{$u$}
    \end{hexboard}
    $}
  \]
  \caption{(a) A Hex board of size $4\times 12$. (b) and (c): two
    virtual connections for Black.}
  \label{fig-4xn}
\end{figure}
It is well-known that White has a first-player winning strategy when
$n\geq k$ and a second-player winning strategy when $n>k$
{\cite{Nash,Gardner}}. The question we would like to answer is: what is the
minimal number of black stones that must be placed on the board to
create a virtual connection between Black's edges?  Here, by ``virtual
connection'', we mean that Black will win the game even if White moves
first, or equivalently, that the game has combinatorial value
$\top$. Two examples of such virtual connections are shown in
Figures~\ref{fig-4xn}(b) and (c). Both of these virtual connections use
6 stones. We will later prove that this is the minimal number required
for a board of size $4\times 12$.

To see that the patterns of stones in Figures~\ref{fig-4xn}(b) and (c)
are indeed virtual connections, it suffices to consider a simple
pairing strategy. Whenever White moves in a cell that is labelled with
a letter, Black moves in the other cell with the same letter. It is
not hard to see that this strategy guarantees a Black connection. We
note that although both patterns use 6 stones, the pattern in
Figure~\ref{fig-4xn}(b) requires less ``space'': the unlabeled cells
are not required for Black's connection, and may as well be occupied
by White. On the other hand, experienced Hex players will have no
trouble seeing that if any of the labelled cells in
Figures~\ref{fig-4xn}(b) and (c) are occupied by White, Black no longer
has a virtual connection.  We also note that the patterns shown in
Figures~\ref{fig-4xn}(b) and (c) both generalize to larger board sizes;
obvious continuations of these patterns work for boards of size
$4\times 15$, $4\times 18$, $4\times 21$, and so on.

Note that we are not claiming that {\em every} virtual connection
admits a simple pairing strategy. We are merely saying that this is
the case for the ones shown in Figures~\ref{fig-4xn}(b) and (c).

So how would one go about proving minimality? More importantly, how
would one do this for, say, all boards of size $k\times n$, for fixed
$k$ and arbitrary $n$? Combinatorial game theory is the perfect tool
for this job. We will illustrate the method for $k=4$, but it also
works for other fixed values of $k$. (However, the computations get
exponentially harder when $k$ increases.)

\subsection{Open Regions}

We start by considering {\em open regions} of height $4$. By this, we
mean a region that includes the left, top, and bottom edges, but not
the right edge, as shown in Figure~\ref{fig-open-region}.
\begin{figure}
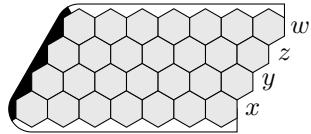

  \def\scale{0.6}
  \[
  \m{$
    \begin{hexboard}[scale=\scale]
      \rotation{30}
      \foreach\i in {1,...,7} {
        \foreach\j in {1,...,4} {
          \hex(\j,\i)
        }
      }
      \edge[\nw](1,1)(4,1)
      \edge[\sw\noobtusecorner](1,1)(1,7)
      \edge[\ne\noacutecorner](4,7)(4,1)
      \cell(1,8)\label{$x$}
      \cell(2,8)\label{$y$}
      \cell(3,8)\label{$z$}
      \cell(4,8)\label{$w$}
    \end{hexboard}
    $}
  \]
  \caption{An open region of height $4$.}
  \label{fig-open-region}
\end{figure}
For an open region that is completely filled with stones, the outcome
class is determined by the following information: whether White's
edges are connected to each other (outcome $\bot$), and if they are
not connected: which of the cells adjacent to the region (labelled
$x,y,z,w$ in Figure~\ref{fig-open-region}) are connected to the left
edge by a black chain, and which such cells are connected to each
other by a black chain. It turns out that for open regions of height
$4$, there are exactly 10 distinct outcomes, representatives of which
are shown in Figure~\ref{fig-open-region-outcomes}. For example, for
outcome $f$, the cells $x$, $z$, and $w$ are connected to the left
edge (and therefore to each other), but $y$ is not. For outcome $h$,
only $x$ is connected to the left edge, but $y$, $z$, and $w$ are
connected to each other. Figure~\ref{fig-open-region-outcomes} also
shows the partial order on these outcomes.
\begin{figure}
  \def\scale{0.6}
  \def\xpace{\hspace{-7.1mm}}
  \def\squad{\hspace{1ex}}
  \[
  \stack{\m{$
    \begin{hexboard}[scale=\scale]
      \rotation{30}
      \foreach\i in {1,...,1} {
        \foreach\j in {1,...,4} {
          \hex(\j,\i)
        }
      }
      \edge[\nw](1,1)(4,1)
      \edge[\sw\noobtusecorner](1,1)(1,1)
      \edge[\ne\noacutecorner](4,1)(4,1)
      \white(1,1)
      \white(2,1)
      \white(3,1)
      \white(4,1)
    \end{hexboard}
    $}}{\bot}
  \xpace
  \stack{\m{$
    \begin{hexboard}[scale=\scale]
      \rotation{30}
      \foreach\i in {1,...,1} {
        \foreach\j in {1,...,4} {
          \hex(\j,\i)
        }
      }
      \edge[\nw](1,1)(4,1)
      \edge[\sw\noobtusecorner](1,1)(1,1)
      \edge[\ne\noacutecorner](4,1)(4,1)
      \white(1,1)
      \white(2,1)
      \white(3,1)
      \black(4,1)
    \end{hexboard}
    $}}{a}
  \xpace
  \stack{\m{$
    \begin{hexboard}[scale=\scale]
      \rotation{30}
      \foreach\i in {1,...,1} {
        \foreach\j in {1,...,4} {
          \hex(\j,\i)
        }
      }
      \edge[\nw](1,1)(4,1)
      \edge[\sw\noobtusecorner](1,1)(1,1)
      \edge[\ne\noacutecorner](4,1)(4,1)
      \white(1,1)
      \white(2,1)
      \black(3,1)
      \white(4,1)
    \end{hexboard}
    $}}{b}
  \xpace
  \stack{\m{$
    \begin{hexboard}[scale=\scale]
      \rotation{30}
      \foreach\i in {1,...,1} {
        \foreach\j in {1,...,4} {
          \hex(\j,\i)
        }
      }
      \edge[\nw](1,1)(4,1)
      \edge[\sw\noobtusecorner](1,1)(1,1)
      \edge[\ne\noacutecorner](4,1)(4,1)
      \black(1,1)
      \black(2,1)
      \white(3,1)
      \white(4,1)
    \end{hexboard}
    $}}{c}
  \xpace
  \stack{\m{$
    \begin{hexboard}[scale=\scale]
      \rotation{30}
      \foreach\i in {1,...,1} {
        \foreach\j in {1,...,4} {
          \hex(\j,\i)
        }
      }
      \edge[\nw](1,1)(4,1)
      \edge[\sw\noobtusecorner](1,1)(1,1)
      \edge[\ne\noacutecorner](4,1)(4,1)
      \black(1,1)
      \white(2,1)
      \white(3,1)
      \white(4,1)
    \end{hexboard}
    $}}{d}
  \xpace
  \stack{\m{$
    \begin{hexboard}[scale=\scale]
      \rotation{30}
      \foreach\i in {1,...,1} {
        \foreach\j in {1,...,4} {
          \hex(\j,\i)
        }
      }
      \edge[\nw](1,1)(4,1)
      \edge[\sw\noobtusecorner](1,1)(1,1)
      \edge[\ne\noacutecorner](4,1)(4,1)
      \white(1,1)
      \white(2,1)
      \black(3,1)
      \black(4,1)
    \end{hexboard}
    $}}{e}
  \xpace
  \stack{\m{$
    \begin{hexboard}[scale=\scale]
      \rotation{30}
      \foreach\i in {1,...,1} {
        \foreach\j in {1,...,4} {
          \hex(\j,\i)
        }
      }
      \edge[\nw](1,1)(4,1)
      \edge[\sw\noobtusecorner](1,1)(1,1)
      \edge[\ne\noacutecorner](4,1)(4,1)
      \black(1,1)
      \white(2,1)
      \white(3,1)
      \black(4,1)
    \end{hexboard}
    $}}{f}
  \xpace
  \stack{\m{$
    \begin{hexboard}[scale=\scale]
      \rotation{30}
      \foreach\i in {1,...,1} {
        \foreach\j in {1,...,4} {
          \hex(\j,\i)
        }
      }
      \edge[\nw](1,1)(4,1)
      \edge[\sw\noobtusecorner](1,1)(1,1)
      \edge[\ne\noacutecorner](4,1)(4,1)
      \black(1,1)
      \black(2,1)
      \black(3,1)
      \white(4,1)
    \end{hexboard}
    $}}{g}
  \xpace
  \stack{\m{$
    \begin{hexboard}[scale=\scale]
      \rotation{30}
      \foreach\i in {1,...,2} {
        \foreach\j in {1,...,4} {
          \hex(\j,\i)
        }
      }
      \edge[\nw](1,1)(4,1)
      \edge[\sw\noobtusecorner](1,1)(1,2)
      \edge[\ne\noacutecorner](4,2)(4,1)
      \black(1,1)
      \white(2,1)
      \white(3,1)
      \white(4,1)
      \black(1,2)
      \white(2,2)
      \black(3,2)
      \black(4,2)
    \end{hexboard}
    $}}{h}
  \xpace
  \stack{\m{$
    \begin{hexboard}[scale=\scale]
      \rotation{30}
      \foreach\i in {1,...,1} {
        \foreach\j in {1,...,4} {
          \hex(\j,\i)
        }
      }
      \edge[\nw](1,1)(4,1)
      \edge[\sw\noobtusecorner](1,1)(1,1)
      \edge[\ne\noacutecorner](4,1)(4,1)
      \black(1,1)
      \black(2,1)
      \black(3,1)
      \black(4,1)
    \end{hexboard}
    $}}{\top}
  \squad
  \mp{0.36}{
    \begin{tikzpicture}[yscale=0.65,xscale=0.85]
      \node(bot) at (0,0) {$\bot$};
      \node(a) at (-1,1) {$a$};
      \node(b) at (0,1) {$b$};
      \node(d) at (1,1) {$d$};
      \node(c) at (0.75,2) {$c$};
      \node(e) at (-1.5,3) {$e$};
      \node(f) at (-0.5,3) {$f$};
      \node(g) at (0.5,3) {$g$};
      \node(h) at (1.5,3) {$h$};
      \node(top) at (0,4) {$\top$};
      \draw (bot) -- (a) -- (f) -- (top);
      \draw (bot) -- (b) -- (e) -- (top);
      \draw (bot) -- (d) -- (c) -- (g) -- (top);
      \draw (d) -- (h) -- (top);
      \draw (a) -- (e);
      \draw (b) -- (g);
      \draw (d) -- (f);
    \end{tikzpicture}
    }
  \]
  \caption{The atomic outcomes for open regions of height $4$.}
  \label{fig-open-region-outcomes}
\end{figure}
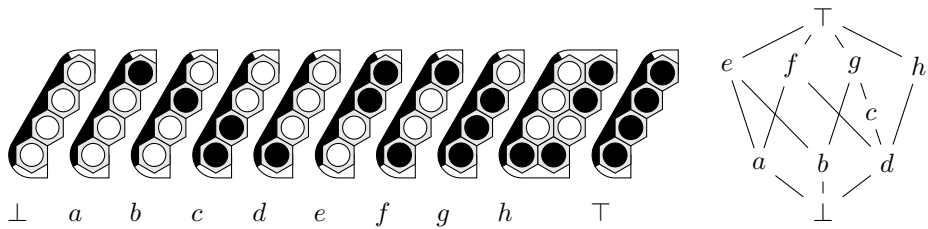

Armed with this outcome poset, we can now calculate the value of
positions in open regions. For example, Figure~\ref{fig-open-position}
shows such a position $P$ and its combinatorial value, which we write
$\val(P)$. The point of such combinatorial values is not that they
should be readable and understandable by humans, but that they can be
calculated and compared.
\begin{figure}
  \[
  P = \m{$
    \begin{hexboard}[scale=0.6]
      \rotation{30}
      \foreach\i in {1,...,7} {
        \foreach\j in {1,...,4} {
          \hex(\j,\i)
        }
      }
      \edge[\nw](1,1)(4,1)
      \edge[\sw\noobtusecorner](1,1)(1,7)
      \edge[\ne\noacutecorner](4,7)(4,1)
      \black(2,3)
      \black(3,4)
      \black(2,6)
      \black(3,7)
    \end{hexboard}
    $}
  \]
  \vspace{0.5ex}
  \[
  \begingroup
  \catcode`|=\active
  \let|\mid
  \begin{array}{r@{~}c@{~}l}
    \val(P) &=&
    \{\{\top|g\},\{\top|\{\{\top|h\},\{\top|e\}|\{h|\bot\},\{e|\bot\}\}\}|\{g|\{\{g|d\},\\
    && ~\,\{g|b\}|\{d|\bot\},\{b|\bot\}\}\}, \{\{\{\top|h\},\{\top|g\}|\{h|d\},\{g|d\}\},\\
    && ~\,\{\{\top|h\},\{\top|e\}|\{h|\bot\},\{e|\bot\}\},\{\{\top|g\},\{\top|e\}|\{g|b\},\{e|b\}\}\\
    && ~{}|\{\{h|d\}|\bot\},\{\{g|d\},\{g|b\}|\{d|\bot\},\{b|\bot\}\},\{\{e|b\}|\bot\}\}\}
  \end{array}
  \endgroup
  \]
  \caption{A position $P$ in an open region and its value $\val(P)$.}
  \label{fig-open-position}
\end{figure}

\subsection{Column-Wise Computation}

It is important that values of open regions can be computed
efficiently. Note that the most naive method is not efficient. The Hex
position in Figure~\ref{fig-open-position} has 24 empty cells. If we
were to calculate its value by brute force, we would first obtain a
game with 24 left and right options, each of which has 23 left
and right options and so on, to depth 24. This game would be of
astronomical size, containing roughly $10^{31}$ atomic positions.
Although its canonical form (shown in Figure~\ref{fig-open-position})
is much smaller, it would not be feasible to compute it in this way.

The key to computing values of open positions efficiently is the
following. Consider a {\em column}, which is a region of size $4\times
1$ with top and bottom edges (but no left and right edges), as shown
in Figure~\ref{fig-column}.
A column has exactly 16 distinct atomic outcomes, namely, all of the
ways of filling the column with black and white stones, as shown in
Figure~\ref{fig-column-outcomes}.
Given an open position $P$ of size $4\times n$ and a column $C$, we
write $P+C$ for the open position of size $4\times (n+1)$ obtained by
adding $C$ to the right of $P$. This naturally gives rise to a
function on outcome classes, i.e., for $P\in\s{\bot,a,\ldots,h,\top}$
and $C\in\s{k_0,\ldots,k_{15}}$, we define $f(P,C)$ as the outcome class of
$P+C$.  For example, we have $f(c,k_{10}) = d$, because
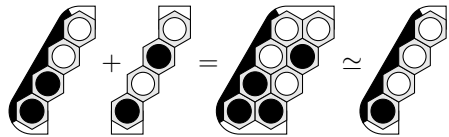
\begin{figure}
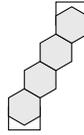

  \[
  \begin{hexboard}[scale=0.6]
    \rotation{30}
    \foreach\i in {1,...,1} {
      \foreach\j in {1,...,4} {
        \hex(\j,\i)
      }
    }
    \edge[\sw\noobtusecorner\noacutecorner](1,1)(1,1)
    \edge[\ne\noacutecorner\noobtusecorner](4,1)(4,1)
  \end{hexboard}
  \]
  \caption{A column of height $4$.}
  \label{fig-column}
\end{figure}
\begin{figure}
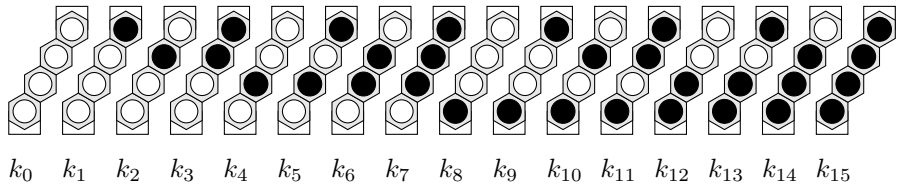

  \def\scale{0.6}
  \def\xpace{\hspace{-7mm}}
  \[
  \stack{\m{$
      \begin{hexboard}[scale=\scale]
        \rotation{30}
        \foreach\i in {1,...,1} {
          \foreach\j in {1,...,4} {
            \hex(\j,\i)
          }
        }
        \edge[\sw\noobtusecorner\noacutecorner](1,1)(1,1)
        \edge[\ne\noacutecorner\noobtusecorner](4,1)(4,1)
        \white(1,1)
        \white(2,1)
        \white(3,1)
        \white(4,1)
      \end{hexboard}
      $}}{k_{0}}
  \xpace
  \stack{\m{$
      \begin{hexboard}[scale=\scale]
        \rotation{30}
        \foreach\i in {1,...,1} {
          \foreach\j in {1,...,4} {
            \hex(\j,\i)
          }
        }
        \edge[\sw\noobtusecorner\noacutecorner](1,1)(1,1)
        \edge[\ne\noacutecorner\noobtusecorner](4,1)(4,1)
        \white(1,1)
        \white(2,1)
        \white(3,1)
        \black(4,1)
      \end{hexboard}
      $}}{k_{1}}
  \xpace
  \stack{\m{$
      \begin{hexboard}[scale=\scale]
        \rotation{30}
        \foreach\i in {1,...,1} {
          \foreach\j in {1,...,4} {
            \hex(\j,\i)
          }
        }
        \edge[\sw\noobtusecorner\noacutecorner](1,1)(1,1)
        \edge[\ne\noacutecorner\noobtusecorner](4,1)(4,1)
        \white(1,1)
        \white(2,1)
        \black(3,1)
        \white(4,1)
      \end{hexboard}
      $}}{k_{2}}
  \xpace
  \stack{\m{$
      \begin{hexboard}[scale=\scale]
        \rotation{30}
        \foreach\i in {1,...,1} {
          \foreach\j in {1,...,4} {
            \hex(\j,\i)
          }
        }
        \edge[\sw\noobtusecorner\noacutecorner](1,1)(1,1)
        \edge[\ne\noacutecorner\noobtusecorner](4,1)(4,1)
        \white(1,1)
        \white(2,1)
        \black(3,1)
        \black(4,1)
      \end{hexboard}
      $}}{k_{3}}
  \xpace
  \stack{\m{$
      \begin{hexboard}[scale=\scale]
        \rotation{30}
        \foreach\i in {1,...,1} {
          \foreach\j in {1,...,4} {
            \hex(\j,\i)
          }
        }
        \edge[\sw\noobtusecorner\noacutecorner](1,1)(1,1)
        \edge[\ne\noacutecorner\noobtusecorner](4,1)(4,1)
        \white(1,1)
        \black(2,1)
        \white(3,1)
        \white(4,1)
      \end{hexboard}
      $}}{k_{4}}
  \xpace
  \stack{\m{$
      \begin{hexboard}[scale=\scale]
        \rotation{30}
        \foreach\i in {1,...,1} {
          \foreach\j in {1,...,4} {
            \hex(\j,\i)
          }
        }
        \edge[\sw\noobtusecorner\noacutecorner](1,1)(1,1)
        \edge[\ne\noacutecorner\noobtusecorner](4,1)(4,1)
        \white(1,1)
        \black(2,1)
        \white(3,1)
        \black(4,1)
      \end{hexboard}
      $}}{k_{5}}
  \xpace
  \stack{\m{$
      \begin{hexboard}[scale=\scale]
        \rotation{30}
        \foreach\i in {1,...,1} {
          \foreach\j in {1,...,4} {
            \hex(\j,\i)
          }
        }
        \edge[\sw\noobtusecorner\noacutecorner](1,1)(1,1)
        \edge[\ne\noacutecorner\noobtusecorner](4,1)(4,1)
        \white(1,1)
        \black(2,1)
        \black(3,1)
        \white(4,1)
      \end{hexboard}
      $}}{k_{6}}
  \xpace
  \stack{\m{$
      \begin{hexboard}[scale=\scale]
        \rotation{30}
        \foreach\i in {1,...,1} {
          \foreach\j in {1,...,4} {
            \hex(\j,\i)
          }
        }
        \edge[\sw\noobtusecorner\noacutecorner](1,1)(1,1)
        \edge[\ne\noacutecorner\noobtusecorner](4,1)(4,1)
        \white(1,1)
        \black(2,1)
        \black(3,1)
        \black(4,1)
      \end{hexboard}
      $}}{k_{7}}
  \xpace
  \stack{\m{$
      \begin{hexboard}[scale=\scale]
        \rotation{30}
        \foreach\i in {1,...,1} {
          \foreach\j in {1,...,4} {
            \hex(\j,\i)
          }
        }
        \edge[\sw\noobtusecorner\noacutecorner](1,1)(1,1)
        \edge[\ne\noacutecorner\noobtusecorner](4,1)(4,1)
        \black(1,1)
        \white(2,1)
        \white(3,1)
        \white(4,1)
      \end{hexboard}
      $}}{k_{8}}
  \xpace
  \stack{\m{$
      \begin{hexboard}[scale=\scale]
        \rotation{30}
        \foreach\i in {1,...,1} {
          \foreach\j in {1,...,4} {
            \hex(\j,\i)
          }
        }
        \edge[\sw\noobtusecorner\noacutecorner](1,1)(1,1)
        \edge[\ne\noacutecorner\noobtusecorner](4,1)(4,1)
        \black(1,1)
        \white(2,1)
        \white(3,1)
        \black(4,1)
      \end{hexboard}
      $}}{k_{9}}
  \xpace
  \stack{\m{$
      \begin{hexboard}[scale=\scale]
        \rotation{30}
        \foreach\i in {1,...,1} {
          \foreach\j in {1,...,4} {
            \hex(\j,\i)
          }
        }
        \edge[\sw\noobtusecorner\noacutecorner](1,1)(1,1)
        \edge[\ne\noacutecorner\noobtusecorner](4,1)(4,1)
        \black(1,1)
        \white(2,1)
        \black(3,1)
        \white(4,1)
      \end{hexboard}
      $}}{k_{10}}
  \xpace
  \stack{\m{$
      \begin{hexboard}[scale=\scale]
        \rotation{30}
        \foreach\i in {1,...,1} {
          \foreach\j in {1,...,4} {
            \hex(\j,\i)
          }
        }
        \edge[\sw\noobtusecorner\noacutecorner](1,1)(1,1)
        \edge[\ne\noacutecorner\noobtusecorner](4,1)(4,1)
        \black(1,1)
        \white(2,1)
        \black(3,1)
        \black(4,1)
      \end{hexboard}
      $}}{k_{11}}
  \xpace
  \stack{\m{$
      \begin{hexboard}[scale=\scale]
        \rotation{30}
        \foreach\i in {1,...,1} {
          \foreach\j in {1,...,4} {
            \hex(\j,\i)
          }
        }
        \edge[\sw\noobtusecorner\noacutecorner](1,1)(1,1)
        \edge[\ne\noacutecorner\noobtusecorner](4,1)(4,1)
        \black(1,1)
        \black(2,1)
        \white(3,1)
        \white(4,1)
      \end{hexboard}
      $}}{k_{12}}
  \xpace
  \stack{\m{$
      \begin{hexboard}[scale=\scale]
        \rotation{30}
        \foreach\i in {1,...,1} {
          \foreach\j in {1,...,4} {
            \hex(\j,\i)
          }
        }
        \edge[\sw\noobtusecorner\noacutecorner](1,1)(1,1)
        \edge[\ne\noacutecorner\noobtusecorner](4,1)(4,1)
        \black(1,1)
        \black(2,1)
        \white(3,1)
        \black(4,1)
      \end{hexboard}
      $}}{k_{13}}
  \xpace
  \stack{\m{$
      \begin{hexboard}[scale=\scale]
        \rotation{30}
        \foreach\i in {1,...,1} {
          \foreach\j in {1,...,4} {
            \hex(\j,\i)
          }
        }
        \edge[\sw\noobtusecorner\noacutecorner](1,1)(1,1)
        \edge[\ne\noacutecorner\noobtusecorner](4,1)(4,1)
        \black(1,1)
        \black(2,1)
        \black(3,1)
        \white(4,1)
      \end{hexboard}
      $}}{k_{14}}
  \xpace
  \stack{\m{$
      \begin{hexboard}[scale=\scale]
        \rotation{30}
        \foreach\i in {1,...,1} {
          \foreach\j in {1,...,4} {
            \hex(\j,\i)
          }
        }
        \edge[\sw\noobtusecorner\noacutecorner](1,1)(1,1)
        \edge[\ne\noacutecorner\noobtusecorner](4,1)(4,1)
        \black(1,1)
        \black(2,1)
        \black(3,1)
        \black(4,1)
      \end{hexboard}
      $}}{k_{15}}
  \]
  \caption{The atomic outcomes for columns of height $4$.}
  \label{fig-column-outcomes}
\end{figure}
\[
\def\scale{0.6}
\m{$
  \begin{hexboard}[scale=\scale]
    \rotation{30} \foreach\i in {1,...,1} { \foreach\j in {1,...,4} {
        \hex(\j,\i) } } \edge[\nw](1,1)(4,1)
    \edge[\sw\noobtusecorner](1,1)(1,1)
    \edge[\ne\noacutecorner](4,1)(4,1) \black(1,1) \black(2,1)
    \white(3,1) \white(4,1)
  \end{hexboard}
  $} +\hspace{-1.5ex} \m{$
  \begin{hexboard}[scale=\scale]
    \rotation{30} \foreach\i in {1,...,1} { \foreach\j in {1,...,4} {
        \hex(\j,\i) } }
    \edge[\sw\noobtusecorner\noacutecorner](1,1)(1,1)
    \edge[\ne\noacutecorner\noobtusecorner](4,1)(4,1) \black(1,1)
    \white(2,1) \black(3,1) \white(4,1)
  \end{hexboard}
  $} =\hspace{-1ex} \m{$
  \begin{hexboard}[scale=\scale]
    \rotation{30} \foreach\i in {1,...,2} { \foreach\j in {1,...,4} {
        \hex(\j,\i) } } \edge[\nw](1,1)(4,1)
    \edge[\sw\noobtusecorner](1,1)(1,2)
    \edge[\ne\noacutecorner](4,2)(4,1) \black(1,1) \black(2,1)
    \white(3,1) \white(4,1) \black(1,2) \white(2,2) \black(3,2)
    \white(4,2)
  \end{hexboard}
  $} \eq\hspace{-1ex} \m{$
  \begin{hexboard}[scale=\scale]
    \rotation{30} \foreach\i in {1,...,1} { \foreach\j in {1,...,4} {
        \hex(\j,\i) } } \edge[\nw](1,1)(4,1)
    \edge[\sw\noobtusecorner](1,1)(1,1)
    \edge[\ne\noacutecorner](4,1)(4,1) \black(1,1) \white(2,1)
    \white(3,1) \white(4,1)
  \end{hexboard}
  $}
\]
This function $f$ on atomic games can then be extended to non-atomic
games as in Section~\ref{ssec-map}. Namely, if $P$ and $C$ are (not
necessarily atomic) positions for an open region and a column, then
$\val(P+C)=\val(P)\pf\val(C)$. From now on, by a slight abuse of
notation, we just write $+$ instead of $\pf$, so that we can use the
notation $P+C$ whether $P$ and $C$ are positions, outcomes, or values.
We occasionally confuse positions with their values or vice versa.

We now have a reasonably efficient method for computing values of open
positions.  Namely, we can start from a position of size $4\times 0$,
then add one column, reduce to canonical form, add the next column,
reduce to canonical form, and so on.

\subsection{Best Patterns}

Now that we have a good method for calculating the value of open
$4\times n$ positions, we need one more ingredient before we can prove
results on the minimality of virtual connections on $4\times
n$-boards.  We certainly cannot do this by calculating the values of
all possible arrangements of black stones. The final key idea is that
we do not need to calculate the values of all such arrangements. It
suffices to calculate the values of the {\em best} arrangements. For
brevity, we will use the term {\em pattern} for a position in an open
region that uses only black stones.

\begin{definition}
  A pattern $P$ is {\em unacceptable} if $\val(P)\tri\bot$. In this
  case, White has a winning move in $P$, so clearly $P$ can never be
  part of a virtual connection for Black.
\end{definition}

\begin{definition}
  To each acceptable pattern $P$, we associate a triple $\triple{P} =
  (G, s, n)$, where:
  \begin{itemize}
  \item $G = \val(P)$ is the combinatorial value of $P$,
  \item $s$ is the number of black stones used, and
  \item $n$ is the {\em width} of $P$, i.e., the number of columns in $P$.
  \end{itemize}
  For convenience, we also define $\triple{P}=(\bot,\infty,0)$ if $P$
  is an unacceptable pattern. For example, for the pattern $P$ from
  Figure~\ref{fig-open-position}, we have $\triple{P} = (G, 4, 7)$,
  where $G$ is the value shown in Figure~\ref{fig-open-position}.  We
  define an ordering on such triples by $(G,s,n)\leq (G',s',n')$ if
  $G\leq G'$ and $s'\leq s$ and $n\leq n'$. We also write $P\leq P'$
  if $\triple{P}\leq\triple{P'}$, and we say that $P'$ is {\em at
    least as good as} $P$. When $P\leq P'$ and $P'\leq P$, we say that
  $P$ and $P'$ are {\em equivalent}. Note that all unacceptable
  patterns are equivalent to each other, and are worse than all
  acceptable patterns.
\end{definition}

Informally, the ordering on triples is justified as follows. If two
patterns $P$ and $P'$ cover the same distance with the same number of
black stones, but $P'$ has a higher combinatorial value than $P$, then
$P'$ is clearly better, in the sense that everything that can be
achieved with $P$ can also be achieved with $P'$. In other words, if
extending $P$ with some additional columns results in a virtual
connection, then so does adding the same columns to $P'$. Next,
if $P$ and $P'$ have the same combinatorial value and cover the same
distance, but $P'$ uses fewer stones, then $P'$ is the better of the
two. Finally, if $P$ and $P'$ have the same combinatorial value and
use the same number of stones, but $P'$ covers a greater distance,
then again $P'$ is better.

We note that if $P=P'+C$, then $\triple{P}$ is completely determined
by $\triple{P'}$ and $C$. In particular, if $P$ and $Q$ are
equivalent, then so are $P+C$ and $Q+C$.

We can now systematically enumerate the {\em best} patterns of each
width $n$, i.e., patterns that are maximal with respect to the
preorder $\leq$. When several patterns are equivalent, we only keep
one of them, selected arbitrarily. The enumeration proceeds as
follows: For $n=0$, there is only one pattern, and it is
automatically best. Now suppose we have the list $P_1,\ldots,P_m$ of
the best patterns of width $n$. We generate a list of patterns of
width $n+1$ by adding each of the 16 possible column patterns (using blank
cells and black stones) to each of $P_1,\ldots,P_m$. Of these, we keep
only the best ones, throwing out the rest.

Perhaps surprisingly, the number of best patterns does not grow
exponentially as the width $n$ increases; it remains bounded. In fact,
in an appropriate sense, the sequence eventually repeats.
Figure~\ref{fig-best} shows all best patterns (up to equivalence) of
width 4, 5, 6, and 7.
\begin{figure}
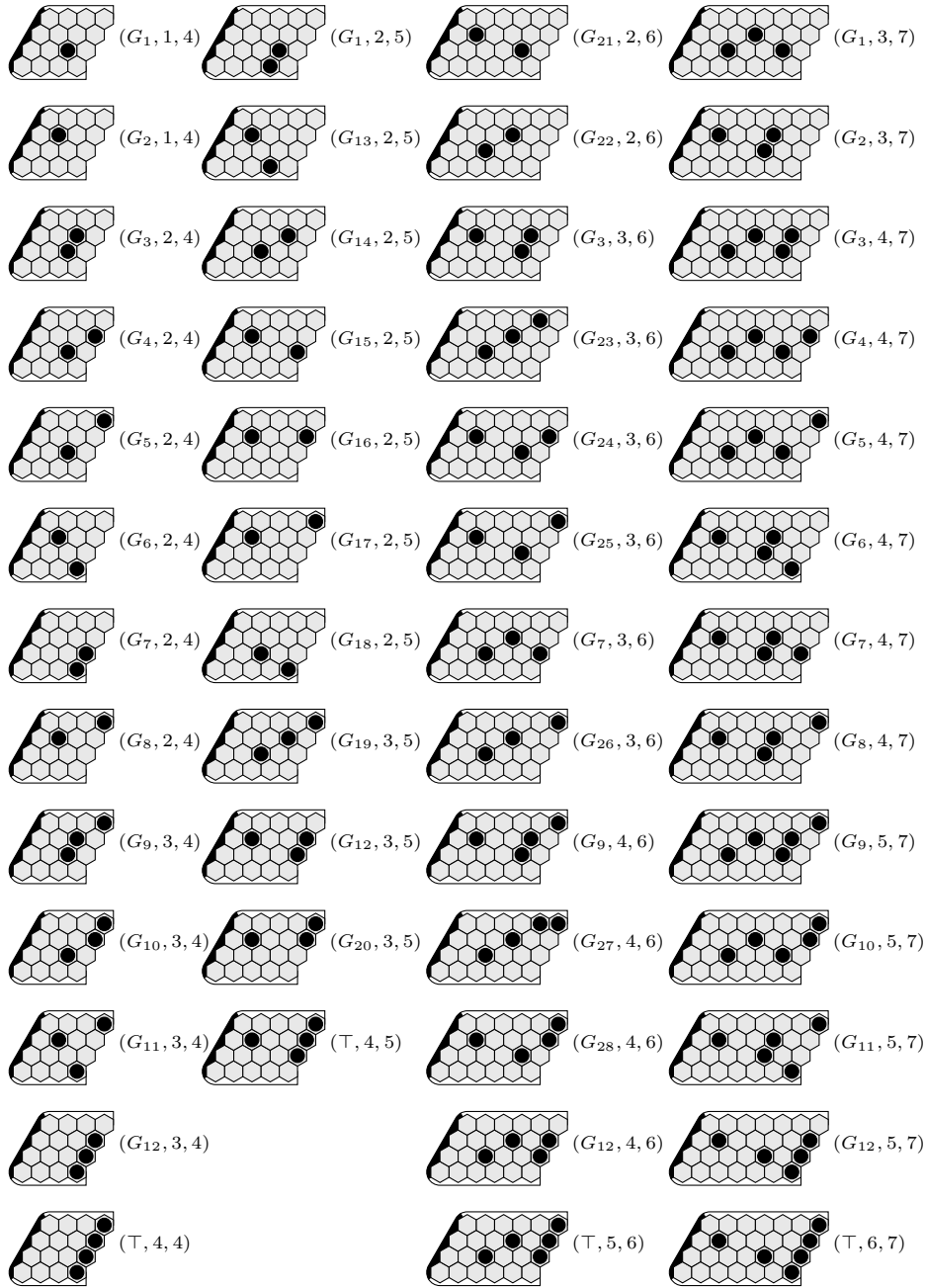

  \def\scale{0.35}
  \def\stack#1#2{#1\,\mbox{\scriptsize$#2$}}
  \def\sep{\\\\[-0.5ex]}
  \def\xpace{\hspace{-3mm}}
  \[
  \begin{array}[t]{l@{}}
    \stack{\m{$
    \begin{hexboard}[scale=\scale]
      \rotation{30}
      \foreach\i in {1,...,4} {
        \foreach\j in {1,...,4} {
          \hex(\j,\i)
        }
      }
      \edge[\nw](1,1)(4,1)
      \edge[\sw\noobtusecorner](1,1)(1,4)
      \edge[\ne\noacutecorner](4,4)(4,1)
      \black(2,3)
    \end{hexboard}
    $}}{(G_{1},1,4)} 
    \sep
    \stack{\m{$
    \begin{hexboard}[scale=\scale]
      \rotation{30}
      \foreach\i in {1,...,4} {
        \foreach\j in {1,...,4} {
          \hex(\j,\i)
        }
      }
      \edge[\nw](1,1)(4,1)
      \edge[\sw\noobtusecorner](1,1)(1,4)
      \edge[\ne\noacutecorner](4,4)(4,1)
      \black(3,2)
    \end{hexboard}
    $}}{(G_{2},1,4)} 
    \sep
    \stack{\m{$
    \begin{hexboard}[scale=\scale]
      \rotation{30}
      \foreach\i in {1,...,4} {
        \foreach\j in {1,...,4} {
          \hex(\j,\i)
        }
      }
      \edge[\nw](1,1)(4,1)
      \edge[\sw\noobtusecorner](1,1)(1,4)
      \edge[\ne\noacutecorner](4,4)(4,1)
      \black(2,3)
      \black(3,3)
    \end{hexboard}
    $}}{(G_{3},2,4)} 
    \sep
    \stack{\m{$
    \begin{hexboard}[scale=\scale]
      \rotation{30}
      \foreach\i in {1,...,4} {
        \foreach\j in {1,...,4} {
          \hex(\j,\i)
        }
      }
      \edge[\nw](1,1)(4,1)
      \edge[\sw\noobtusecorner](1,1)(1,4)
      \edge[\ne\noacutecorner](4,4)(4,1)
      \black(2,3)
      \black(3,4)
    \end{hexboard}
    $}}{(G_{4},2,4)} 
    \sep
    \stack{\m{$
    \begin{hexboard}[scale=\scale]
      \rotation{30}
      \foreach\i in {1,...,4} {
        \foreach\j in {1,...,4} {
          \hex(\j,\i)
        }
      }
      \edge[\nw](1,1)(4,1)
      \edge[\sw\noobtusecorner](1,1)(1,4)
      \edge[\ne\noacutecorner](4,4)(4,1)
      \black(2,3)
      \black(4,4)
    \end{hexboard}
    $}}{(G_{5},2,4)} 
    \sep
    \stack{\m{$
    \begin{hexboard}[scale=\scale]
      \rotation{30}
      \foreach\i in {1,...,4} {
        \foreach\j in {1,...,4} {
          \hex(\j,\i)
        }
      }
      \edge[\nw](1,1)(4,1)
      \edge[\sw\noobtusecorner](1,1)(1,4)
      \edge[\ne\noacutecorner](4,4)(4,1)
      \black(3,2)
      \black(1,4)
    \end{hexboard}
    $}}{(G_{6},2,4)} 
    \sep
    \stack{\m{$
    \begin{hexboard}[scale=\scale]
      \rotation{30}
      \foreach\i in {1,...,4} {
        \foreach\j in {1,...,4} {
          \hex(\j,\i)
        }
      }
      \edge[\nw](1,1)(4,1)
      \edge[\sw\noobtusecorner](1,1)(1,4)
      \edge[\ne\noacutecorner](4,4)(4,1)
      \black(1,4)
      \black(2,4)
    \end{hexboard}
    $}}{(G_{7},2,4)} 
    \sep
    \stack{\m{$
    \begin{hexboard}[scale=\scale]
      \rotation{30}
      \foreach\i in {1,...,4} {
        \foreach\j in {1,...,4} {
          \hex(\j,\i)
        }
      }
      \edge[\nw](1,1)(4,1)
      \edge[\sw\noobtusecorner](1,1)(1,4)
      \edge[\ne\noacutecorner](4,4)(4,1)
      \black(3,2)
      \black(4,4)
    \end{hexboard}
    $}}{(G_{8},2,4)} 
    \sep
    \stack{\m{$
    \begin{hexboard}[scale=\scale]
      \rotation{30}
      \foreach\i in {1,...,4} {
        \foreach\j in {1,...,4} {
          \hex(\j,\i)
        }
      }
      \edge[\nw](1,1)(4,1)
      \edge[\sw\noobtusecorner](1,1)(1,4)
      \edge[\ne\noacutecorner](4,4)(4,1)
      \black(2,3)
      \black(3,3)
      \black(4,4)
    \end{hexboard}
    $}}{(G_{9},3,4)} 
    \sep
    \stack{\m{$
    \begin{hexboard}[scale=\scale]
      \rotation{30}
      \foreach\i in {1,...,4} {
        \foreach\j in {1,...,4} {
          \hex(\j,\i)
        }
      }
      \edge[\nw](1,1)(4,1)
      \edge[\sw\noobtusecorner](1,1)(1,4)
      \edge[\ne\noacutecorner](4,4)(4,1)
      \black(2,3)
      \black(3,4)
      \black(4,4)
    \end{hexboard}
    $}}{(G_{10},3,4)} 
    \sep
    \stack{\m{$
    \begin{hexboard}[scale=\scale]
      \rotation{30}
      \foreach\i in {1,...,4} {
        \foreach\j in {1,...,4} {
          \hex(\j,\i)
        }
      }
      \edge[\nw](1,1)(4,1)
      \edge[\sw\noobtusecorner](1,1)(1,4)
      \edge[\ne\noacutecorner](4,4)(4,1)
      \black(3,2)
      \black(1,4)
      \black(4,4)
    \end{hexboard}
    $}}{(G_{11},3,4)} 
    \sep
    \stack{\m{$
    \begin{hexboard}[scale=\scale]
      \rotation{30}
      \foreach\i in {1,...,4} {
        \foreach\j in {1,...,4} {
          \hex(\j,\i)
        }
      }
      \edge[\nw](1,1)(4,1)
      \edge[\sw\noobtusecorner](1,1)(1,4)
      \edge[\ne\noacutecorner](4,4)(4,1)
      \black(1,4)
      \black(2,4)
      \black(3,4)
    \end{hexboard}
    $}}{(G_{12},3,4)} 
    \sep
    \stack{\m{$
    \begin{hexboard}[scale=\scale]
      \rotation{30}
      \foreach\i in {1,...,4} {
        \foreach\j in {1,...,4} {
          \hex(\j,\i)
        }
      }
      \edge[\nw](1,1)(4,1)
      \edge[\sw\noobtusecorner](1,1)(1,4)
      \edge[\ne\noacutecorner](4,4)(4,1)
      \black(1,4)
      \black(2,4)
      \black(3,4)
      \black(4,4)
    \end{hexboard}
    $}}{(\top,4,4)} 
  \end{array}
  \xpace
  \begin{array}[t]{l}
  \stack{\m{$
    \begin{hexboard}[scale=\scale]
      \rotation{30}
      \foreach\i in {1,...,5} {
        \foreach\j in {1,...,4} {
          \hex(\j,\i)
        }
      }
      \edge[\nw](1,1)(4,1)
      \edge[\sw\noobtusecorner](1,1)(1,5)
      \edge[\ne\noacutecorner](4,5)(4,1)
      \black(1,4)
      \black(2,4)
    \end{hexboard}
    $}}{(G_{1},2,5)} 
  \sep
  \stack{\m{$
    \begin{hexboard}[scale=\scale]
      \rotation{30}
      \foreach\i in {1,...,5} {
        \foreach\j in {1,...,4} {
          \hex(\j,\i)
        }
      }
      \edge[\nw](1,1)(4,1)
      \edge[\sw\noobtusecorner](1,1)(1,5)
      \edge[\ne\noacutecorner](4,5)(4,1)
      \black(1,4)
      \black(3,2)
    \end{hexboard}
    $}}{(G_{13},2,5)} 
  \sep
  \stack{\m{$
    \begin{hexboard}[scale=\scale]
      \rotation{30}
      \foreach\i in {1,...,5} {
        \foreach\j in {1,...,4} {
          \hex(\j,\i)
        }
      }
      \edge[\nw](1,1)(4,1)
      \edge[\sw\noobtusecorner](1,1)(1,5)
      \edge[\ne\noacutecorner](4,5)(4,1)
      \black(2,3)
      \black(3,4)
    \end{hexboard}
    $}}{(G_{14},2,5)} 
  \sep
    \stack{\m{$
    \begin{hexboard}[scale=\scale]
      \rotation{30}
      \foreach\i in {1,...,5} {
        \foreach\j in {1,...,4} {
          \hex(\j,\i)
        }
      }
      \edge[\nw](1,1)(4,1)
      \edge[\sw\noobtusecorner](1,1)(1,5)
      \edge[\ne\noacutecorner](4,5)(4,1)
      \black(3,2)
      \black(2,5)
    \end{hexboard}
    $}}{(G_{15},2,5)} 
  \sep
    \stack{\m{$
    \begin{hexboard}[scale=\scale]
      \rotation{30}
      \foreach\i in {1,...,5} {
        \foreach\j in {1,...,4} {
          \hex(\j,\i)
        }
      }
      \edge[\nw](1,1)(4,1)
      \edge[\sw\noobtusecorner](1,1)(1,5)
      \edge[\ne\noacutecorner](4,5)(4,1)
      \black(3,2)
      \black(3,5)
    \end{hexboard}
    $}}{(G_{16},2,5)} 
  \sep
    \stack{\m{$
    \begin{hexboard}[scale=\scale]
      \rotation{30}
      \foreach\i in {1,...,5} {
        \foreach\j in {1,...,4} {
          \hex(\j,\i)
        }
      }
      \edge[\nw](1,1)(4,1)
      \edge[\sw\noobtusecorner](1,1)(1,5)
      \edge[\ne\noacutecorner](4,5)(4,1)
      \black(3,2)
      \black(4,5)
    \end{hexboard}
    $}}{(G_{17},2,5)} 
  \sep
    \stack{\m{$
    \begin{hexboard}[scale=\scale]
      \rotation{30}
      \foreach\i in {1,...,5} {
        \foreach\j in {1,...,4} {
          \hex(\j,\i)
        }
      }
      \edge[\nw](1,1)(4,1)
      \edge[\sw\noobtusecorner](1,1)(1,5)
      \edge[\ne\noacutecorner](4,5)(4,1)
      \black(1,5)
      \black(2,3)
    \end{hexboard}
    $}}{(G_{18},2,5)} 
  \sep
    \stack{\m{$
    \begin{hexboard}[scale=\scale]
      \rotation{30}
      \foreach\i in {1,...,5} {
        \foreach\j in {1,...,4} {
          \hex(\j,\i)
        }
      }
      \edge[\nw](1,1)(4,1)
      \edge[\sw\noobtusecorner](1,1)(1,5)
      \edge[\ne\noacutecorner](4,5)(4,1)
      \black(2,3)
      \black(3,4)
      \black(4,5)
    \end{hexboard}
    $}}{(G_{19},3,5)} 
  \sep
    \stack{\m{$
    \begin{hexboard}[scale=\scale]
      \rotation{30}
      \foreach\i in {1,...,5} {
        \foreach\j in {1,...,4} {
          \hex(\j,\i)
        }
      }
      \edge[\nw](1,1)(4,1)
      \edge[\sw\noobtusecorner](1,1)(1,5)
      \edge[\ne\noacutecorner](4,5)(4,1)
      \black(2,5)
      \black(3,2)
      \black(3,5)
    \end{hexboard}
    $}}{(G_{12},3,5)} 
  \sep
    \stack{\m{$
    \begin{hexboard}[scale=\scale]
      \rotation{30}
      \foreach\i in {1,...,5} {
        \foreach\j in {1,...,4} {
          \hex(\j,\i)
        }
      }
      \edge[\nw](1,1)(4,1)
      \edge[\sw\noobtusecorner](1,1)(1,5)
      \edge[\ne\noacutecorner](4,5)(4,1)
      \black(3,2)
      \black(3,5)
      \black(4,5)
    \end{hexboard}
    $}}{(G_{20},3,5)} 
  \sep
    \stack{\m{$
    \begin{hexboard}[scale=\scale]
      \rotation{30}
      \foreach\i in {1,...,5} {
        \foreach\j in {1,...,4} {
          \hex(\j,\i)
        }
      }
      \edge[\nw](1,1)(4,1)
      \edge[\sw\noobtusecorner](1,1)(1,5)
      \edge[\ne\noacutecorner](4,5)(4,1)
      \black(2,5)
      \black(3,2)
      \black(3,5)
      \black(4,5)
    \end{hexboard}
    $}}{(\top,4,5)} 
  \end{array}
  \xpace
  \begin{array}[t]{l}
    \stack{\m{$
    \begin{hexboard}[scale=\scale]
      \rotation{30}
      \foreach\i in {1,...,6} {
        \foreach\j in {1,...,4} {
          \hex(\j,\i)
        }
      }
      \edge[\nw](1,1)(4,1)
      \edge[\sw\noobtusecorner](1,1)(1,6)
      \edge[\ne\noacutecorner](4,6)(4,1)
      \black(2,5)
      \black(3,2)
    \end{hexboard}
    $}}{(G_{21},2,6)} 
  \sep
    \stack{\m{$
    \begin{hexboard}[scale=\scale]
      \rotation{30}
      \foreach\i in {1,...,6} {
        \foreach\j in {1,...,4} {
          \hex(\j,\i)
        }
      }
      \edge[\nw](1,1)(4,1)
      \edge[\sw\noobtusecorner](1,1)(1,6)
      \edge[\ne\noacutecorner](4,6)(4,1)
      \black(2,3)
      \black(3,4)
    \end{hexboard}
    $}}{(G_{22},2,6)} 
  \sep
    \stack{\m{$
    \begin{hexboard}[scale=\scale]
      \rotation{30}
      \foreach\i in {1,...,6} {
        \foreach\j in {1,...,4} {
          \hex(\j,\i)
        }
      }
      \edge[\nw](1,1)(4,1)
      \edge[\sw\noobtusecorner](1,1)(1,6)
      \edge[\ne\noacutecorner](4,6)(4,1)
      \black(2,5)
      \black(3,2)
      \black(3,5)
    \end{hexboard}
    $}}{(G_{3},3,6)} 
  \sep
    \stack{\m{$
    \begin{hexboard}[scale=\scale]
      \rotation{30}
      \foreach\i in {1,...,6} {
        \foreach\j in {1,...,4} {
          \hex(\j,\i)
        }
      }
      \edge[\nw](1,1)(4,1)
      \edge[\sw\noobtusecorner](1,1)(1,6)
      \edge[\ne\noacutecorner](4,6)(4,1)
      \black(2,3)
      \black(3,4)
      \black(4,5)
    \end{hexboard}
    $}}{(G_{23},3,6)} 
    \sep
    \stack{\m{$
    \begin{hexboard}[scale=\scale]
      \rotation{30}
      \foreach\i in {1,...,6} {
        \foreach\j in {1,...,4} {
          \hex(\j,\i)
        }
      }
      \edge[\nw](1,1)(4,1)
      \edge[\sw\noobtusecorner](1,1)(1,6)
      \edge[\ne\noacutecorner](4,6)(4,1)
      \black(2,5)
      \black(3,2)
      \black(3,6)
    \end{hexboard}
    $}}{(G_{24},3,6)} 
  \sep
    \stack{\m{$
    \begin{hexboard}[scale=\scale]
      \rotation{30}
      \foreach\i in {1,...,6} {
        \foreach\j in {1,...,4} {
          \hex(\j,\i)
        }
      }
      \edge[\nw](1,1)(4,1)
      \edge[\sw\noobtusecorner](1,1)(1,6)
      \edge[\ne\noacutecorner](4,6)(4,1)
      \black(2,5)
      \black(3,2)
      \black(4,6)
    \end{hexboard}
    $}}{(G_{25},3,6)} 
  \sep
    \stack{\m{$
    \begin{hexboard}[scale=\scale]
      \rotation{30}
      \foreach\i in {1,...,6} {
        \foreach\j in {1,...,4} {
          \hex(\j,\i)
        }
      }
      \edge[\nw](1,1)(4,1)
      \edge[\sw\noobtusecorner](1,1)(1,6)
      \edge[\ne\noacutecorner](4,6)(4,1)
      \black(2,3)
      \black(2,6)
      \black(3,4)
    \end{hexboard}
    $}}{(G_{7},3,6)} 
  \sep
    \stack{\m{$
    \begin{hexboard}[scale=\scale]
      \rotation{30}
      \foreach\i in {1,...,6} {
        \foreach\j in {1,...,4} {
          \hex(\j,\i)
        }
      }
      \edge[\nw](1,1)(4,1)
      \edge[\sw\noobtusecorner](1,1)(1,6)
      \edge[\ne\noacutecorner](4,6)(4,1)
      \black(2,3)
      \black(3,4)
      \black(4,6)
    \end{hexboard}
    $}}{(G_{26},3,6)} 
    \sep
    \stack{\m{$
    \begin{hexboard}[scale=\scale]
      \rotation{30}
      \foreach\i in {1,...,6} {
        \foreach\j in {1,...,4} {
          \hex(\j,\i)
        }
      }
      \edge[\nw](1,1)(4,1)
      \edge[\sw\noobtusecorner](1,1)(1,6)
      \edge[\ne\noacutecorner](4,6)(4,1)
      \black(2,5)
      \black(3,2)
      \black(3,5)
      \black(4,6)
    \end{hexboard}
    $}}{(G_{9},4,6)} 
  \sep
    \stack{\m{$
    \begin{hexboard}[scale=\scale]
      \rotation{30}
      \foreach\i in {1,...,6} {
        \foreach\j in {1,...,4} {
          \hex(\j,\i)
        }
      }
      \edge[\nw](1,1)(4,1)
      \edge[\sw\noobtusecorner](1,1)(1,6)
      \edge[\ne\noacutecorner](4,6)(4,1)
      \black(2,3)
      \black(3,4)
      \black(4,5)
      \black(4,6)
    \end{hexboard}
    $}}{(G_{27},4,6)} 
  \sep
    \stack{\m{$
    \begin{hexboard}[scale=\scale]
      \rotation{30}
      \foreach\i in {1,...,6} {
        \foreach\j in {1,...,4} {
          \hex(\j,\i)
        }
      }
      \edge[\nw](1,1)(4,1)
      \edge[\sw\noobtusecorner](1,1)(1,6)
      \edge[\ne\noacutecorner](4,6)(4,1)
      \black(2,5)
      \black(3,2)
      \black(3,6)
      \black(4,6)
    \end{hexboard}
    $}}{(G_{28},4,6)} 
  \sep
    \stack{\m{$
    \begin{hexboard}[scale=\scale]
      \rotation{30}
      \foreach\i in {1,...,6} {
        \foreach\j in {1,...,4} {
          \hex(\j,\i)
        }
      }
      \edge[\nw](1,1)(4,1)
      \edge[\sw\noobtusecorner](1,1)(1,6)
      \edge[\ne\noacutecorner](4,6)(4,1)
      \black(2,3)
      \black(2,6)
      \black(3,4)
      \black(3,6)
    \end{hexboard}
    $}}{(G_{12},4,6)} 
    \sep
    \stack{\m{$
    \begin{hexboard}[scale=\scale]
      \rotation{30}
      \foreach\i in {1,...,6} {
        \foreach\j in {1,...,4} {
          \hex(\j,\i)
        }
      }
      \edge[\nw](1,1)(4,1)
      \edge[\sw\noobtusecorner](1,1)(1,6)
      \edge[\ne\noacutecorner](4,6)(4,1)
      \black(2,3)
      \black(2,6)
      \black(3,4)
      \black(3,6)
      \black(4,6)
    \end{hexboard}
    $}}{(\top,5,6)} 
  \end{array}
  \xpace
  \begin{array}[t]{l}
    \stack{\m{$
    \begin{hexboard}[scale=\scale]
      \rotation{30}
      \foreach\i in {1,...,7} {
        \foreach\j in {1,...,4} {
          \hex(\j,\i)
        }
      }
      \edge[\nw](1,1)(4,1)
      \edge[\sw\noobtusecorner](1,1)(1,7)
      \edge[\ne\noacutecorner](4,7)(4,1)
      \black(2,3)
      \black(3,4)
      \black(2,6)
    \end{hexboard}
    $}}{(G_{1},3,7)} 
    \sep
    \stack{\m{$
    \begin{hexboard}[scale=\scale]
      \rotation{30}
      \foreach\i in {1,...,7} {
        \foreach\j in {1,...,4} {
          \hex(\j,\i)
        }
      }
      \edge[\nw](1,1)(4,1)
      \edge[\sw\noobtusecorner](1,1)(1,7)
      \edge[\ne\noacutecorner](4,7)(4,1)
      \black(3,2)
      \black(2,5)
      \black(3,5)
    \end{hexboard}
    $}}{(G_{2},3,7)} 
    \sep
    \stack{\m{$
    \begin{hexboard}[scale=\scale]
      \rotation{30}
      \foreach\i in {1,...,7} {
        \foreach\j in {1,...,4} {
          \hex(\j,\i)
        }
      }
      \edge[\nw](1,1)(4,1)
      \edge[\sw\noobtusecorner](1,1)(1,7)
      \edge[\ne\noacutecorner](4,7)(4,1)
      \black(2,3)
      \black(3,4)
      \black(2,6)
      \black(3,6)
    \end{hexboard}
    $}}{(G_{3},4,7)} 
    \sep
    \stack{\m{$
    \begin{hexboard}[scale=\scale]
      \rotation{30}
      \foreach\i in {1,...,7} {
        \foreach\j in {1,...,4} {
          \hex(\j,\i)
        }
      }
      \edge[\nw](1,1)(4,1)
      \edge[\sw\noobtusecorner](1,1)(1,7)
      \edge[\ne\noacutecorner](4,7)(4,1)
      \black(2,3)
      \black(3,4)
      \black(2,6)
      \black(3,7)
    \end{hexboard}
    $}}{(G_{4},4,7)} 
    \sep
    \stack{\m{$
    \begin{hexboard}[scale=\scale]
      \rotation{30}
      \foreach\i in {1,...,7} {
        \foreach\j in {1,...,4} {
          \hex(\j,\i)
        }
      }
      \edge[\nw](1,1)(4,1)
      \edge[\sw\noobtusecorner](1,1)(1,7)
      \edge[\ne\noacutecorner](4,7)(4,1)
      \black(2,3)
      \black(3,4)
      \black(2,6)
      \black(4,7)
    \end{hexboard}
    $}}{(G_{5},4,7)} 
    \sep
    \stack{\m{$
    \begin{hexboard}[scale=\scale]
      \rotation{30}
      \foreach\i in {1,...,7} {
        \foreach\j in {1,...,4} {
          \hex(\j,\i)
        }
      }
      \edge[\nw](1,1)(4,1)
      \edge[\sw\noobtusecorner](1,1)(1,7)
      \edge[\ne\noacutecorner](4,7)(4,1)
      \black(3,2)
      \black(2,5)
      \black(3,5)
      \black(1,7)
    \end{hexboard}
    $}}{(G_{6},4,7)} 
    \sep
    \stack{\m{$
    \begin{hexboard}[scale=\scale]
      \rotation{30}
      \foreach\i in {1,...,7} {
        \foreach\j in {1,...,4} {
          \hex(\j,\i)
        }
      }
      \edge[\nw](1,1)(4,1)
      \edge[\sw\noobtusecorner](1,1)(1,7)
      \edge[\ne\noacutecorner](4,7)(4,1)
      \black(3,2)
      \black(2,5)
      \black(3,5)
      \black(2,7)
    \end{hexboard}
    $}}{(G_{7},4,7)} 
    \sep
    \stack{\m{$
    \begin{hexboard}[scale=\scale]
      \rotation{30}
      \foreach\i in {1,...,7} {
        \foreach\j in {1,...,4} {
          \hex(\j,\i)
        }
      }
      \edge[\nw](1,1)(4,1)
      \edge[\sw\noobtusecorner](1,1)(1,7)
      \edge[\ne\noacutecorner](4,7)(4,1)
      \black(3,2)
      \black(2,5)
      \black(3,5)
      \black(4,7)
    \end{hexboard}
    $}}{(G_{8},4,7)} 
    \sep
    \stack{\m{$
    \begin{hexboard}[scale=\scale]
      \rotation{30}
      \foreach\i in {1,...,7} {
        \foreach\j in {1,...,4} {
          \hex(\j,\i)
        }
      }
      \edge[\nw](1,1)(4,1)
      \edge[\sw\noobtusecorner](1,1)(1,7)
      \edge[\ne\noacutecorner](4,7)(4,1)
      \black(2,3)
      \black(3,4)
      \black(2,6)
      \black(3,6)
      \black(4,7)
    \end{hexboard}
    $}}{(G_{9},5,7)} 
    \sep
    \stack{\m{$
    \begin{hexboard}[scale=\scale]
      \rotation{30}
      \foreach\i in {1,...,7} {
        \foreach\j in {1,...,4} {
          \hex(\j,\i)
        }
      }
      \edge[\nw](1,1)(4,1)
      \edge[\sw\noobtusecorner](1,1)(1,7)
      \edge[\ne\noacutecorner](4,7)(4,1)
      \black(2,3)
      \black(3,4)
      \black(2,6)
      \black(3,7)
      \black(4,7)
    \end{hexboard}
    $}}{(G_{10},5,7)} 
    \sep
    \stack{\m{$
    \begin{hexboard}[scale=\scale]
      \rotation{30}
      \foreach\i in {1,...,7} {
        \foreach\j in {1,...,4} {
          \hex(\j,\i)
        }
      }
      \edge[\nw](1,1)(4,1)
      \edge[\sw\noobtusecorner](1,1)(1,7)
      \edge[\ne\noacutecorner](4,7)(4,1)
      \black(3,2)
      \black(2,5)
      \black(3,5)
      \black(1,7)
      \black(4,7)
    \end{hexboard}
    $}}{(G_{11},5,7)} 
    \sep
    \stack{\m{$
    \begin{hexboard}[scale=\scale]
      \rotation{30}
      \foreach\i in {1,...,7} {
        \foreach\j in {1,...,4} {
          \hex(\j,\i)
        }
      }
      \edge[\nw](1,1)(4,1)
      \edge[\sw\noobtusecorner](1,1)(1,7)
      \edge[\ne\noacutecorner](4,7)(4,1)
      \black(3,2)
      \black(2,5)
      \black(1,7)
      \black(2,7)
      \black(3,7)
    \end{hexboard}
    $}}{(G_{12},5,7)} 
    \sep
    \stack{\m{$
    \begin{hexboard}[scale=\scale]
      \rotation{30}
      \foreach\i in {1,...,7} {
        \foreach\j in {1,...,4} {
          \hex(\j,\i)
        }
      }
      \edge[\nw](1,1)(4,1)
      \edge[\sw\noobtusecorner](1,1)(1,7)
      \edge[\ne\noacutecorner](4,7)(4,1)
      \black(3,2)
      \black(2,5)
      \black(1,7)
      \black(2,7)
      \black(3,7)
      \black(4,7)
    \end{hexboard}
    $}}{(\top,6,7)} 
  \end{array}
  \]
  \caption{All best patterns of height $k=4$ and widths $n=4, 5, 6,
    7$.}
  \label{fig-best}
\end{figure}
Each pattern $P$ is shown with its triple $\triple{P}$. The game
values $G_{1},\ldots,G_{28}$ have not been written out as they can be
quite long; notice, however, that several values appear more than
once. As is evident from Figure~\ref{fig-best}, for each best pattern
with triple $(G,s,4)$, there is a best pattern with triple
$(G,s+2,7)$, and vice versa.  As remarked above, the best triples of
width $n+1$ are completely determined by the best triples of width
$n$. It follows that the entire sequence of best triples repeats in
the following sense: for all $n\geq 4$, $(G,s,n)$ is a best triple if
and only if $(G,s+2,n+3)$ is a best triple. In other words, whenever
the width increases by $3$, exactly $2$ additional black stones are
required to achieve the same combinatorial value. Moreover,
Figure~\ref{fig-best} shows that the minimum number of stones in an
acceptable pattern of width $4$, $5$, and $6$ is $1$, $2$, and $2$,
respectively. These facts imply the following theorem.

\begin{theorem}\label{thm-minimum-cost}
  For all $n\geq 4$, the minimal number of stones required for a
  virtual connection between Black's edges on a board of size $4\times
  n$ is exactly $\ceil{\frac{2}{3}n-2}$. In table form:
  \[
  \begin{array}{l|ccccccccccccc}
    \mbox{\rm Width of board} &
    4 & 5 & 6 & 7 & 8 & 9 & 10 & 11 & 12 & 13 & 14 & 15 & 16 \\\hline
    \mbox{\rm Minimum stones} &
    1 & 2 & 2 & 3 & 4 & 4 & 5 & 6 & 6 & 7 & 8 & 8 & 9 
  \end{array}
  \]
\end{theorem}

This confirms --- at least in the case of boards of height 4 --- what
Hex players have long suspected: namely that the asymptotic cost of a
``long distance'' virtual connection is 2 stones per 3 columns of
distance travelled. The optimal connections claimed in
Theorem~\ref{thm-minimum-cost} can be achieved by patterns similar to
the ones shown in Figure~\ref{fig-4xn}.

The size of minimal connecting sets is also known for boards of size
$1\times n$, $2\times n$, $3\times n$, and $5\times n$, and is $n$,
$\ceil{\frac{2}{3}(n-1)}$, $\ceil{\frac{2}{3}n - 1}$, and
$\ceil{\frac{2}{3}n-3}$, respectively, with the exception of the
$5\times 6$ board, where 2 black stones are required. This can be
proved by the same method, although for the cases $k=1,2,3$, there
exist simpler proofs that do not use combinatorial game theory.  Some
typical optimal connections for $5\times n$ are shown in
Figure~\ref{fig-5xn}; note that in addition to the two long-distance
patterns we already saw in Figure~\ref{fig-4xn}, an additional new
pattern emerges in Figure~\ref{fig-5xn}(c). It uses more space,
but still requires 2 stones per 3 columns of distance.
\begin{figure}
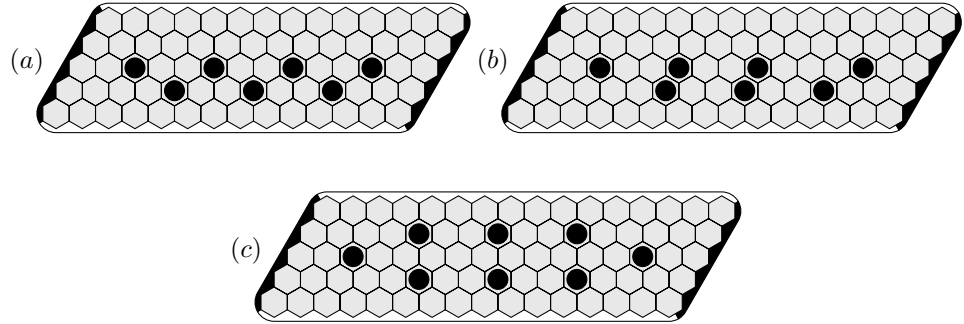

  \def\scale{0.5}
  \def\xspace{\hspace{-0.75ex}}
  \def\squad{\hspace{0.5ex}}
  \[
  (a)\xspace\m{$
    \begin{hexboard}[scale=\scale]
      \rotation{30}
      \board(5,14)
      \black(3,3)
      \black(2,5)
      \black(3,6)
      \black(2,8)
      \black(3,9)
      \black(2,11)
      \black(3,12)
    \end{hexboard}
    $}
  \squad
  (b)\xspace\m{$
    \begin{hexboard}[scale=\scale]
      \rotation{30}
      \board(5,15)
      \black(3,3)
      \black(2,6)
      \black(3,6)
      \black(2,9)
      \black(3,9)
      \black(2,12)
      \black(3,13)
    \end{hexboard}
    $}
  \]
  \vspace{1em}
  \[
  (c)\xspace\m{$
    \begin{hexboard}[scale=\scale]
      \rotation{30}
      \board(5,16)
      \black(3,3)
      \black(4,5)
      \black(2,6)
      \black(4,8)
      \black(2,9)
      \black(4,11)
      \black(2,12)
      \black(3,14)
    \end{hexboard}
    $}
  \]
  \caption{Optimal virtual connections on boards of size
    $5\times 14$, $5\times 15$, and $5\times 16$.}
  \label{fig-5xn}
\end{figure}

\subsection{An Inductive Proof of Theorem~\ref{thm-minimum-cost}}

We have derived Theorem~\ref{thm-minimum-cost} from a computation,
rather than giving a proof in the traditional sense. Of course, it is
possible to convert this computation into a proof by induction.  In
this section, we show how to do this. Given a subset $S$ of any
preordered set, let us write $\downdeal S$ for the down-closure of
$S$, i.e., $\downdeal S = \s{x \mid \text{there exists $y\in S$ such
    that $x\leq y$}}$.

\begin{definition}
  We define the {\em cost} of a pattern $P$ as
  \[
  \cost(P) = 6 + 3s - 2n,
  \]
  where $s$ is the number of black stones and $n$ is the number of
  columns in $P$. We define the {\em benefit} of $P$ as
  \[
  \ben(P) =
  \begin{choices}
    -\infty & \mbox{if $P$ is unacceptable} \\
    0 & \mbox{otherwise, if $\val(P)\in\downdeal\s{G_{21},G_{22}}$,} \\
    1 & \mbox{otherwise, if $\val(P)\in\downdeal\s{G_{1},G_{2}}$,} \\
    2 & \mbox{otherwise, if $\val(P)\in\downdeal\s{G_{13},G_{14},G_{15},G_{16},G_{17},G_{18}}$,} \\
    3 & \mbox{otherwise, if $\val(P)\in\downdeal\s{G_{3},G_{7},G_{23},G_{24},G_{25},G_{26}}$,} \\
    4 & \mbox{otherwise, if $\val(P)\in\downdeal\s{G_{4},G_{5},G_{6},G_{8}}$,} \\
    5 & \mbox{otherwise, if $\val(P)\in\downdeal\s{G_{12},G_{19},G_{20}}$,} \\
    6 & \mbox{otherwise, if $\val(P)\in\downdeal\s{G_{9},G_{27},G_{28}}$,} \\
    7 & \mbox{otherwise. Note that in this case, $\val(P)\in\downdeal\s{\top}$.}
  \end{choices}
  \]
  Here, $G_{1},\ldots,G_{28}$ are the same values that are shown in
  Figure~\ref{fig-best}. Note that $\ben$ is a monotone function.
  We also define the cost of a column to be $\cost(C) = 3s - 2$, where
  $s$ is the number of black stones in it. Note that
  $\cost(P+C)=\cost(P)+\cost(C)$.
\end{definition}

The point of this definition is that we will prove that a position's
benefit is never greater than its cost. The key to this result is the
following lemma.

\begin{lemma}\label{lem-best-induction-step}
  Whenever $P$ is a pattern and $C$ is a column containing only blank
  cells and black stones, $\ben(P+C)\leq\ben(P)+\cost(C)$.
\end{lemma}

\begin{proof}
  We first consider the special case where
  $\val(P)\in\s{G_{1},\ldots,G_{28},\top}$.  Since there are only 29
  patterns and 16 columns to consider, this special case of the lemma
  can be proved by exhaustively checking all 464 cases.

  Now consider some arbitrary pattern $P$ and column $C$. If
  $\ben(P)=-\infty$, then so is $\ben(P+C)$ and there is nothing to
  show. Otherwise, $\ben(P)$ is some number. To illustrate the proof,
  we consider the case $\ben(P)=4$; all other cases are similar.  By
  the definition of $\ben$, we know that $\val(P)\leq G_4$,
  $\val(P)\leq G_5$, $\val(P)\leq G_6$, or $\val(P)\leq G_8$. We
  consider the case $\val(P)\leq G_4$; again, the other cases are
  similar. Now we have
  \[
  \begin{array}{rcll}
    \ben(P+C)
    &\leq& \ben(G_4+C) & \mbox{by monotonicity of $\ben$,} \\
    &\leq& \ben(G_4)+\cost(C) & \mbox{by the above special case,} \\
    &\leq& 4 + \cost(C)  & \mbox{since $\ben(G_4)\leq 4$ by definition of $\ben$,} \\
    &=& \ben(P)+\cost(C) & \mbox{by assumption, since $\ben(P)=4$.}
  \end{array}
  \]
  This proves the lemma.  
\end{proof}

\begin{lemma}\label{lem-best-induction}
  For all patterns $P$ of width $n\geq 4$, $\ben(P)\leq\cost(P)$.
\end{lemma}

\begin{proof}
  We prove this by induction on $n$. For the base case $n=4$, there
  are finitely many patterns to check. Moreover, it suffices to check
  the 13 patterns shown in column 1 of Figure~\ref{fig-best}, because
  they are best, i.e., they maximize the benefit for a given cost, or
  equivalently, minimize the cost for a given benefit.

  For the induction step, suppose the claim is true for patterns of
  width $n$, and consider a pattern $P+C$ of width $n+1$. Then
  $\ben(P+C)\leq\ben(P)+\cost(C)\leq\cost(P)+\cost(C) = \cost(P+C)$,
  where the first inequality holds by
  Lemma~\ref{lem-best-induction-step} and the second inequality holds
  by the induction hypothesis.
\end{proof}

\begin{proof}[Proof of Theorem~\ref{thm-minimum-cost}]
  We have already found virtual connections with the claimed number of
  stones; what must be proved is the minimality. Let $P$ be a virtual
  connection of width $n$, using $s$ stones. Since $P$ is an
  acceptable pattern, we have $0\leq \ben(P)$, hence by
  Lemma~\ref{lem-best-induction}, $0\leq\cost(P)$, i.e., $0\leq 6 + 3s
  - 2n$, i.e., $s\geq \frac{2}{3}n-2$. Since $s$ is an integer, we
  have $s\geq\ceil{\frac{2}{3}n-2}$, as claimed.
\end{proof}

\section{Conclusion and Future Work}\label{sec-conclusion}

We described a theory of combinatorial games over partially ordered
atom sets that is appropriate for Hex and other monotone set coloring
games. The fundamental theorem about these games is that monotone and
passable games are the same thing up to equivalence; however, passable
games are the more robust class, as they are closed under taking
canonical forms. Passable games support many of the usual operations
of combinatorial game theory, including sums and opposite games. We
briefly discussed the notion of global decisiveness, under which
certain games that are otherwise non-equivalent can become equivalent.
We enumerated all passable games up to equivalence for certain small
atom posets, and proved that the class of passable games is infinite
in all other cases. We showed that many game values over small atom
posets are realizable as Hex positions, but also proved that there are
some values that are not realizable in Hex. Finally, we illustrated
the usefulness of this theory by proving a theorem about the minimal
size of connecting sets in $k\times n$ Hex.

Many problems are left open. It is not currently known whether every
passable game value can be realized by a planar connection game or
a vertex Shannon game. It is also not known whether every passable
abstract game value for a 3-terminal region is realizable as a Hex
position, though we proved this to be false for 4-terminal regions.
Indeed, it is not even known whether the number of combinatorial
values that are realizable as 3-terminal Hex positions is finite or
infinite. The theory of global decisiveness has not been worked out in
any detail; in particular, no simple recursive characterization of the
globally decisive order has been given, and no canonical form is known
for games up to global decisiveness. Finally, we found the size of
minimal connecting sets in $k\times n$ Hex for $k\leq 5$. While the
method generalizes to larger $k$, the computations get exponentially
harder as $k$ increases. The answer to this question is not yet known
for $k\geq 6$.

There are many potential applications of this theory that can be
explored. The theory of passable games can be applied to the analysis
of Hex positions, and might potentially lead to useful optimizations
in computer Hex. Specifically, for positions that partition the board
into multiple disjoint regions, current computer Hex implementations
typically have to search many sequences of moves that alternate
between the regions. It may be possible to speed up the analysis by
analyzing each region separately. One situation where this is likely
to be useful is when a region is already {\em settled}, meaning that
its combinatorial value is atomic. In that case, no further moves in
that region need to be explored at all.  Combinatorial game theory may
also lend itself to a more systematic analysis of ladders and ladder
escapes in Hex. It can also be used to reason about, and potentially
simplify the theory of, capture and domination, which helps players
find locally optimal moves and avoid inferior moves. Another possible
application of this theory is to help resolve the inferiority or
non-inferiority of template intrusions, such as Conjecture~1 of
{\cite{Henderson-Hayward}}. Some of these applications are already
being worked on and will appear in forthcoming papers.

\vspace{20pt}
\noindent {\bf Acknowledgements.}
I owe an enormous debt of gratitude to Eric Demer, a master Hex player
who taught me most of what I know about the game, and who has spent
countless late-night hours discussing the combinatorial game theory of
Hex with me. Eric has been an invaluable source of insight, examples,
and counterexamples, and he has generously permitted me to include
many of his contributions in this paper. Among other things, he
suggested the region of Figure~\ref{fig-3-terminal-decisive-gap} as an
example of a Hex region with 4 linearly ordered outcomes, helped find
several of the Hex positions in
Figure~\ref{fig-realizable-3-terminal-decisive-gap}, and introduced me
to the concept of global decisiveness and explained its importance. He
also provided valuable feedback on an earlier draft of this paper. I
would also like to thank the anonymous referee for valuable comments.
Any errors are of course mine. This research was supported by NSERC.


\end{document}